\documentclass[opre,nonblindrev]{informs_no_template}

\usepackage{endnotes}
\let\footnote=\endnote

\usepackage{amsmath}
\usepackage{amsfonts}
\usepackage{amssymb}
\usepackage{epsfig}
\usepackage{diagbox}
\usepackage{amsmath}
\usepackage{amsfonts}
\usepackage{amssymb}
\usepackage{dsfont}
\usepackage{mathrsfs}
\usepackage{bbm,bm}
\usepackage{algorithm}
\usepackage{multirow,multicol}

\newcommand{\setd}{{ d \kern -.15em l}}
\newcommand{\hatsetd}{ d \hat{\kern -.15em l }}
\newcommand{\dd}{\mathsf {d\kern -0.07em l}} 

\newcommand{\x}{\bm{x}}

\newcommand{\bgeqn}{\begin{eqnarray}}
\newcommand{\edeqn}{\end{eqnarray}}
\newcommand{\bgeq}{\begin{eqnarray*}}
\newcommand{\edeq}{\end{eqnarray*}}
\newcommand{\bec}{\begin{center}}
\newcommand{\enc}{\end{center}}
\newcommand{\R}{\mathbb{R}}

\newcommand{\inmat}[1]{\mbox{\rm {#1}}}

\newcommand{\Z}{{\cal Z}}
\newcommand{\F}{{\cal F}}

\newcommand{\B}{{\cal B}}

\def\dom{{\rm dom}}

\def\bbe{{\Bbb{E}}} 

\makeatletter
\newenvironment{breakablealgorithm}
  {
   \begin{center}
     \refstepcounter{algorithm}
     \hrule height.8pt depth0pt \kern2pt
     \renewcommand{\caption}[2][\relax]{
       {\raggedright\textbf{\ALG@name~\thealgorithm} ##2\par}%
       \ifx\relax##1\relax 
         \addcontentsline{loa}{algorithm}{\protect\numberline{\thealgorithm}##2}%
       \else 
         \addcontentsline{loa}{algorithm}{\protect\numberline{\thealgorithm}##1}%
       \fi
       \kern2pt\hrule\kern2pt
     }
  }{
     \kern2pt\hrule\relax
   \end{center}
  }
\makeatother

 \newcommand{\LJ}[1]{{#1}}

\allowdisplaybreaks[4]

\usepackage{tikz}
\usetikzlibrary{positioning,shapes}
\usetikzlibrary{matrix}
\usetikzlibrary{graphs}
\usetikzlibrary{chains}

\TheoremsNumberedThrough     
\ECRepeatTheorems

\EquationsNumberedThrough    

\MANUSCRIPTNO{OPRE-2021-09-599} 

\begin{document}


\RUNAUTHOR{Liu, Chen and Xu}

\RUNTITLE{Multistage Utility Preference Robust Optimization}

\TITLE{Multistage Utility Preference Robust Optimization}

\ARTICLEAUTHORS{%
\AUTHOR{Jia Liu and Zhiping Chen}
\AFF{School of Mathematics and Statistics, Xi'an Jiaotong University, Xi'an, Shaanxi, P. R. China,\\
 Center for Optimization Technique and Quantitative Finance, Xi'an International Academy
for Mathematics and Mathematical Technology, Xi'an, P. R. China, \EMAIL{jialiu@xjtu.edu.cn, zchen@mail.xjtu.edu.cn} } 
\AUTHOR{Huifu Xu}
\AFF{Department of Systems Engineering and Engineering Management,
The Chinese University of Hong Kong, Hong Kong, \EMAIL{hfxu@se.cuhk.edu.hk}}
} 

\ABSTRACT{%
In this paper, we consider a multistage expected utility maximization problem where the decision maker's  utility function at each stage depends on historical data
and the information on the true utility function is incomplete.
To mitigate { adverse impact} arising from ambiguity of the true utility, we propose a maximin robust model where the optimal policy is based on the worst-case sequence of utility functions from an ambiguity set constructed with partially available information about the decision maker's preferences.
We then show that the multistage maximin problem is time consistent when the utility functions are state-dependent and demonstrate with a counter example that the time consistency may not be retained when the utility functions are state-independent. With the time consistency, we show the maximin problem can be solved by a recursive formula whereby a one-stage maximin problem is solved at each stage beginning from the last stage.
Moreover, we propose two approaches to construct the ambiguity set: a pairwise comparison approach and a $\zeta$-ball approach where a ball of utility functions centered at a nominal utility function under $\zeta$-metric is considered. To overcome the difficulty arising from solving the infinite dimensional optimization problem in computation of the worst-case expected utility value, we propose piecewise linear approximation of the utility functions and derive error bound for the approximation under moderate conditions.
{Finally, we use the stochastic dual dynamic programming (SDDP) method and the nested Benders' decomposition method to solve the multistage
state-dependent preference robust problem and the scenario tree method
to solve the state-independent problem, and carry out comparative analysis
on the efficiency of the computational schemes as well as out-of-sample performances of the state-dependent and state-independent models. The preliminary results
show that the state-dependent preference robust model solved by SDDP algorithm
displays overall superiority.
}


}%

\KEYWORDS{Preference robust optimization, state-dependent utility,
rectanglarity,
time consistency, Kantorovich ball, { scenario tree method,
SDDP,
nested Benders' decomposition method}
} \HISTORY{This paper was
first submitted on Sep. 2021, revised on Feb. 2023.}

\maketitle

\section{Introduction}

Decision making under uncertainty has two important elements: belief and taste.
Belief is the decision maker's (DM for brevity)
view on the state of nature of the underlying uncertainty whereas taste
is the DM's preference. When there is an ambiguity about belief, one may base the optimal decision on the worst-case scenario
of the uncertainty or worst-case probability distribution 
\cite{GiS89} 
 and this is the case of robust optimization or
distributionally
robust optimization.
Over the past few decades, a lot of research have been conducted on robust optimization and
distributionally
robust optimization
models, see monograph by Ben-Tal et al.~\cite{BGN09} and
a comprehensive overview by Rahimian and Mehrotra \cite{RaM19}.

Ambiguity may also occur with regard to taste.
{ An explicit utility function might not be available
when information on the DM's preference is incomplete.
There is a multitude of
ways about how to use
preference information to construct a utility function.
In the literature of decision analysis and behavioural economics,
a popular method is
to elicit the DM's preferences
with paired gambling approaches for preference comparisons 
\cite{Far84},
use the elicited information to
identify the value
 of the utility function at a discrete set of points and}
construct an approximate utility function
via some interpolation methods,
see for instance
\cite{ClR01}.

Armbruster and Delage \cite{AmD15}
{ argue that the interpolation
approach has some drawbacks because not only
it is often difficult to identify a non-parametric
utility function purely based on the DM's preferences
over pairwise comparison lotteries
but also it could be risky to use a single approximate utility function without considering other plausible ones nearby. Consequently, they propose an alternative
approach, that is,}
instead of trying to find
{ a single} approximate von Neumann and Morgenstern's utility function,
they
{ propose to} use
available information of the DM's preferences
such as preferring certain lotteries over other lotteries,
being risk averse over gains and risk taking over losses
to
construct an ambiguity
 set of plausible utility functions
 and then base the optimal decision
 on the worst-case utility function from the ambiguity set.
The approach is called preference robust optimization (PRO)
as it  follows the general philosophy of robust optimization.
{
In the case that the ambiguity set is constructed through
pairwise comparisons, the PRO model
may be viewed as an extension
of the well-known
stochastic programs with stochastic dominance constraints
(Dentcheva and Ruszczy\'nski
\cite{DeR03}).
}
Hu and Mehrotra \cite{HuM15}  also take a PRO approach
to tackle the ambiguity of the true utility function but
in a slightly different manner.
They consider a probabilistic representation of the class of increasing concave utility functions by
  confining them to a compact interval and scaling them to being bounded by $1$.
In doing so, they propose a
 moment-type
 framework for constructing
  the ambiguity set of the DM's utility functions
  which covers a number of important approaches such as
  the certainty equivalent and pairwise comparison.

Over the past few years,
PRO has attracted increasing attentions.
For instances,
Haskell et al.~\cite{HFD16}
propose a robust model which handles the ambiguity of DM's
 belief and taste.
Hu and Stepanyan \cite{HuS17} propose a so-called reference-based almost
stochastic dominance method for constructing
a set of utility functions
near a reference utility
which satisfy certain stochastic dominance relationship
and use the set to characterize the DM's preference.
Hu et al.~\cite{HBM18} consider a PRO model with an ambiguity set of
general utility functions and propose a  Lagrangian function approach
for solving the resulting maximin problem.
Guo and Xu \cite{GuX21} propose
a piecewise linear approximation approach for solving
a PRO model with the ambiguity set
being specified by moment-type conditions and derive
a bound for approximation error in terms of the ambiguity set, the optimal value and optimal solutions.
The PRO approach has also been effectively applied to risk management problems where the DM's risk preferences
are ambiguous, see \cite{DeL17,GuX21a,Li21,WaX20,WHHX21,ZXW20}.


In this paper, we extend this stream of research to multistage decision making process.
There are several
modelling approaches in multistage decision making such as
multistage stochastic optimization,
Markovian decision making,
and approximate dynamic programming \cite{Pow19}.
Among them, multistage stochastic optimization (MSO) has been widely studied and applied in long term financial planning, pension fund management, energy production and trading, supply chain management and inventory control \cite{PfP14}, as it can flexibly characterize the dynamic dependent structure of random data process.
A key component in the multistage decision making modelling is the dynamic decision criterion, i.e., the objective function for the multistage stochastic optimization model.
One of the most widely adopted objective functions is the multistage expected utility models, which can be also
understood
as a kind of multistage risk function where the utility function at each stage characterizes the dynamic preference of the DM \cite{Duf10}.
There are basically three types of multistage expected utility models:
terminal utility model,
additive utility model
and recursive utility model
\cite{CCL17}.
Like terminal risk measures,
the terminal utility model may lead to time inconsistent optimal policies
as it only measures the utility of reward at the terminal stage \cite{CLW12}.

The additive utility model, which is most extensively studied,
considers the sum of utilities of rewards at different stages,
thus the DM's intertemporal preferences
are risk neutral
\cite{Duf10,SDR14}.
The recursive utility model,
also known as stochastic differentiable utility
in continuous time setting, characterizes DM's
nonlinear intertemporal preferences. The model has a natural connection with time consistency
of the optimal 
policy.
Important contributions include the recursive expected utility \cite{Koo1960}
and
the well-known Kreps-Porteus utility which is recursive, but not necessarily expected utility \cite{KrP1978}.
However, traditional expected utility theory has received many criticisms
for its failure to explain
some experimental observations and
theoretical puzzles such as Allais paradox.
Rank-dependent expected utility theory and cumulative prospect theory are subsequently proposed to address the drawbacks, see monograph by Puppe \cite{Puppe12} for an overview of the development of the theories.
In dynamic setting,
Hu et al.~\cite{HJZ2020} study a
continuous-time portfolio selection model
where a sequence of time-dependent
probability weighting functions and
 rank-dependent utility functions are used to
capture a DM's overweighting and underweighting behaviours
on tail losses/rewards at different stages, see also \cite{CSZ20} for
empirical studies.

In all these works, the DM's utility functions are assumed to be known exactly
and fixed in the decision making process.
However, as we discussed earlier, the DM's utility function
may be ambiguous and this motivates us to propose a
PRO model for the multistage decision making process.
Moreover, many studies argue that 
utility functions may be state-dependent.
The most widely adopted approach is to consider a parametric form of utility functions where the parameters are state-dependent.
For instance, Strub and Li \cite{StL2020}
consider a sequence of S-shaped utility functions parameterized by
a sequence of state-dependent reference points
and show that
failing to update the reference point as state changes may
lead to time inconsistent investments.
Likweise, He et al.~\cite{HSZ20} consider a series of
 state-dependent  distortion functions
 when they apply the  rank-dependent expected
 utility theory to continuous time investment problems.
 Bj\"{o}rk et al.~\cite{BMZ14} adopt a state-dependent
  risk-aversion parameter in the multistage mean-risk model. 
There is also a specific stream of research on so-called
{\em habit formation utility} where the DM's consumption habit level and
her/his utility at a particular stage and/or state
is determined by the historical consumption process \cite{Duf10}.

In this paper,
we will also use the habit formation utility model and concentrate
on a situation where the DM's utility at each stage is ambiguous but it is possible to
use partially available information to construct a set (called ambiguity set later on) of plausible utility functions which capture the DM's preferences.
Two ways are proposed to construct the ambiguity set. One is to
use the pairwise comparison approach which are widely used in the literature of
PRO models and behavioural economics. The other is to construct
a ball of utility functions centered at
a nominal utility function under some pseudo-metrics. The main challenge to be
tackled is to develop efficient
computational schemes for solving the
resulting multistage PRO models.


As far as we are concerned, the main contributions of the paper can be summarized as follows.

First, we propose a multistage PRO model where the DM's utility preferences at different stages depend on not only the current stage and  state but also the history of the underlying random data process leading to the state.
We introduce a definition of ambiguity set comprising
certain sequences of state-wise utility functions and a maximin optimization model where the optimal policy
is based on the worst-case summed 
expected utility values of the random reward functions at different stages computed with the ambiguity set.

Second, we introduce the concept of rectangularity of
the ambiguity set of utility functions. Under some moderate conditions,
we show the multistage maximin problem with state-dependent ambiguity set is time consistent and  demonstrate through a simple example that the problem is time inconsistent when the utility functions are state-independent.

Third, by utilizing the time consistency,
we derive a recursive formula for solving the multistage PRO problem when the utility functions are
state-dependent.
For the ambiguity of general utility functions, error bounds for both the ambiguity set and the optimal value
are derived when the utility functions in the ambiguity set are approximated by piecewise linear utility functions at each stage.
To tackle time inconsistency and nonlinearity in solving maximin problem at each stage, we propose
a scenario tree approach which reformulates the holistic maximin problem
as a single mixed integer linear program, {
we 
propose to use the
stochastic dual dynamic programming (SDDP)
method and the nested Benders' decomposition (NBD) method to solve the state-dependent multistage PRO models.}

Fourth, we apply the proposed PRO model and the computational scheme to
 a multistage investment-consumption problem
 {
and carry out
comparative analysis
on the efficiency of the computational schemes as well as
out-of-sample performances of the state-dependent
and state-independent models. The preliminary results
show that the state-dependent preference robust model solved by SDDP algorithm
displays overall superiority.

 }


The rest of the paper is organized as follows.
Section \ref{sec-msp} defines the multistage expected utility maximization models to be discussed in this paper.
Section 3 introduces the robust counterparts and
discusses rectangularity and time consistency of the models when the utility functions are state-dependent.
Section 4 details construction of the ambiguity set with two approaches: pairwise comparisons and $\zeta$-ball.
In the latter approach, a piecewise linear approximation
approach is proposed to approximate the general utility functions and error bounds are derived.
{Section 5 discusses computational schemes for solving multistage PRO models by
the scenario tree method and
dynamic programming algorithms. 
Section 6 reports a number of 
numerical results and
comparative analysis.}
Finally, Section 7
gives some concluding remarks.
Due to the limitation of pages in the main body of the paper, all proofs
{ of the technical results,
some examples and
the detailed algorithmic procedures
of SDDP and NBD methods
are
moved to
 Electronic Companions}.

\section{Multistage expected utility models}\label{sec-msp}
We begin by introducing notions and notations that are commonly used in multistage stochastic optimization.
Let $\xi=\{\xi_t\}_{t=1}^T$ be a stochastic process defined on some probability space $(\Omega, \mathcal F,\mathbb{P})$, where $\xi_t:\Omega\to \mathbb R^{d_t}$ is a random vector supported on $\Xi_t$ for $t=1,\dots,T$.
For simplicity of notation,
we write $\xi_{[t]}$ for historical information $(\xi_1,\dots,\xi_t)$.
Let $\mathcal F_t$ denote the sigma algebra in the sample space $\Omega$  generated (induced) by $\xi_{[t]}$, that is,
$ \mathcal F_t =\left\{(\xi_{[t]})^{-1} B: B \in \mathscr{B}(\Xi_{[t]}) \right\}$,
where $\mathscr{B}(\Xi_{[t]})$ denotes the Borel sigma algebra of set $\Xi_{[t]}:=\Xi_1\times\cdots\times\Xi_t$.
By convention, we assume that there is an initial state $\xi_0$ which is deterministic
and
corresponds to the deterministic events $\mathcal F_0 =\{\emptyset,\,\Omega\}$.
Consequently, we have
$\mathcal F_0\subset \mathcal F_1\subset \dots\subset \mathcal F_t\subset \mathcal F_{t+1} \dots \subset \mathcal F_T \subset \mathcal F.$
As $\mathcal F_t$ is generated by $\xi_{[t]}$, we denote $\mathbb{E}_{|\mathcal{F}_t}[\cdot]:=\mathbb{E}[\cdot\mid \xi_t]$ for simplicity and $\mathbb{E}_{|\mathcal{F}_0}[\cdot]:=\mathbb{E}[\cdot]$.

Let
$\mathcal{L}^p(\Omega, \mathcal F,\mathbb{P};\mathbb{R})$
denote
the set of random variables $\psi:
(\Omega, \mathcal F,\mathbb{P})\to \mathbb{R}$
with finite $p$-th moments, i.e.,
$\int_{\Omega} |\psi(\omega)|^p d\mathbb{P}(\omega)<\infty$, {for $p\geq 1$}.
Let $\mathcal{L}^p({\mathbb{R}})$ denote
the set of real functions $u: \mathbb{R}\to \mathbb{R}$  integrable to the $p$-th order and
$\mathcal{L}^p
(\Omega, \mathcal F,\mathbb{P};\mathcal{L}^p({\mathbb{R}}))$
denote the set of random integrable functions $\mathfrak{\hat{u}}:
(\Omega, \mathcal F,\mathbb{P})\to \mathcal{L}^p({\mathbb{R}})$
with finite $p$-th moments.
{
We use
$\mathfrak{{u}_t}(x,\xi_{[t-1]})$ to denote a state-dependent utility function (strictly speaking, we should call it historical-data dependent. {We call it state-dependent for simplification.
This should be distinguished from Markov decision making process
whereas the states are decision dependent).}
Here,  for each fixed $x\in\R$,
the mapping $
\mathfrak{{u}_t}(x,\xi_{[t-1]}(\cdot)
):\Omega\to \R$ is
$\F_{t-1}$-measurable
and for each fixed $\omega\in \Omega$,
$\mathfrak{{u}_t}(\cdot,\xi_{[t-1]}(\omega)):\R\to \R$
is a continuous and non-decreasing function.
The dependence
on $\xi_{[t-1]}(\omega)$ reflects the fact that in general a decision maker's risk preference depends not only on the current state but also on the DM's
past experiences.
To be consistent with the existing PRO models in one-stage decision making problems (see e.g.~\cite{AmD15}),
we assume that the DM's utility preference is not affected by
future uncertainty.
In Section 4, we will explain how the requirement on the
measurability of the utility function may be fulfilled.
In the case that
the utility function is state-independent,
we write $u_t(\cdot)$ for
$\mathfrak{{u}_t}(\cdot,\xi_{[t-1]})$.
}

\subsection{
Models}

We consider the following multistage expected utility maximization problem
\bgeqn
\label{eq-msp-sdep-a}
&&\max_{x_1\in \mathscr{X}_1} \bbe \left[u_1(h_1(x_1,\xi_1))+
\max_{x_2\in \mathscr{X}_2(x_1,\xi_1)}
\LJ{\bbe_{|\F_1}}
\Big[
\mathfrak{u}_2(h_2(x_2,\xi_2),\xi_1)+\right.\nonumber\\
&&\left.
\quad\quad\quad\quad\quad\quad\cdots+
\max_{x_T\in \mathscr{X}_T(x_{[T-1]},\xi_{[T-1]})}
\LJ{\bbe_{|\F_{T-1}}}\big[
\mathfrak{u}
_T(h_T(x_T,\xi_T),\xi_{[T-1]})
\big]
\Big]
\right],
\edeqn
where $h_t: \mathbb{R}^{n_{t}} \times\mathbb{R}^{d_{t}} \rightarrow \mathbb{R}$
is a continuous reward function at stage $t$,
and $\mathfrak{u}_t:\mathbb{R} \times \mathbb{R}^{\sum_{i=1}^{t-1
}d_{i}}
\rightarrow \mathbb{R}$ is the utility function
 characterizing the DM's utility value of the reward at stage $t$, $x_t$ is
 the decision vector,
 $x_{[t]}$ is the historical decision process $(x_1,\dots,x_t)$ till stage $t$, and $\mathscr{X}_t(x_{[t-1]},\xi_{[t-1]})$
is the set of feasible decisions at stage $t$ for $t=2,\cdots,T$,
the expectation at stage 1 is taken with respect to 
the distribution of $\xi_1$,
and the expectation at stage $t$ is taken w.r.t. the distribution of $\xi_t$ conditional on
\LJ{the filtration $\F_{t-1}$, i.e.,} the historical data  $\xi_{[t-1]}$, for $t=2,\dots,T$.
In this setup, the DM chooses an optimal decision $x_t$ from $\mathscr{X}_t(x_{[t-1]},\xi_{[t-1]})$ so that
\bgeq
\LJ{\bbe_{|\F_{t-1}}} \left[\mathfrak{u}_t(h_t(x_t,\xi_t),\xi_{[t-1]})+
\cdots+
\max_{x_T\in \mathscr{X}_T(x_{[T-1]},\xi_{[T-1]})}
\LJ{\bbe_{|\F_{T-1}}} \Big[
\mathfrak{u}_T(h_T(x_T,\xi_T),\xi_{[T-1]}) 
\Big] 
\right]
\edeq
is maximized. The utility function at stage $t$ depends not only on the current stage (indicated by the subscript) but also on the historical state (realization of) $\xi_{[t-1]}$.
The choice of the optimal decision
$x_t$ is independent of the
realizations of $\xi_t$ which means
the decision is made before the realization of uncertainty $\xi_t$,
and it is not a recourse action.
Of course, we can interpret $\mathfrak{u}_t(h_t(x_t,\xi_t),\xi_{[t-1]})$ as the optimal value arising from a recourse action.
In particular, if the random reward function at stage $t$
depends on the current state $\xi_{t-1}$ rather than
state $\xi_t$ at next stage (mathematically replacing
$h_t(x_t,\xi_{t})$ with
$h_t(x_t,\xi_{t-1})$),
then problem \eqref{eq-msp-sdep-a}  can be written as
\bgeqn
\label{eq-msp-sdep-recourse}
&&\max_{x_1\in \mathscr{X}_1} u_1(h_1(x_1,\xi_0))+\bbe
\left[
\max_{x_2\in \mathscr{X}_2(x_1,\xi_1)}
\mathfrak{u}_2(h_2(x_2,\xi_1),\xi_1)+ \right.\nonumber\\
&&\left.
\quad\quad\quad\quad\cdots+
\LJ{\bbe_{|\F_{T-2}}}
\Big[\max_{x_T\in \mathscr{X}_T(x_{[T-1]},\xi_{[T-1]})}
\mathfrak{u}_T(h_T(x_T,\xi_{T-1}),\xi_{[T-1]})
\Big]
\right].
\edeqn
In this formulation, a recourse action $x_t$ is taken before realization of $\xi_{t-1}$ is observed.
We skip the details on recourse actions
so that we may focus on the key issues in this paper.
Note that if we interpret the utility function as the DM's {\em taste} and the distribution of the future uncertainty and of the reward as {\em belief} in the literature of decision analytics, then we can see that
the randomness in $\mathfrak{u}_t$ (the taste) arises from historical data $\xi_{[t-1]}$ whereas the belief is concerned with future uncertainty of rewards.
Unless specified otherwise, we assume that $\mathscr{X}_t(x_{[t-1]},\xi_{[t-1]})$
is a convex and compact subset of $\mathbb{R}^{n_t}$ for $t=1,\cdots,T$.

A simplified version of (\ref{eq-msp-sdep-a}) is
that the utility functions at each stage are state-independent, that is,
\bgeqn\label{eq-msp-indep}
&&\max_{x_1\in \mathscr{X}_1} \bbe\left[u_1(h_1(x_1,\xi_1))+
\max_{x_2\in \mathscr{X}_2(x_1,\xi_1)}
\LJ{\bbe_{|\F_1}}\Big[u_2(h_2(x_2,\xi_2))+\right.\nonumber\\
&&\left.
\quad\quad\quad\quad\quad\quad\quad\quad\cdots+
\max_{x_T\in \mathscr{X}_T(x_{[T-1]},\xi_{[T-1]})}
\LJ{\bbe_{|\F_{T-1}}}\big[u_T(h_T(x_T,\xi_T)) 
\big]
\Big]\right].
\edeqn
In this model, the DM has the same
utility preference
in all states at stage $t$ regardless of the
overall wealth accumulated over the past $t-1$ stages.
In the case when the DM takes an identical view
on utilities over all stages, the model may be further
simplified to
\bgeqn\label{eq-msp-indep-same}
&&\max_{x_1\in \mathscr{X}_1} \bbe\left[u(h_1(x_1,\xi_1))+
\max_{x_2\in \mathscr{X}_2(x_1,\xi_1)}
\LJ{\bbe_{|\F_{1}}}\Big[u(h_2(x_2,\xi_2))+\right.\nonumber\\
&&\left.
\quad\quad\quad\quad\quad\quad\cdots+
\max_{x_T\in \mathscr{X}_T(x_{[T-1]},\xi_{[T-1]})}
\LJ{\bbe_{|\F_{T-1}}}\big[
u(h_T(x_T,\xi_T))
\big]
\Big]
\right],
\edeqn

In the classical multistage stochastic programming (MSP) models, utility functions
at different stages may be different but they are pre-determined
at the beginning
which means the DM cannot adjust
her/his utility at later stages as dynamic stochastic environment changes
and this is inconsistent with practical decision making process.
Indeed, many theoretical and empirical studies
show that the utility function should depend on the current and/or historical state
\cite{CSZ20,HJZ2020,Mon20}.
For example, the DM's utility of wearing a mask in the year of 2020 (at the peak of COVID-19 epidemic) must be totally different from the utility in normal circumstances. Even at different stages of the COVID-19 epidemic, the utility of wearing a mask varies.
Here we assume that at stage $t$, the
DM can adjust her/his utility function according to the current and historical states.
Of course, regardless of the stage that the DM is in,
her/his utility function  must be specified prior to the decision making at the stage, i.e.,
$u_t$ does not depend on $\xi_t$.

In this paper, we will focus on the case when the utility functions
are ambiguous, and we will see that utility preference robust formulations of the above three models  will have completely different properties.

\subsection{Reformulations}

In model \eqref{eq-msp-sdep-a}, optimal decision at stage $t$
is a vector in $\mathbb{R}^{n_t}$. However, if we view the decision making from
stage 1, then we may regard it as a random function of $\xi_{[t-1]}$.
Consequently, we may reformulate
the multistage expected utility maximization problem \eqref{eq-msp-sdep-a} as
\begin{equation}\label{eq-msp-sdep}\begin{array}{cl}
\max\limits_{\x_{[
T]}}  & \mathbb{E}
\big[u_1(h_1\left( x_{1},\xi_1\right))
+\mathfrak{u}_2(h_{2}\left( {\x_{2}(\xi_1)}, \xi_{2}\right),{\xi_{1}})+\cdots+\mathfrak{u}_T(h_T \left( {\x_{T}(\xi_{[T-1]})}, \xi_{T}\right),\xi_{[T-1]})\big] \\
\text { s.t. }& x_{1} \in \mathscr{X}_{1}, {\x_{t}(\xi_{[t-1]})} \in \mathscr{X}_{t}\left( {\x_{[t-1]}(\xi_{[t-2]})}, \xi_{[t-1]}\right),\ \inmat{for}\  t=2, \ldots, T,
\end{array}\end{equation}
where the expectation is taken w.r.t. the distribution of $\xi_{[T]}$
and we write $\x_{[1,T]}$ {(or $\x_{[T]}$ when the decision process starts from the initial stage)} for a sequence of decisions $(x_1,\x_2(\cdot)\dots,\x_T(\cdot))$, which is also known as an implementable policy.
{ We denote $\x_{[t-1]}(\xi_{[t-2]}):=( x_{1},\x_{2}(\xi_{1}),\ldots,\x_{t-1}(\xi_{[t-2]}) )$ the $\xi_{[t-2]}$
-dependent historical decision process up to stage $t-1$}.
The reformulation is fundamentally related to Bellman's principle in dynamic programming that an optimal policy at the initial planning stage is consistent
with the optimal decisions at each
of the remaining stages, we will come back to this in Section 3.
The reformulation requires some moderate conditions, see Lemma \ref{lemma-interchange-expectation} on Page 13 for the two stage case.
Likewise, we can reformulate \eqref{eq-msp-indep} as
\begin{equation}\label{eq-msp-indep-comp}
\begin{array}{cl}
\max\limits_{x_{[
T]}}  & \mathbb{E}
\big[u_1(h_1\left( x_{1},\xi_1\right))
+\mathfrak{u}_2(h_{2}\left( {\x_{2}(\xi_1)}, \xi_{2}\right))+\cdots+\mathfrak{u}_T(h_T \left( {\x_{T}(\xi_{[T-1]})}, \xi_{T}\right))\big] \\
\text { s.t. }& x_{1} \in \mathscr{X}_{1}, {\x_{t}(\xi_{[t-1]})} \in \mathscr{X}_{t}\left( {\x_{[t-1]}(\xi_{[t-2]})}, \xi_{[t-1]}\right),\ 
t=2, \ldots, T.
\end{array}
\end{equation}

A practical application of the multistage utility maximization model
is multistage portfolio selection problem, see
an example in \ref{ec-example-portfolio}.

\section{Robust models}\label{sec-robust}
In the multistage expected utility optimization models that we
presented in the previous section,
the true utility functions which capture the DM's preferences at each stage
are assumed to be known.
This assumption may not be satisfied in practice as we discussed in the introduction section.
It motivates us to consider a robust model where the optimal decision at each stage is based on the worst-case utility function
from a set of plausible utility functions.
Since the robust model is essentially built upon von Neumann-Morgenstern expected utility theory,
we make a blanket assumption as follows.

\begin{assumption}
\label{Assu:VNM-DM}
The DM's preference  can be represented by von Neumann-Morgenstern expected utility theory
{and is consistent at each state.}
\end{assumption}

{ The assumption on the preference
consistency means that at each state, there exists at least one VNM's utility function which can be used to represent all of the elicited/observed preferences of the DM at the state.
}
In practice, { however},
DM's utility preferences may be inconsistent
due to cognitive biases \cite{TvK74} and/or elicitation errors \cite{AmD15}.
Bertsimas and O'Hair \cite{BeO13} and
Armbruster and Delage \cite{AmD15}  proposed
some approaches
to handle the issue.
Here by introducing Assumption \ref{Assu:VNM-DM},
we restrict our discussions to the consistent preference
case so that we may focus on the key challenges  arising from multistage maximin problems.

\subsection{Multistage PRO models}
In the expected utility maximization model
(\ref{eq-msp-sdep-a}) or its equivalent formulation (\ref{eq-msp-sdep}),
the sequence of dynamic decisions is made with respect to a sequence of utility functions $\{\mathfrak{u}_t\}$. However,
a DM may not have complete information
to identify  a sequence of true  utility preferences
but it is possible to use partial information
to build  an ambiguity set of plausible utility functions.
We begin with a formal definition of the ambiguity set which captures DM's utility preferences at each stage.

\begin{definition}[Ambiguity set of
utility functions]
\label{def-pro-set}
\LJ{Let $\mathbb{U}$ be the set of all continuous, bounded and monotonically increasing functions in $\mathcal{L}^p({\mathbb{R}})$ and
$\mathcal{U}_t$ be a $\mathcal{F}_{t-1}$-measurable set-valued mapping.
For any given $\xi_{[t-1]}$, $\mathcal{U}_t(\xi_{[t-1]})$ is a
subset of $\mathbb{U}$, for $t=1,\cdots,T$.}
Define the ambiguity set
\bgeqn
\label{eq:Ambiguityset-U-SD}
\mathcal{U} :=\{\vec{\mathfrak{u}} \mid \vec{\mathfrak{u}}=[\mathfrak{u}_1,\mathfrak{u}_2,\ldots,\mathfrak{u}_T]^\top,\
\mathfrak{u}_t{(\cdot,\xi_{[t-1]})}\in \mathcal{U}_t(\xi_{[t-1]}), \text{for any}\;  \xi_{[t-1]},
\ t=1,\ldots,T\},
\edeqn
where $\mathfrak{u}_1(\cdot,\xi_{[0]})=u_1(\cdot)$ is a
real-valued function in the deterministic ambiguity set $U_1$.
{ We say that
the sequence of  utility functions
$\{\mathfrak{u}_t{(\cdot,\xi_{[t]})}\}$
is state-independent if 
$\mathcal{U}_t$ is $\F_0$-measurable, i.e., a deterministic set, for $t=1,\ldots,T$. In
this case,
we 
write $u_t(\cdot)$ for
$\mathfrak{u}_t(\cdot,\xi_{[t-1]})$
and $\mathcal{U}_t$
for $\mathcal{U}_t(\xi_{[t-1]})$.}
\end{definition}

In this definition, each utility function in the set $\mathcal{U}_t(\xi_{[t-1]})$
depends on historical information $\xi_{[t-1]}$, which means the DM's utility preference
at { state $\xi_{t-1}$
is affected not only by the current state $\xi_{t-1}$ (
at the point of the decision making)
but also the earlier experiences.
The $\F_{t-1}$-measurability of $\mathcal{U}_t$ paves the way for the rectangularity of the ambiguity set to be stated in the forthcoming Proposition~\ref{prop-rectangular}.
}
A classical example of such state-dependent utility function is
the habit-formation utility model where
a DM's utility $\mathfrak{u}_t(c_t,h_t)$ at stage $t$
depends on both  the current consumption $c_t$ and
the historical habit level of consumption $h_t=\sum_{j=1}^{t}\alpha_j c_{t-j}$ where
$\alpha_j, j=1,\cdots,t$ are positive numbers.
The latter can be understood as
historical path $\xi_{[t-1]}$ dependent, see \cite{Duf10,DuS86}.
Likewise,  an investor who has experienced tough economic
circumstances in the past may be more risk averse at the current stage.
The structure of the ambiguity set depends on available information
in concrete decision making problems, we will come back to details
about this in Section \ref{sec-detail-set}.

To mitigate the model risk arising from ambiguity of the true utility functions
in the decision-making process under model \eqref{eq-msp-sdep},
we propose a robust counterpart
where the optimal policy is based on the worst-case sequence of utility functions:
\begin{eqnarray}
\label{eq-mpro-tc}
&\inmat{(MS-PRO-SD)}\nonumber\\
&\max\limits_{\x_{[
T]}}  & \inf\limits_{\vec{\mathfrak{u}}\in\mathcal{U}}\mathbb{E}\left[{u}_1(h_1\left(x_{1},\xi_1\right))
+\mathfrak{u}_2(h_{2}\left({\x_{2}(\xi_1)}, \xi_{2}\right),\xi_{1})+\cdots+\mathfrak{u}_T(h_T\left(
{x_{T}(\xi_{[T-1]})}, \xi_{T}\right),\xi_{[T-1]})\right] \nonumber\\
&\text { s.t. }& x_{1} \in \mathscr{X}_{1}, {\x_{t}(\xi_{[t-1]})} \in \mathscr{X}_{t}\left( {\x_{[t-1]}(\xi_{[t-2]})}, \xi_{[t-1]}\right),\ 
t=2, \ldots, T.
\end{eqnarray}
Here the maximin robust formulation is
based on a holistic view at the very beginning of the decision making process
on both the optimal policy  and the expected utility.
Specifically, instead of considering the maximin robust formulation at each stage, we compute, for every sequence of feasible decisions $x_{[
T]}$, the worst-case expected utility
$$
\mathbb{E}\left[\mathfrak{u}_1(h_1\left(x_{1},\xi_1\right))
+\mathfrak{u}_2(h_{2}\left({\x_{2}(\xi_1)}
, \xi_{2}\right),\xi_{1})+\cdots+\mathfrak{u}_T(h_T\left(
{x_{T}(\xi_{[T-1]})}, \xi_{T}\right),\xi_{[T-1]})\right]
$$
with a sequence of utility functions  $\vec{\mathfrak{u}}$ from the ambiguity set.
The optimal policy is subsequently identified via the largest worst-case expected utility value.
This kind of maximin robust approach is consistent with the philosophy of robust optimization,
particularly the recent multistage distributionally robust optimization models \cite{Sha16}.
We call it  multistage utility preference robust optimization models.


In the case that the utility functions are independent of states,
we may obtain a PRO counterpart for model
\eqref{eq-msp-indep-comp}:
\begin{equation}
\label{eq-pro-intc}
\inmat{(MS-PRO-SID)}\
\begin{array}{cl}
 \max\limits_{\x_{[
 T]}} & \inf\limits_{\vec{u}\in{U}}\mathbb{E}\left[u_1(h_1\left(x_{1},\xi_1\right))+u_2(h_{2}\left(
 {\x_{2}(\xi_1)}, \xi_{2}\right))+\cdots+u_T(h_T\left(
 {x_{T}(\xi_{[T-1]})}, \xi_{T}\right))\right] \\
\text { s.t. }& x_{1} \in \mathscr{X}_{1}, {\x_{t}(\xi_{[t-1]})} \in \mathscr{X}_{t}\left( {\x_{[t-1]}(\xi_{[t-2]})}, \xi_{[t-1]}\right),\
t=2, \ldots, T,
\end{array}
\end{equation}
where $\vec{u}=[u_1,u_2,\ldots,u_T]^\top$ and
${U} \subset
\mathcal{L}^p(\mathbb{R})\times \cdots \times \mathcal{L}^p(\mathbb{R})$
is an ambiguity set of the vectors of utility functions in product form.

\subsection{Time consistency}

An important and widely accepted
practice in multistage stochastic programming
is that the optimal policy 
determined at stage 1 should be consistent with the optimal sub-policy
to be set at stage $t$
for $t\geq 1$, which is known as time consistency or
Bellman's optimality principle \cite{ADE07,CLW12,WaF11}.
The principle
is not automatically
fulfilled in the multistage PRO models
 unless the ambiguity set of utility functions
 is structured properly. This motivates us to introduce the next definition.

\begin{definition}[Time consistency of dynamic policy]
A multistage PRO model is  said to be {\em time consistent}
if any optimal policy for the multistage PRO model over the entire time horizon
also satisfies the local optimality conditions of the sub-PRO model from period $t$ to period $T$, for any given historical $\xi_{[t-1]}$, for
all $t = 2, \ldots , T$.
\end{definition}

In multistage
risk minimization problems, the time consistency of the optimal dynamic policy
can be achieved if the corresponding
multistage risk measure is time consistent.
The concept of time consistency on multistage risk measure
characterizes an order keeping 
relationship among 
different stages:
given two investment positions $A$ and $B$, if $A$ is at least as good as
$B$ under
a specific risk measure at some future time $\tau$, and they are identical between now (time $t$) and the future time $\tau$,
then $A$ is at least as good as
$B$ under the same measure from today ($t$)'s perspective \cite{BoF06,Rus10}.
All time consistent risk measures can be written in a nested form \cite{Rus10}.

In multistage distributionally robust optimization,
time consistency of the optimal dynamic policy is related to
the structure of the dynamic ambiguity set of probability distributions.
If the distributionally robust counterpart can be written
in a nested form of stage-wise conditional distributionally robust counterparts,
known as the {\em rectangular set}
or {\em recursive multiple-priors set} \cite{EpS03,Sha16},
then the optimal policy is time consistent
and the dynamic programming equation 
may follow \cite{Sha16}.

Likewise, the time consistency of the optimal dynamic policy of the multistage
preference robust optimization problem
relies on the structure of the preference ambiguity set.
We shall define a property on the decomposability
of the preference set.
To this end, we introduce the concept of rectangularity of the ambiguity set of utility functions.

\subsubsection{Rectangularity of the ambiguity set}
To ease the exposition, we denote the reward function
$h_t\left(x_{t}(\xi_{[t-1]}), \xi_{t}\right)$ by
an $\mathcal{F}_{t}$-adaptable random variable $Z_t$.

\begin{definition}[Rectangularity of
the ambiguity set]
Let $\mathcal{U}$ be a nonempty set of utility sequences $\vec{\mathfrak{u}}$,
$\mathcal{U}$  is said to be {\em rectangular} if
\begin{equation}
\begin{array}{ll}
&
\inf\limits_{\vec{\mathfrak{u}}\in \mathcal{U}}\mathbb{E}
\left[{u}_1(Z_1)+\mathfrak{u}_2(Z_2,\xi_1)+\cdots+\mathfrak{u}_T(Z_T,\xi_{[T-1]})\right] \\
=&\inf\limits_{ u_1\in
\mathcal{U}_1}
\mathbb{E}
\Bigg[ u_1(Z_1)+\inf\limits_{u_2\in
\mathcal{U}_2(\xi_{[1]})
} \mathbb{E}_{|\mathcal{F}_1}\bigg[ u_2(Z_2)+\cdots  +\inf\limits_{u_T\in
\mathcal{U}_T(\xi_{[T-1]})
} \mathbb{E}_{|\mathcal{F}_{T-1}}\big[ u_T(Z_T)\big]
\bigg]\Bigg] \\
\end{array}
\label{eq:Rectangl}
\end{equation}
holds for {any $\{x_{t}\}$, $\{\xi_{t}\}$ and $\{Z_t:=h_t(x_t(\xi_{[t-1]}),\xi_{t})\}$}, where
\begin{equation}\label{eq-decomp}
\begin{array}{ll}
\mathcal{U}_{t}\left(\xi_{[t-1]}\right) & := \mathcal{U}_{t}\left( \vec{\mathfrak{u}}_{[1,t-1]}(\cdot,\xi_{[t-1]}) , \xi_{[t-1]}\right) \\
& =\left\{u_t \in \mathcal{L}^p(\mathbb{R}
) \bigg| \begin{array}{l}  \exists \vec{u}_{[t+1,T]} \in\mathcal{L}^p(\mathbb{R})\times \cdots\times \mathcal{L}^p(\mathbb{R})\\ \inmat{ such that }\left[
\vec{\mathfrak{u}}_{[1,t-1]}(\cdot,\xi_{[t-1]})
;\; u_t ;\; \vec{u}_{[t+1,T]}\right]^{\top} \in \mathcal{U}\end{array}\right\}, \\
&\qquad\qquad\qquad\qquad\qquad\qquad\qquad\qquad\qquad \forall \xi_{[t-1]} \in \mathcal{L}^p(\Omega,\mathcal{F}_{t-1},\mathbb{P};\mathbb{R}^{d_1}\times\dots\mathbb{R}^{d_t-1}).
\end{array}
\end{equation}
\end{definition}

{
The property has two important components: one is the
interchangeability of
infimum operation (with respect to the utility function) and
the conditional expectation operation,
which indicates
the consistency
between the global worst-case utility
sequence $\vec{\mathfrak{u}}_{[1,T]}$ and
the local worst-case
utility functions $\vec{\mathfrak{u}}_{[t,T]}$; 
the other is the consistency that
each of the current utility function
${u}_{t}
\in \mathcal{U}_{t}\left(\xi_{[t-1]}\right)$
can be paired up with the DM's potential
utility sequence at the remaining stages $[t+1,T]$
given the utility sequence over stage $[1,t-1]$,
to form an element in the specified ambiguity set ${\cal U}$.
This 
is similar to
time consistency in \cite{DeI15,Sha09} and local property in \cite{Rus10}.
In the forthcoming discussions, we will show that
the ambiguity set defined in Definition \ref{def-pro-set} satisfies the rectangularity.
}




Analogous to the rectangularity of
distributionally ambiguity set for multistage DRO problems \cite{Sha16} and the conditional (state-dependent) decomposition of uncertainty
set
in multistage
parametric robust
optimization problem \cite{DeI15},
the proposed rectangularity is built on a broad decomposable structure of the inner minimization problem without relying on a specific form of the ambiguity set.
The concept
differs from the  rectangularity in some
MSP literature \cite{PiS21} or MDP literature \cite{Iye05,WKR13}, where 
a product form of sub-ambiguity sets in different stages or states are considered.
As noted by Pichler and Shapiro
\cite{PiS21}, a product form of sub-ambiguity sets
is not enough to guarantee the decomposability of
a multistage DRO problem.
In the multistage PRO problems,
this means a product form with deterministic sub-ambiguity sets may lead to
state-independent PRO problems
which are not rectangular and time inconsistent
(see Appendix \ref{sec-example}).
By adding state-dependent property to sub-ambiguity sets,
we can show in Proposition \ref{prop-rectangular} that
the defined state-dependent preference ambiguity
set in Definition \ref{def-pro-set} is rectangular.


To study the time consistency of the optimal policy of a PRO model, we shall investigate whether the
global optimal solution is consistent with the local optimal solution of the sub-PRO problem over a sub-horizon.
If we consider the sub-PRO model of (MS-PRO-SID) \eqref{eq-pro-intc} from period $t$ to period $T$,
\begin{equation}\label{eq-pro-intc-t}
\begin{array}{cl}
\max\limits_{\bm{x}_{[t,T]}} & \inf\limits_{{\vec{u}_{[t,T]}\in {U}_{[t,T]}}}\mathbb{E}_{|\mathcal{F}_{t-1}}\left[u_t(h_t\left({\bm{x}_{t}(\xi_{[t-1]})},\xi_t\right))+u_{t+1}(h_{t+1}\left({\bm{x}_{t+1}(\xi_{[t]})}, \xi_{t+1}\right))+\cdots+u_T(h_T\left({\bm{x}_{T}(\xi_{[T-1]})}, \xi_{T}\right))\right] \\
\text { s.t. }& {\bm{x}_{s}(\xi_{[s-1]}) \in \mathscr{X}_{s}\left(\bm{x}_{[s-1]}(\xi_{[s-2]}), \xi_{[s-1]}\right),\ s=t, \ldots, T,}
\end{array}
\end{equation}
{where
$\bm{x}_{[t,T]}:=(x_t(\cdot)\dots,x_T(\cdot))$,
}
{${U}_{[t,T]}=\{\vec{u}_{[t,T]}\mid \exists \vec{u}_{[1,t-1]} \inmat{ such that } [\vec{u}_{[1,t-1]},\vec{u}_{[t,T]}]^\top \in U\}$,
we may find that the worst-case utility series $\vec{u}_{[t,T]}$ of the sub-PRO problem (\ref{eq-pro-intc-t})
depends on historical states $\xi_{[t-1]}$ and historical decisions $x_{[t-1]}$. However, 
the worst-case utility series of the global PRO problem (\ref{eq-pro-intc})
is a deterministic function series.
Then, such an inconsistency of the worst-case utility series between the global PRO problem and the sub-PRO problem leads the inconsistency of their optimal solutions.
An example which shows the point is given in Appendix \ref{sec-example}.

In what follows, we will show that the ambiguity set ${\cal U}$
{defined as in (\ref{eq:Ambiguityset-U-SD})}
is rectangular and the PRO model $\inmat{(MS-PRO-SD)}$
is time consistent.
To this end, we introduce an interchangeability principle for the preference robust counterpart.
{
In the literature of stochastic programming and variational
analysis, there have been several results on the principle of interchangeability, see for example
\cite[Proposition 5]{RuS03},
\cite[Theorem 14.60]{RoW09},
\cite[Proposition 6.37, Theorem 7.80]{SDR14} and
\cite[Theorem 2.1]{STR06}.
While these results are derived under some different conditions,
they are all stated in the finite dimensional space. Here we
need a principle of interchangeability which is in the infinite-dimensional space.




Let $\mathbb{Z}$ be a Polish space with
Borel field $\B(\mathbb{Z})$ and $\Omega$ be a sample 
space associated with filtration $\F$ and measure
$\mathbb{P}$.
We say a random function $f: \mathbb{Z} \times \Omega \rightarrow {\mathbb{R}}$ is a {\em Carath\'edory function} \cite{SDR14}
if $\omega \rightarrow f(z, \omega)$ is $\F$-measurable for every fixed $z\in \mathbb{Z}$ and
the function $z \rightarrow f(z,\omega)$ is continuous for almost every fixed $\omega\in\Omega$.

\begin{lemma}\label{lemma-interchange-expectation}
 Consider a Polish
 space $\mathbb{Z}$ and a probability space $(\Omega,\F,
 \mathbb{P}
 )$.
 Let $Z:\Omega \rightrightarrows \mathbb{Z}$ be a
 $\F$-measurable set-valued mapping with closed values.
 Let $\mathfrak{M}$ be
 a linear space of measurable functions $ \mathfrak{z}: \Omega \rightarrow \mathbb{Z}$
 and
 $\mathfrak{M}_Z := \{ \mathfrak{z} \in \mathfrak{M}: \mathfrak{z}(\omega) \in {Z}(\omega)
 \subset \mathbb{Z},\; \text{for a.e.}\ \omega\in \Omega \}$.
Let $f: \mathbb{Z} \times \Omega \rightarrow \bar{{\mathbb{R}}}$
 be a 
 Carath\'edory function.
 Suppose that either $\mathbb{E}\left[ \left(\inf_{z \in {Z}(\omega) } f(z, \omega)\right)_+\right]< \infty$ or $\mathbb{E}\left[\left(-\inf_{z \in {Z}(\omega)} f(z, \omega)\right)_+\right]< \infty$, where $(a)_+=\max(0,a)$.
 Then
\bgeqn
\mathbb{E}\left[\inf _{z \in {Z}(\omega)
} f(z, \omega)\right]=\inf_{\mathfrak{z} \in \mathfrak{M}_Z} \mathbb{E}\left[F_{\mathfrak{z}}\right],
\label{eq:interchange}
\edeqn
where $F_{\mathfrak{z}}(\omega):=f(\mathfrak{z}(\omega), \omega)$.
\end{lemma}
The main difference with  existing results
in the literature is that here
the infinite dimensionality of variable
$z$ poses
more rigorous requirements
on the measurability. For this, we
exploit
some fundamental results  about measurability of random functions in infinite-dimensional space from monograph \cite{AuF09}.
Another main difference is that here we consider $Z(\omega)$, which is a random set of functions in the space $\mathbb{Z}$ rather than a deterministic set of functions as in \cite[Proposition 5]{RuS03}),
\cite[Theorem 14.60]{RoW09}, or
\cite[Proposition 6.37, Theorem 7.80]{SDR14}.
Because of the differences, we include a proof in \ref{EC-interchangibility}
for completeness.

{ With the new version of the principle of interchangebility, we are able to
address the interchangeability in the expected utility case. The next lemma states this.
}

\begin{lemma}\label{lem-int-change}
Let $\mathscr{U}:=\{{u}\in \mathcal{L}^p({\mathbb{R}\rightarrow \mathbb{R}})\mid u
\text{ is a bounded and continuous
function}\}$ and $ \mathcal{U}(\tau)$ be a nonempty
subset of $\mathscr{U}$.
Let
$\mathfrak{M}_{\mathcal{U}}:=\{\mathfrak{u}\in \mathcal{L}^p(\mathbb{R} \times \mathbb{R}^{d}
\rightarrow \mathbb{R}) \mid \mathfrak{u}(\cdot,\tau) \in
\mathcal{U}(\tau),  \text{ for any } \tau\in\R^d \},
$
where $\mathcal{L}^p(\mathbb{R} \times \mathbb{R}^{d}
\rightarrow \mathbb{R})$ denotes the set of
all state-dependent
Lebesgue
integrable utility functions $\mathfrak{u}(\cdot,\cdot)$.
Let $(\Omega,\F,\mathbb{P})$ be a probability space
with sigma algebra $\F$ and probability measure $\mathbb{P}$.
Let $\eta:\Omega\to\R$ be a random variable
representing reward and $\xi:\Omega\to\R^d$ be a random vector representing state.
Then
\begin{equation}
\label{eq:lemma2-condi-expt-F-xi}
\begin{array}{ll}
&\inf\limits_{\mathfrak{u}\in \mathfrak{M}_{\mathcal{U}}}\mathbb{E}\left[\mathfrak{u}(\eta,\xi)\right]
= \mathbb{E}\left[  \inf\limits_{u\in \mathcal{U}(\xi)} \mathbb{E}\left[ u(\eta)\mid \F_{\xi
} \right]   \right],
\end{array}
\end{equation}
where $\mathcal{F}_{\xi}$ is the minimal sub-sigma algebra of $\F$  to which 
$\xi$ is
adapted.
\end{lemma}

{
We give an explanation
about the relation (\ref{eq:lemma2-condi-expt-F-xi}).
Observe first that $\mathfrak{M}_{\mathcal{U}}$
is a set of deterministic utility functions
in $ \mathcal{L}^p(\mathbb{R} \times \mathbb{R}^{d}
\rightarrow \mathbb{R})$ such that
$\mathfrak{u}(\cdot,\tau) \in \mathcal{U}(\tau)$
for any $\tau\in \R^d$.
The left-hand side of
equation (\ref{eq:lemma2-condi-expt-F-xi}) denotes
the worst-case expected utility value
for a given pair of
reward function $\eta$ and state $\xi$
when $\mathfrak{u}$ is restricted to set $\mathfrak{M}_{\mathcal{U}}$.
The right-hand side of (\ref{eq:lemma2-condi-expt-F-xi})
is the expectation of the worst-case expected utility value of $\eta$ conditional on $\F_\xi$
when the utility function is taken from
${\cal U}(\xi)$. Here the set ${\cal U}(\xi)$
depends on the state $\xi$. The difference between
$\mathfrak{M}_{\mathcal{U}}$ and  ${\cal U}(\xi)$
is that the former stipulates a set-valued mapping  from $\R^d$
to a set of utility functions with the specific structure ($\mathfrak{u}(\cdot,\tau) \in \mathcal{U}(\tau)$ for any state  $\tau$) whereas the latter is the image of the set-valued mapping when $\tau=\xi$.
We refer readers to \ref{EC-Lemma2} for the details of the proof.


}

}

With Lemma~\ref{lem-int-change}, we are ready to deliver
the rectangularity of the ambiguity set introduced in Definition \ref{def-pro-set}  in the next proposition.

\begin{proposition}\label{prop-rectangular}
{
Let ${\cal U}$ be defined as in Definition \ref{def-pro-set}.
Then (\ref{eq:Rectangl}) holds.}
\end{proposition}

In problem \eqref{eq-mpro-tc}, at each stage, the utility function is taken in the worst-case sense from a random set depending on historical information.
By Proposition \ref{prop-rectangular},
problem \eqref{eq-mpro-tc} can be rewritten as
\begin{equation}\label{eq-mpro-tc-dy}\begin{array}{cl}
\max\limits_{\x_{[
T]}}  & \inf\limits_{ u_1\in \mathcal{U}_1} \mathbb{E}
\Bigg[ u_1(h_1\left(x_{1},\xi_1\right))+\inf\limits_{u_2 \in \mathcal{U}_2(\xi_{[1]})} \mathbb{E}_{|\mathcal{F}_1}\bigg[  u_2(h_{2}\left(
{\x_{2}(\xi_1)}, \xi_{2}\right))+\cdots \\
&\qquad\quad +\inf\limits_{u_T \in \mathcal{U}_T(\xi_{[T-1]})} \mathbb{E}_{|\mathcal{F}_{T-1}}\big[ u_T(h_T\left(
{x_{T}(\xi_{[T-1]})}, \xi_{T}\right))\big]\cdots\bigg]\Bigg] \\
\text { s.t. }& x_{1} \in \mathscr{X}_{1},\ {\x_{t}(\xi_{[t-1]})} \in \mathscr{X}_{t}\left( {\x_{[t-1]}(\xi_{[t-2]})}, \xi_{[t-1]}\right),\
t=2, \ldots, T.
\end{array}\end{equation}
Here, ${U}_1={U}_1(\xi_{[0]})$ relies only on $\xi_{[0]}$
and thus is deterministic.

The reformulations from \eqref{eq-mpro-tc} to \eqref{eq-mpro-tc-dy} rely on the inter-changeability between operation
$\inf\limits_{\mathfrak{u}_t\in \mathcal{U}_t}$
and the expectation
$\mathbb{E}
$, $t=2,\ldots,T$.
However, the $\inf\limits_{u_t\in \mathcal{U}_t(\xi_{[t-1]})}$ cannot be further interchanged with $\mathbb{E}_{|\mathcal{F}_{t-1}}$
as the worst-case utility function $u_t$
and the preference ambiguity set $\mathcal{U}_t(\xi_{[t-1]})$ are determined
by information ${\mathcal{F}_{t-1}}$.
At each stage, we would meet a single period preference robust optimization problem which can be viewed as the well-studied static PRO models.

\subsubsection{Time consistency of \inmat{(MS-PRO-SD)}}

From \eqref{eq-mpro-tc-dy}, we can
see that the multistage preference robust utility function can be described in a nested form. Analogous to the multistage risk aversion
models \cite{CCL17,Rus10} and multistage distributionally robust optimization
models \cite{Sha16}, the nested form guarantees
the time consistency of the optimal dynamic policy of problem \eqref{eq-mpro-tc-dy}, i.e., it can be solved in a recursive dynamic programming procedure.

{
\begin{theorem}\label{theorem-dp-reform}
Let $\mathcal{U}_t(\xi_{[t-1]})$,
$t=2,\cdots,T$, and
$\mathcal{U}$
be defined as those 
in Definition 1. 
Assume:
(a) for $t=2,\cdots,T$,
the utility functions in $\mathcal{U}_t(\xi_{[t-1]})$ are
Lipschitz continuous
with modulus being bounded by
$\kappa(\xi_{[t-1]})$
and
$\mathcal{U}_t(\xi_{[t-1]})$ is a compact set for any $\xi_{[t-1]}$;
(b) the reward function $h_t:\R^{n_t}\times
\R^{d_t}\to \R$ is 
Lipschitz continuous in $x_t$ with
modulus $\sigma_t$ where
$\bbe_{\F_{t-1}}[\sigma_t]<+\infty$,
for $t=1,\cdots,T$;
(c) for $t=2,\cdots,T$, the feasible set $\mathscr{X}_t(x_{[t-1]},\xi_{[t-1]})$
is compact for { any}
fixed $x_{[t-1]}$ and $\xi_{[t-1]}$ and
as set-value mapping of
$x_{[t-1]}$,
$\mathscr{X}_t(\cdot,\xi_{[t-1]})$
is Lipschitz continuous.
Then the (MS-PRO-SD) problem has the following dynamic programming reformulation:
\begin{equation}
\label{eq:thm-recursive-formula}
\begin{array}{l}
V_{t}\left(x_{[t-1]}, \xi_{[t-1]}\right) =\max\limits_{x_{t} \in \mathscr{X}_{t}\left(x_{[t-1]}, \xi_{[t-1]}\right)}\inf\limits_{ u_t \in \mathcal{U}_t(\xi_{[t-1]})}\mathbb{E}_{|\mathcal{F}_{t-1}}\left[ u_{t}\left(h_t(x_{t}, \xi_{t})\right)+V_{t+1}\left(x_{[t]}, \xi_{[t]}\right)\right]
\end{array}
\end{equation}
for $t=1, \ldots, T$,
where
$V_{T+1}(\cdot, \cdot) := 0$, and $V_1$
coincides with
 the optimal value of problem {\rm (MS-PRO-SD)}. The optimal policy of {\rm (MS-PRO-SD)} is time consistent. 
\end{theorem}
}

The proof is given in \ref{EC-theorem-dp-reform}.
In some applications,
construction
of the scenario (historical path) dependent preference ambiguity set $\mathcal{U}_t(\xi_{[t-1]})$ is a bit complicated for practical use. There are potentially two ways to simplify.
One is to
consider the Markovian preference ambiguity set $\mathcal{U}_t(\xi_{t-1})$ which relies only
on the randomness at current stage $\xi_{t-1}$.
{The other is
to consider a discrete approximation of the random process $\xi_{[t-1]}$.}



\begin{remark}
It is worth noting that, we only need to interchange the order of
the conditional expectation with the infimum operator
as the considered utility function is
additive in both probability and  temporal dimension.
For some utility considering nonlinear elasticity of intertemporal substitution, such as recursive or temporal utility functions,  Lemma \ref{lemma-interchange-expectation} is not enough to guarantee 
time consistency as we need a stronger version which studies the interchangeability between the infimum operator and both conditional expectation operator and  utility functions in previous stages.
\end{remark}

\section{Construction of the ambiguity set}
\label{sec-detail-set}

The structure of the ambiguity set of utility functions
is determined by available information
on the DM's utility preferences at each stage. Here we follow two approaches
which are widely used in the literature of PRO models:
pairwise comparison \cite{AmD15,GuX21,HuM15} and nominal utility approach \cite{HuS17,WaX20,WaX21}.
The former elicits DM's preferences via pairwise comparison questionnaires such as a lottery vs
a deterministic gain/loss and translates the preferences (answers)
into a characterization/specification of the true
utility function, whereas the latter constructs an ambiguity set of utility functions
in a neighborhood of a plausible nominal utility function.

To simplify the discussion, here we restrict the domain of utility functions
to $[a,b]$ which means the range of reward function $h_t(x_t,\xi_t)$ falls within the interval, and normalize the utility function with $u(a)=0$ and $u(b)=1$ for $t=1,\cdots, T$. The normalization does not affect the utility preferences.
Let $\mathscr{U}$ be the set of continuous and normalized non-decreasing
utility functions in $\mathcal{L}^p([a,b])$ with  $u(a) =0,\ u(b) =1$, and $\mathscr{U}^c$ a subset  where the utility functions are concave.

\subsection{Pairwise comparisons}

We begin with the pairwise comparison approach which is based on Von Neumann-Morgenstern's expected utility theory, that is, any preference between two random prospects
by the DM can be represented by expected utility of the random prospects albeit
such a utility is unknown. To narrow down the scope of the true utility function, one may
design more pairwise comparison questionnaires and ask the DM to make a choice on each pair of them, see Armbruster and Delage  \cite{AmD15} for details.

In a dynamic decision making process, the DM's preference
depends on not only the { stage she/he is standing,
but also the historical path.}
The latter is particularly important because the DM's preference
may be affected by the current environment.
For instance, an investor in a bull market
may prefer high growth stocks with higher tolerance to volatility,
while in a bear market, she/he may prefer less volatile
stocks even with a lower return rate. This means her/his answer to the same
questionnaires may be affected by her/his risk attitude under different macro-market conditions.
This motivates us to introduce ambiguity set of
state-dependent utility functions in Definition \ref{def-pro-set} by setting
\begin{eqnarray}
\label{eq:U-moment}
\mathcal{U}^P_t(\xi_{[t-1]})
:=\left\{u\in \mathscr{U}^c
\left|  \begin{array}{l}
  z_{k}(\xi_{[t-1]})\mathbb{E}\left[u\left(W_{k}\right)\mid \xi_{[t-1]}\right] \geq z_{k}(\xi_{[t-1]}) \mathbb{E}\left[u\left(Y_{k}\right)\mid \xi_{[t-1]}\right],\\\qquad\qquad\qquad \inmat{for}\; k=1,\ldots,K,\\
 \inmat{Lip}(u) \leq L(\xi_{[t-1]})
  \end{array} \right.\right\}
  \end{eqnarray}
where $\{(W_k,Y_k), k=1,\cdots,K\}$
is a set of  prospects for pairwise comparison.
Note that this set may be fixed or evolved 
over the process, which means the questionnaires
used in stage $t-1$ will
be used in stage $t$, but the DM might have different
answers due to the change of stage/state.
Here $z_{k}(\xi_{[t-1]})\in \{+1,-1,0\}$
is used to indicate the choice of the decision-maker at stage $t$.
If the DM prefers $W_k$ to $Y_k$, then
 $z_{k}(\xi_{[t-1]})=1$,
otherwise $z_{k}(\xi_{[t-1]})=-1$.
In the case of no preference,
$z_{k}(\xi_{[t-1]})=0$.
Consequently the ambiguity of the utility functions is time-dependent as opposed to static
in one stage PRO models.
{Under Assumption \ref{Assu:VNM-DM},
$\mathcal{U}^P_t(\xi_{[t-1]})
\neq \emptyset$.}
$ \inmat{Lip}(u) \leq L(\xi_{[t-1]}) $ means that $u$ is Lipschitz continuous with modulus bounded by $L(\xi_{[t-1]})$,
$\mathscr{U}^c$ involves the  concavity constraint of $u$.
It
means that the DM is risk averse at all stages
over the time horizon.
Obviously $\mathcal{U}_t^{P}(\xi_{[t-1]})$ is a convex set.

{Armbruster and Delage \cite{AmD15} show that a static PRO problem with pairwise comparison ambiguity set
can be reformulated as an LP,
when the supports of $W_k$ and
$Y_k$ are finite.
In Section \ref{sec-sce-pariwise}, we will derive a tractable
LP reformulation of multistage PRO problem with the ambiguity set 
defined as $\mathcal{U}_t^{P}(\xi_{[t-1]})$.}
{
Note that
the ambiguity set constructed
as such
in (\ref{eq:U-moment})
has some limitations: the utility function is independent of past decisions or
the current financial position (e.g. cumulative wealth up to date).
The reformulation under the scenario tree will be much more complex if the ambiguity set is 
decision-dependent or  wealth-dependent (when $z_k(s)$
in \eqref{eq-sce-reformu2}
is replaced by $z_{k,s}(x(s^-))$, it contains bi-linear terms and 0-1 valued non-smooth functions).
}

Let
$$
\mathscr{S}:=\{a\} \cup \bigcup_{k=1}^{K}\left(\operatorname{supp}\left(Y_{k}\right) \cup \operatorname{supp}\left(W_{k}\right)\right) \cup \{b\}
$$
and $N:=|\mathscr{S}|$ denotes the cardinality of set
$\mathscr{S}$,
let $\{{y}_{j}\}_{j=1,\ldots,N}$ be the ordered sequence of points in $\mathscr{S}$ with fixed
${y}_{1}=a$, ${y}_{N}=b$.
In the forthcoming discussions, we will
use utility values at $\mathscr{S}$ to
characterize the property of the true unknown
utility function. {The details are given in EC.4.1 and EC.4.3.
}

\subsection{$\zeta$-ball approach}\label{zeta-ball}
In some decision making problems,
a DM may be able to ``roughly''  identify a
nominal utility function which captures most of the DM's preferences
either elicited through empirical data, or
based on subjective judgement or from partially
elicited preference information, but there is
incomplete information to tell whether the nominal utility is the true utility.
Under such a circumstance,
it might be sensible to consider
a set of utility functions near the nominal utility
and base the optimal decision on the worst-case utility function
from the set. We call this a nominal approach.

We begin by defining a kind of semi-distance between any two utility functions.
Let $\mathscr{G}$ be a set of measurable functions defined over $[a,b]$.
For  $u,v\in \mathscr{U}$,
 define the semi-distance
 between $u$ and $v$ by
 {
$\dd_{\mathscr{G}} (u,v):=\sup_{g \in \mathscr{G}}
 \left|\int_a^b g(z)du(z)- \int_a^b g(z)dv(z) \right|$,
where the integrals are in the sense of Lebesgue-Stieltjes integration,
$g$ might be viewed as
a test function and $\dd_{\mathscr{G}} (u,v)=0$
means that for all of the test functions in $\mathscr{G}$,
there is no difference between $u$ and $v$ albeit that $u\neq v$.
}
In the case that the utility functions in $\mathscr{U}$ are normalized with $u(a)=0, u(b)=1
$, $\dd_{\mathscr{G}} (u,v)$ resembles the pseudo-metric of $\zeta$-structure in probability theory.
In this paper, we are interested in two cases:
	\begin{equation}
		\mathscr{G}=\mathscr{G}_{L}:=\left\{g: [a,b]\to\mathbb{R} \mid \ g\text{ is } \text{ Lipschitz continuous with modulus bounded by 1}\right\}
\label{eq:G-kant}
	\end{equation}
	and
\bgeqn
\mathscr{G} = \mathscr{G}_{I} =\left\{ g:=\mathbbm{1}_{(a, z]}(\cdot) \mid
	\inmat{where}\;
 \mathbbm{1}_{(a, z]}(s):=
	1 \; \text{if }\; s\in (a, z] \; \inmat{and} \;
	0 \; \text{otherwise}
	\right\}.
\label{eq:G-kolm}	
\edeqn
The former corresponds to the Kantorovich metric, denoted by $\dd_K(u,v)$,
and the latter corresponds to the uniform Kolmogorov metric.
With the definition of the $\zeta$-metric, we are ready to introduce the definition of $\zeta$-ball in the space of the utility functions $\mathscr{U}$. We begin with the static case.

\begin{definition}[Static $\zeta$-ball of utility functions] Let $\mathscr{U}$ be the set of all continuous, non-decreasing utility functions defined over interval $[a,b]$,
$u(a)=0$, $u(b)=1$ for all $u\in \mathscr{U}$.
For a fixed $\tilde{u}\in \mathscr{U}$, the $\zeta$-ball of utility functions in $\mathscr{U}$ centered at $\tilde{u}$ with radius $r$ under metric $\dd_\mathscr{G}$ is defined as:
\bgeqn
\mathbb{B}(\tilde{u},r) :=\left\{ {u}\in \mathscr{U} \mid \dd_\mathscr{G}({u},\tilde{u})\leq r\right\}.
\label{eq:U-zeta-ball}
\edeqn
\end{definition}


In this paper, our focus is on the construction of an ambiguity set of
a sequence of state-dependent utility functions  specified in
Definition \ref{def-pro-set}.

\begin{definition}[{Dynamic $\zeta$-ball based ambiguity set of utility functions}]
Consider the ambiguity set in \eqref{eq:Ambiguityset-U-SD}.
For given nominal state-dependent utility function
$\tilde{\mathfrak{u}}_t(\cdot,\xi_{[t-1]})\in \mathscr{U}$,
define for all $\xi_{[t-1]}$,
\begin{eqnarray}
\label{eq:U-normial-dynamic}
\mathcal{U}^{\mathbb{B}}_t(\xi_{[t-1]}) :=\left\{u\in \mathscr{U}^c\left| \begin{array}{l}
u\in \mathbb{B}(\tilde{\mathfrak{u}}_t(\cdot,\xi_{[t-1]}),r_t(\xi_{[t-1]})), \\
 \inmat{Lip}(u) \leq L(\xi_{[t-1]})
\end{array}\right.
\right\}.
\end{eqnarray}
\end{definition}
In this formulation, $\mathcal{U}^{\mathbb{B}}_t(\xi_{[t-1]})$
is determined by the center
 $\tilde{\mathfrak{u}}_t(\cdot,\xi_{[t-1]})$, the radius $r_t(\xi_{[t-1]})$ and
the pseudo-metric $\dd_\mathscr{G}$.
The choice of functions in set $\mathscr{G}$ may depend on
{ historical data $\xi_{[t-1]}$.}
The nominal utility function
$\tilde{\mathfrak{u}}_t(\cdot,\xi_{[t-1]})$ may be identified from empirical data, that is,
the utility function is inferred from the DM's past utility preferences
and the feedback (represented by historical path
$\xi_{[t-1]}$).
As the time goes on, we can collect more data/information about the DM's preferences
and subsequently a more accurate nominal utility as well as
a smaller radius.

{

\begin{proposition}\label{prop-compact}
Let
$\mathcal{U}^{\mathbb{B}}_t(\xi_{[t-1]})$
{ and $\mathcal{U}^{P}_t(\xi_{[t-1]})$}
be defined as in
(\ref{eq:U-moment}) and
(\ref{eq:U-normial-dynamic}).
Then the following assertions hold.
\begin{itemize}
    \item[(i)] For each fixed $\omega$,
    $\mathcal{U}^{\mathbb{B}}_t(\xi_{[t-1]}(\omega))$
and $\mathcal{U}^{P}_t(\xi_{[t-1]}(\omega))$ are compact sets.

\item[(ii)]
If $\tilde{\mathfrak{u}}_t(\cdot,\xi_{[t-1]}), r_t(\xi_{[t-1]})$ and $ L(\xi_{[t-1]})$
are continuous in $\xi_{[t-1]}$,
then
$\mathcal{U}^{\mathbb{B}}_t(\xi_{[t-1]}(\cdot))$
and $\mathcal{U}^{P}_t(\xi_{[t-1]}(\cdot)$ are $\F_{t-1}$-measurable.

\item[(iii)] The ambiguity $\cal{U}$ constructed from 
$\mathcal{U}^{\mathbb{B}}_t(\xi_{[t-1]})$ ($\mathcal{U}^{P}_t(\xi_{[t-1]})$)
in the form of \eqref{eq:Ambiguityset-U-SD} satisfies
the rectangularity (the conditions in Definition \ref{def-pro-set}).

\end{itemize}
\end{proposition}
}

The next proposition quantifies the difference between
two $\zeta$-balls of utility functions with different nominals and radii under the
Hausdorff distance. For any two sets $U, V\subset \mathscr{U}$, define
$
\mathbb{D}(U,V;\dd_{\mathscr{G}})
:= \sup_{u\in U}\inf_{v\in V} \dd_{\mathscr{G}}(u,v),
$
which quantifies the deviation of $U$ from $V$
and
$
\mathbb{H}(U,V;\dd_{\mathscr{G}})  := \max\left\{\mathbb{D}(U,V;\dd_{\mathscr{G}}), \mathbb{D}(V,U;\dd_{\mathscr{G}})\right\},
$
the Hausdorff distance between the two sets under the pseudo-metric.

\begin{proposition}\label{prop-zeta-err}
Let $u,v\in \mathscr{U}$ and $r_1,r_2\in \mathbb{R}_+$. Then
\bgeqn
\mathbb{H}(\mathbb{B}(u,r_1), \mathbb{B}(v,r_2);   \dd_\mathscr{G}) \leq \dd_\mathscr{G}(u,v) + |r_2-r_1|.
\label{eq:H-two-balls-U}
\edeqn
In particular, if $u^*$ is the true utility function and $u_{ref}$ is a nominal utility function, then
 \bgeq
\mathbb{H}(u^*, \mathbb{B}(u_{ref},r);   \dd_\mathscr{G}) \leq \dd_\mathscr{G}(u^*,u_{ref}) + r.
\edeq
\end{proposition}

Inequality (\ref{eq:H-two-balls-U}) means that the Hausdorff distance of two balls
is bounded by the distance of their centers
plus the difference of the radii. With the proposition,
we are ready to present a multistage PRO model with
the ambiguity set defined via \eqref{eq-mpro-tc-dy} and \eqref{eq:U-normial-dynamic}
as follows:
\begin{equation}\label{eq-mpro-tc-dy-nomial}
\begin{array}{cl}
\max\limits_{\x_{[
T]}}  & \inf\limits_{
u_1\in \mathcal{U}^{\mathbb{B}}_1
} \mathbb{E}
\Bigg[ {u}_1(h_1\left(x_{1},\xi_1\right))+\inf\limits_{u_2\in \mathcal{U}^{\mathbb{B}}_2(\xi_{[1]})}
\mathbb{E}_{|\mathcal{F}_1}\bigg[  u_2(h_{2}\left(
{\x_{2}(\xi_1)}, \xi_{2}\right))+\cdots \\
&\qquad\quad +\inf\limits_{u_T \in
\mathcal{U}^{\mathbb{B}}_T(\xi_{[T-1]})
} \mathbb{E}_{|\mathcal{F}_{T-1}}\left[ u_T(h_T\left(
{x_{T}(\xi_{[T-1]})},\xi_{{T}}\right))\right]\cdots\bigg]\Bigg] \\
\text { s.t. }& x_{1} \in \mathscr{X}_{1}, {\x_{t}(\xi_{[t-1]})} \in \mathscr{X}_{t}\left(
{\x_{[t-1]}(\xi_{[t-2]})}, \xi_{[t-1]}\right),t=2, \ldots, T.
\end{array}\end{equation}
Here, $\mathbb{B}(\tilde{\mathfrak{u}}_1,r_1)$ in $\mathcal{U}^{\mathbb{B}}_1$ relies only on
deterministic nominal utility $\tilde{\mathfrak{u}}_1$ and radius $r_1$.
By Theorem \ref{theorem-dp-reform},
\eqref{eq-mpro-tc-dy-nomial} can be computed by the following dynamic programming equation,
\begin{equation}\label{eq-mpro-tc-nomial-dy}
\begin{array}{l}
V_{t}\left(x_{[t-1]}, \xi_{[t-1]}\right) =\max\limits_{x_{t} \in \mathscr{X}_{t}\left(x_{[t-1]}, \xi_{[t-1]}\right)}\inf\limits_{ u_t \in
\mathcal{U}^{\mathbb{B}}_t(\xi_{[t-1]})
}\mathbb{E}_{|\mathcal{F}_{t-1}}\left[ u_{t}\left(h_t(x_{t}, \xi_{t})\right)+V_{t+1}\left(x_{[t]}, \xi_{[t]}\right)\right].
\end{array}
\end{equation}

From computational point of view, problem
(\ref{eq-mpro-tc-nomial-dy}) is still not easy to solve
because the inner minimization problem is infinite dimensional.
This motivates us to develop an approximation scheme where
the ball of utility functions $\mathbb{B}(\tilde{\mathfrak{u}}_t(\cdot,\xi_{[t-1]}),r_t(\xi_{[t-1]}))$
is approximated by a ball of piecewise
linear utility functions.

\subsubsection{Piecewise-linear utility functions}

Let $y_1<\cdots<y_N$ be an ordered sequence of points in $[a,b]$ with $y_1=a$ and $y_N=b$
and $Y:=\{y_1,\cdots,y_{N}\}$.
Let
$\mathscr{U}_N $
be a class of continuous, non-decreasing, piecewise linear functions
defined over the interval $[y_1,y_N]$
with breakpoints on $Y$. For a given $v\in \mathscr{U}_N$, let
\bgeqn
\mathbb{B}_N(v,r) :=\left\{ u\in \mathscr{U}_N \mid \dd_\mathscr{G}(u,v)\leq r\right\}
\label{eq:U-zeta-ball-PC}
\edeqn
and
\begin{eqnarray*}
\label{eq-nom-approx}
\mathcal{U}^{\mathbb{B}_N}_t(\xi_{[t-1]}) :=\left\{u\in \mathscr{U}^c\left| \begin{array}{l}
u\in \mathbb{B}_N(\tilde{\mathfrak{u}}_t(\cdot,\xi_{[t-1]}),r_t(\xi_{[t-1]})) \\
 \inmat{Lip}(u) \leq L(\xi_{[t-1]})
\end{array}\right.
\right\}
\end{eqnarray*}
for a given nominal utility function
$\tilde{\mathfrak{u}}_t(\cdot,\xi_{[t-1]})\in \mathscr{U}_N$.
We propose to solve (\ref{eq-mpro-tc-nomial-dy}) by solving
\begin{equation}\label{eq-mpro-tc-nomial-dy-PLA}
\begin{array}{l}
\widetilde{V}_{t}\left(x_{[t-1]}, \xi_{[t-1]}\right) =\max\limits_{x_{t} \in \mathscr{X}_{t}\left(x_{[t-1]}, \xi_{[t-1]}\right)}\inf\limits_{ u_t \in
\mathcal{U}^{\mathbb{B}_N}_t(\xi_{[t-1]})
}\mathbb{E}_{|\mathcal{F}_{t-1}}\left[ u_{t}\left(h_t(x_{t}, \xi_{t})\right)+\widetilde{V}_{t+1}\left(x_{[t]}, \xi_{[t]}\right)\right].
\end{array}
\end{equation}
To justify this, we derive the error between
$V_{t}\left(x_{[t-1]}, \xi_{[t-1]}\right)$
and $\widetilde{V}_{t}\left(x_{[t-1]}, \xi_{[t-1]}\right)$.

\begin{remark}
By restricting
the nominal utility function to be piecewise linear, it is easier to estimate the function from the customer/investor in practice.
We preset two endpoints $a,b$ and some values in $[a,b]$,
and then let the customer/investor score on these values under different scenarios.
By collecting and normalizing the scores in different scenarios and linking the utility scores by
a piecewise linear function, we obtain a normalized nominal utility function in each scenario.
The radius describes the error in the scoring process which depends on the credibility of the scores.
\end{remark}

Differing from $\mathbb{B}(u,r)$ defined in (\ref{eq:U-zeta-ball}), the $\zeta$-ball consists of piecewise
linear utility functions only. In what follows, we quantify the difference between
$\mathbb{B}(u,r)$ and $\mathbb{B}_N(v,r)$ under the $\zeta$-metric so that we will be able to
assess
the impact
when we replace the former with the latter in the utility preference robust optimization model.







\begin{lemma}
\label{l-dist-B-B_N-zeta}
Let $u\in \mathscr{U}_N$ and $v\in \mathscr{U}$, let
 $\mathbb{B}_N(u,r)$ and $\mathbb{B}(v,r)$ be defined as in (\ref{eq:U-zeta-ball-PC}) and (\ref{eq:U-zeta-ball}) respectively.
 Assume: (a) $\mathbb{B}(v,r)$ consists of all
 utility functions which are Lipschitz continuous
 with modulus being bounded by $L$,
 (b) $\mathscr{G}=\mathscr{G}_{L}$ or $\mathscr{G}_{I}$.
Then
\bgeqn
\mathbb{H}(\mathbb{B}_N(u,r), \mathbb{B}(v,r);   \dd_\mathscr{G}) \leq \dd_\mathscr{G}(u,v) + 4\max(2,L)\beta_N.
\label{eq:H-two-balls-U-U_N}
\edeqn
In the case when $u=v_N$ is a projection of $v$ on $\mathscr{U}_N$,
\bgeqn
\mathbb{H}(\mathbb{B}_N(v_N,r), \mathbb{B}(v,r);
\dd_\mathscr{G}) \leq
6\max(2,L)\beta_N,
\label{eq:H-two-balls-U-U_N-u-proj-u}
\edeqn
where $L$
and $\beta_N$ are defined as in Proposition \ref{p-PC-appr}.
\end{lemma}


We are now ready to present the main result of this section.


\begin{theorem}[Error bound]\label{theorem-error-bound}
Let $V_{t}\left(x_{[t-1]}, \xi_{[t-1]}\right)$
and $\tilde{V}_{t}\left(x_{[t-1]}, \xi_{[t-1]}\right)$
be defined as in (\ref{eq-mpro-tc-nomial-dy})
and \eqref{eq-mpro-tc-nomial-dy-PLA}, respectively.
Let $\{\tilde{\mathfrak{u}}_t(\cdot,\xi_{[t-1]})\}$ be a
sequence of nominal utility functions
and $\{\tilde{\mathfrak{u}}_t^N(\cdot,\xi_{[t-1]})\}$
its piecewise linear approximations.
Let
$$
\beta_N(\xi_{[t-1]}):= \max_{i=2,\cdots,N} (y_i-y_{i-1}),
$$
where the breakpoints are chosen
according to historical data $\xi_{[t-1]}$.
Assume that $\tilde{\mathfrak{u}}_t(\cdot,\xi_{[t-1]})$ is Lipschitz continuous with modulus $L(\xi_{[t-1]})$.
Then
\bgeqn
\left|V_{t}\left(x_{[t-1]}, \xi_{[t-1]}\right)-\tilde{V}_{t}\left(x_{[t-1]}, \xi_{[t-1]}\right)\right|\leq
\sum_{s=t}^T 6\mathbb{E}\left[
\max(2,L(\xi_{[s-1]}))\beta_N(\xi_{[s-1]}) \mid \mathcal{F}_{t-1}
\right]
\label{eq:errr-bnd-v}
\edeqn
for $t=1,\dots,T$. In the case when
$\beta_N(\xi_{[s-1]})$ and $L(\xi_{[s-1]})$ are independent of states,
\bgeq
\left|V_{t}\left(x_{[t-1]}, \xi_{[t-1]}\right)-\tilde{V}_{t}\left(x_{[t-1]}, \xi_{[t-1]}\right)\right|\leq  6 (T-t+1) \max(2,L)\beta_N.
\edeq
\end{theorem}

\subsubsection{Kantorovich ball}

Let $\tilde{u}
\in \mathscr{U}_N$. We consider a ball in the space of $\mathscr{U}_N$
with the Kantorovich metric
\bgeqn
\mathbb{B}_K(\tilde{u}
,r) = \left\{u\in \mathscr{U}_N \mid \dd_K(u,\tilde{u}
)\leq r\right\}.
\label{eq:Kant-ball-utl}
\edeqn
  In what follows, we derive tractable formulation for computing $\dd_K(u,\tilde{u})$.
Let $g\in \mathscr{G}$ where $\mathscr{G}$ consists of all Lipschitz continuous functions defined on $[a,b]$ with modulus bounded by $1$.
 By definition
$$
\int_a^b g(t)du(t) = \sum_{j=2}^N \beta_j\int_{y_{j-1}}^{y_j} g(t)dt,
$$
where $\beta_j$ denotes the slope of $u$ at interval
$[y_{j-1}, y_j]$. Since for each $g\in  \mathscr{G}$,
$-g\in  \mathscr{G}$,
$$
\dd_K(u,\tilde{u}) = \sup_{g\in \mathscr{G}}\sum_{j=2}^N (\beta_j-\tilde{\beta}_j)\int_{y_{j-1}}^{y_j} g(t)dt,
$$
where $\tilde{\beta}_j$ denotes the slope of $\tilde{u}$ at interval
$[y_{j-1}, y_j]$.
Note that in this formulation,
 $\dd_K(u,\tilde{u})$ depends on the slopes of $u,\tilde{u}$ rather than their function values,
$ \sum_{j=2}^N \beta_j({y_{j}}-{y_{j-1}}) =u(b)-u(a) =1 $,
$  \sum_{j=2}^N \tilde{\beta}_j({y_{j}}-{y_{j-1}}) =u(b)-u(a) =1$. 
 Let $w_j :=\int_{y_{j-1}}^{y_j} g(t)dt$ and
$z_j=g(y_j)$, $j=2,\dots,N$. Since
$|g(y)-g(y_{j-1})|\leq y-y_{j-1}$ for all $y\in [y_{j-1},y_j]$,
we have
$$
z_{j-1}(y_{j}- y_{j-1})-\frac{1}{2}(y_{j}- y_{j-1})^2 \leq  w_j\leq z_{j-1}(y_{j}- y_{j-1})+\frac{1}{2}(y_{j}- y_{j-1})^2
$$
for $j=2,\cdots,N$.
Likewise, since
$|g(y_{j})-g(y)|\leq y_{j}-y$ for all $y\in [y_{j-1},y_j]$,
we have
$$
z_{j}(y_{j}- y_{j-1})-\frac{1}{2}(y_{j}- y_{j-1})^2 \leq  w_j\leq z_{j}(y_{j}- y_{j-1})+\frac{1}{2}(y_{j}- y_{j-1})^2
$$
for $j=2,\cdots,N$.
Consequently
\begin{subequations}
\label{eq:Kant-u-v-LP}
\begin{eqnarray}
&\dd_K(u,\tilde{u})=&\nonumber\\
&\displaystyle \max_{w_2,\cdots,w_N, z_1,\cdots,z_N}&\sum_{j=2}^N  (\beta_j-\tilde{\beta}_j)w_j\\
&\inmat{s.t.} &  w_j\leq z_{j-1}(y_{j}- y_{j-1})+\frac{1}{2}(y_{j}- y_{j-1})^2,\ j=2,\cdots,N, \\
&&-w_j\leq -z_{j-1}(y_{j}- y_{j-1})+\frac{1}{2}(y_{j}- y_{j-1})^2,\ j=2,\cdots,N, \\
&&  w_j\leq z_{j}(y_{j}- y_{j-1})+\frac{1}{2}(y_{j}- y_{j-1})^2,\  j=2,\cdots,N, \\
&&-w_j\leq -z_{j}(y_{j}- y_{j-1})+\frac{1}{2}(y_{j}- y_{j-1})^2,\ j=2,\cdots,N.
\end{eqnarray}
\end{subequations}
Problem (\ref{eq:Kant-u-v-LP}) is a linear program.
Using Lagrange duality, we can reformulate it as
 \begin{subequations}
\label{eq:Kant-u-v-LP-dual}
\begin{eqnarray}
\displaystyle \min_{\lambda,\mu,\rho,\phi}
&& \frac{1}{2}\sum_{j=2}^N  (\lambda_j + \mu_j + \rho_j + \phi_j)(y_{j}- y_{j-1})^2   \\
\inmat{s.t.}\quad                  &&  \tilde{\beta}_j -\beta_j + \lambda_j-\mu_j + \rho_j - \phi_j =0,\  j=2,\cdots,N, \\                       && (\mu_{2}-\lambda_{2})(y_{2}-y_1) =0, \\
&& (\mu_{j+1}-\lambda_{j+1})(y_{j+1}-y_j) + (\rho_{j}-\phi_{j})(y_{j}-y_{j-1})   =0, j=2,\cdots,N-1, \\
 && (\rho_{N}-\phi_{N})(y_{N}-y_{N-1})   =0, \\
&& \mu_j,\lambda_j,\rho_j,\phi_j \geq 0, j=2,\cdots,N.
\end{eqnarray}
\end{subequations}

The discussion above shows that we can obtain the Kantorovich distance $\dd_K(u,\tilde{u})$ by solving a linear program.
This will facilitate us to derive tractable formulations for solving
problem \eqref{eq-mpro-tc-nomial-dy-PLA} 
by imbedding (\ref{eq:Kant-u-v-LP-dual}) into
the inner minimization problem.


\subsubsection{Tractable formulation of dynamic program
\eqref{eq-mpro-tc-nomial-dy-PLA}}\label{sec-tractable-sddp}

We can easily incorporate
the tractable formulations of the Kantorovich
ball into
the dynamic programming  equation \eqref{eq-mpro-tc-nomial-dy-PLA}
and develop tractable formulations for the latter.
To comply with the setting in Theorem \ref{theorem-error-bound},
we need to impose Lipschitz continuity on the nominal utility function $\tilde{\mathfrak{u}}_t(\cdot,\xi_{[t-1]})$
and its derivative $\tilde{\mathfrak{u}}_t'(\cdot,\xi_{[t-1]})$ as well as the concavity of the utility function.

\begin{theorem}\label{theorem-dp-sddp}
Consider
\begin{eqnarray}
\label{eq:U-normial-dynamic-addLip}
\mathcal{U}^{K}_t(\xi_{[t-1]}) :=\left\{u\in \mathscr{U}^c\left| \begin{array}{l}
u\in \mathbb{B}_K(\tilde{\mathfrak{u}}^N_t(\xi_{[t-1]}),r_t(\xi_{[t-1]})) \\
 \inmat{Lip}(u) \leq L(\xi_{[t-1]})
\end{array}\right.
\right\}
\end{eqnarray}
for all $\xi_{[t-1]}$.
Suppose that the optimal value function at period $t+1$ is
$\widetilde{V}_{t+1}\left(x_{[t]}, \xi_{[t]}\right)$.
Given historical data $\xi_{[t-1]}$ and historical decision $x_{[t-1]}$,
$\xi_t$ is discretely distributed with $S$ scenarios
$\xi_t^1,\ldots,\xi_t^S$ and appearing probability  $\mathbb{P}(\xi_t=\xi^i_t|\xi_{[t-1]})$, $i=1,\ldots,S$,
then the optimal decision $x_t$ at stage $t$ can be derived by solving the following programming problem,
\begin{subequations}
\label{eq:reformu-gx-kantorovich}
\begin{eqnarray}
&\max & \quad \theta_{N-1}+\sum_{i=1}^{S}\left(\mu_{i,N}+ \mathbb{P}(\xi_t=\xi^i_t|\xi_{[t-1]}) \widetilde{V}_{t+1}\left(x_{[t]}, [\xi_{[t-1]},\xi_t^i]\right)\right)-L(\xi_{[t-1]}) \sum_{j=1}^{N-1} \eta_{j}\\
&&\quad -\tilde{L}(\xi_{[t-1]}) \sum_{j=1}^{N-2}\left(\tau_{j}+\sigma_{j}\right)\left({y}_{j+2}-{y}_{j}\right)  -\sum_{j=2}^N \tilde{\beta}_j w_j - r_t(\xi_{[t-1]}) \varsigma \\
&\inmat{s.t.} & \sum_{j=1}^{N} {y}_{j} \mu_{i, j} \leq \mathbb{P}(\xi_t=\xi^i_t|\xi_{[t-1]})
 h_t(x_{t}, \xi^i_{t}),\ i=1,\ldots,S, \label{eq:dp-kan-con1}\\
&& \mathbb{P}(\xi_t=\xi^i_t|\xi_{[t-1]}) -\sum_{j=1}^{N} \mu_{i,j}=0,\ i=1,\ldots,S, \label{eq:dp-kan-con2}\\
&& \theta_{j-1} {y}_{j-1}-\theta_{j-1} {y}_{j}+v_{j-2}\left({y}_{j-1}-{y}_{j-2}\right)+
w_j+\eta_{j-1}+\tau_{j-1}-\tau_{j-2}
+\sigma_{j-2}-\sigma_{j-1} \geq 0,\\
&&\qquad \qquad j=3, \cdots, N-1, \\
&& \theta_{1} {y}_{1}-\theta_{1} {y}_{2}+
w_2+\eta_{1}+\tau_{1}-\sigma_{1} \geq 0\\
&&\theta_{N-1} {y}_{N-1}-\theta_{N-1} {y}_{N}+v_{N-2}\left({y}_{N-1}-{y}_{N-2}\right)+
w_N+\eta_{N-1}-\tau_{N-2}+\sigma_{N-2} \geq 0, \\
&&\theta_{j-1}-\theta_{j}+\sum_{i=1}^S \mu_{i,j}-v_{j-1}+v_{j}=0,\  j=2, \cdots, N-2 \\
&&\theta_{N-2}-\theta_{N-1}+\sum_{i=1}^{S} \mu_{i,N-1}-v_{N-2}=0, \\
&&  w_j\leq z_{j-1}(y_{j}- y_{j-1})+\frac{1}{2}(y_{j}- y_{j-1})^2 \varsigma ,\  j=2,\cdots,N, \\
&&-w_j\leq -z_{j-1}(y_{j}- y_{j-1})+\frac{1}{2}(y_{j}- y_{j-1})^2 \varsigma,\ j=2,\cdots,N, \\
&&  w_j\leq z_{j}(y_{j}- y_{j-1})+\frac{1}{2}(y_{j}- y_{j-1})^2 \varsigma ,\  j=2,\cdots,N, \\
&&-w_j\leq -z_{j}(y_{j}- y_{j-1})+\frac{1}{2}(y_{j}- y_{j-1})^2 \varsigma,\ j=2,\cdots,N, \\
&&x_{t} \in \mathscr{X}_{t}\left(x_{[t-1]}, \xi_{[t-1]}\right),\ \theta \in \mathbb{R}^{N-1}, v \in \mathbb{R}_+^{N-2}, \eta \in \mathbb{R}_+^{N-1}, \tau \in \mathbb{R}_+^{N-2}, \sigma \in \mathbb{R}_+^{N-2}, \\
&&\mu \in \mathbb{R}_+^{S \times  N},\varsigma\in \mathbb{R}_+, w\in \mathbb{R}_+^{N-1}, z\in \mathbb{R}_+^{N},\label{eq:dp-kan-con-n}
\end{eqnarray}
\end{subequations}
where the optimal value is $\widetilde{V}_{t}\left(x_{[t-1]}, \xi_{[t-1]}\right)$.
\end{theorem}

Theorem \ref{theorem-dp-sddp}
establishes 
a connection between the
optimal value functions at the adjacent stages
by solving an optimization problem.
When the optimal value function at period $t+1$ is concave, the reward function $h_t(\cdot,\xi_{t})$ is concave, the feasible set $\mathscr{X}_{t}\left(x_{[t-1]}, \xi_{[t-1]}\right)$ is compact and convex, the optimization problem \eqref{eq:reformu-gx-kantorovich} becomes a convex programming problem which can be solved efficiently by the interior point method.

{
\section{
Computational schemes
}
In this section, we discuss computational schemes for solving
the time consistent MS-PRO model \eqref{eq-mpro-tc} and the time inconsistent MS-PRO model (\ref{eq-pro-intc}). We proceed with two kinds of approaches: the scenario tree method and dynamic programming algorithms including SDDP/NBD
methods.
The scenario tree approach can be used to solve both \eqref{eq-mpro-tc} and (\ref{eq-pro-intc})
whereas  dynamic programming type algorithms
can only be applied to solve \eqref{eq-mpro-tc} on the basis of \eqref{eq:thm-recursive-formula}.



}

\subsection{ Scenario tree method}

Let $\Xi$
be a discrete support set
and $\{\xi\}_{t=1}^{T}$ a scenario tree.
Denote by $S$ the set of all nodes in the scenario tree,
$S^-$ the set of all non-leaf nodes,
and $S(t)$ the set of nodes at stage $t$.
Denote by $s^-$  the father node of $s$, $s^+$ the set of son nodes of $s$,
$\xi[s]$ the historical scenario from the root node to node $s$.
Denote by $t(s)$ the stage of node $s$, and by {$p_s\geq 0$} the appearing probability of node $s$.
Denote the decision at node $s$ by {$x(s):=\x(\xi[s])$}
and the historical decision from the root node to node $s$ by $x[s]$.
Notice that the decision at node $s$ (at stage $t(s)$)
is made according to the future realizations on the father node $s^-$ at stage $t(s)-1$.
Thus, the realization of the reward 
function at node $s$ is $h_{t(s)}\left(x(s^-),\xi(s)\right)$.
{
For the state-dependent problem,
the ambiguity set of the utility functions
upon historical samples
at node $s$
is denoted by $\mathcal{U}(s):=\mathcal{U}_{t(s)}(\xi[s])$.
For the state-independent problem, the ambiguity set of the utility functions at stage $t$
is $\mathcal{U}_t$, which is deterministic (independent of
historical samples).
}

\textbf{\underline{Time consistent (MS-PRO-SD)}.}
Problem \eqref{eq-mpro-tc} can be reformulated as the following min-max problem:
\begin{equation}\label{eq-mpro-scenario}
\begin{array}{cl}
\max\limits_{\{x(s),s\in S^-\}}  &
\sum\limits_{s\in S^-}
p_{s} \inf\limits_{u_{s} \in \mathcal{U}(s)}\left(\sum\limits_{i \in s^{+}} \frac{p_{i}}{p_{s}} u_{s}\left(h_{t(i)}\left({
x(s)}
, \xi(i)\right)\right)\right) \\
\text { s.t. }& x({1}) \in \mathscr{X}_{1}, x(s) \in \mathscr{X}_{t(s)}\left(x{[s^-]}, \xi{[s]}\right),\ \forall s\in S^-\setminus\{1\},
\end{array}
\end{equation}
where the inner minimization is to calculate the worst-case conditional expected utility value over the son nodes of $s$ across all scenarios
whereas the outer maximization is w.r.t. the optimal decision at node $s$.
{
The objective function is an average of all worst-case utility values at all non-leaf nodes of the tree.
We can reformulate it to indicate more clearly stages
and nodes at each stage:
$$\sum\limits_{s\in S^-}
p_{s} \inf\limits_{u_{s} \in \mathcal{U}(s)}\left(\sum\limits_{i \in s^{+}} \frac{p_{i}}{p_{s}} u_{s}\left(h_{t(i)}\left({
x(s)}
, \xi(i)\right)\right)\right)= \sum_{t=1}^{T-1}\sum\limits_{s\in S(t)}
p_{s} \inf\limits_{u_{s} \in \mathcal{U}(s)}\left(\sum\limits_{i \in s^{+}} \frac{p_{i}}{p_{s}} u_{s}\left(h_{t+1}\left({
x(s)}
, \xi(i)\right)\right)\right).$$

\textbf{\underline{Time inconsistent (MS-PRO-SID)}.}
Problem \eqref{eq-pro-intc} can be reformulated as the following min-max  problem:
\begin{equation}\label{eq-mpro-scenario-tic}
\begin{array}{cl}
\max\limits_{\{x(s),s\in S^-\}}  &
\sum\limits_{t=1}^{T-1}
\inf\limits_{u_{t} \in \mathcal{U}_t}
\left[
\sum\limits_{s\in S(t)} p_{s}
\left(\sum\limits_{i \in s^{+}} \frac{p_{i}}{p_{s}} u_{t}\left(h_{t+1}\left({
x(s)}
, \xi(i)\right)\right)\right)
\right] \\
\text { s.t. }& x({1}) \in \mathscr{X}_{1}, x(s) \in \mathscr{X}_{t(s)}\left(x{[s^-]}, \xi{[s]}\right),\ \forall s\in S(t),\ t=2,\ldots,T-1.
\end{array}
\end{equation}
In both \eqref{eq-mpro-scenario} and \eqref{eq-mpro-scenario-tic}, the decisions $x(s)$ are node-dependent. The only difference is that, in \eqref{eq-mpro-scenario}, the ambiguity sets are node-wise 
and we find
the worst-case utility at each node;
whereas in \eqref{eq-mpro-scenario-tic},
the ambiguity sets are
stage-wise 
and we find the worst-case utility
for {\em all} nodes at each stage.
Further detailed
reformulations
depend on the
structure
of the
scenario tree and thr stage-wise
ambiguity set $\mathcal{U}(s)$/$\mathcal{U}_t$.
We refer readers to Appendix \ref{sec-app-3-tree} for details.
}

\subsection{Dynamic programming methods}

{
The reformulation of
\eqref{eq-mpro-tc}
as \eqref{eq:thm-recursive-formula}
paves the way for us to apply the NBD 
and SDDP methods for solving
the problem. The basic idea of the DP-type
algorithms is to
develop
an approximation of the cost-to-go function
$
{V}_{t}(x_{[t-1]},\xi_{[t-1]})$,
use the optimal solution based on the approximate
problem
as an approximate optimal solution of
\eqref{eq:thm-recursive-formula}
(and ultimately \eqref{eq-mpro-tc})
and improve the approximations over an iterative
forward and backward process.
To this end, we need to make the following standard assumption.

\begin{assumption}\label{assum-linear-recourse}
Denote $\xi_t=(c_t,W_t,b_t,D_t)$.
(a) The decision $x_t$ is $\xi_{[t-1]}$-dependent,
(b) the constraints at recourse stages in the MS-PRO problem have a linear block-diagonal  structure, i.e., only consecutive stages can be linked by linear constraints, i.e., $\mathscr{X}_t=\{ x_t\mid W_{t-1}(\xi_{[t-1]}) x_t=b_{t-1}(\xi_{[t-1]})-D_{t-1}(\xi_{[t-1]}) x_{t-1}\}$, $W_{t-1}$ is invertible or fixed,
$t=2,\ldots,T$,
(c) the reward functions are linear, i.e., $h_t(\xi_t,x_t)=c_t(\xi_t)^\top x_t$.
\end{assumption}

}



{


Assumption \ref{assum-linear-recourse}
ensures  concavity of
$V_t(x_{[t-1]},\xi_{[t-1]})$ (see \eqref{eq:thm-recursive-formula}) in $x_{[t-1]}$ and $\xi_{[t-1]}$ for $t=T,\cdots,2$.
This enables us to
construct piecewise linear approximations of $V_t(x_{[t-1]},\xi_{[t-1]})$,
which underlies NBD
algorithm and SDDP algorithm,
and guarantees the
strong duality of the inner minimization problem of \eqref{eq:thm-recursive-formula} and thus the final convergence of the
algorithms.
Here we give a sketch of the algorithmic structure and refer readers to \ref{sec-dy-algo} for details.

\vspace{2mm}
\begin{breakablealgorithm}\label{algo-1}
    \caption{\small Outline of NBD/SDDP
    algorithms }
\begin{flushleft}
\textbf{Input:} 
A finite set of scenarios $\mathcal{K}$ \\
\textbf{while} $i<N_{\max}$ \textbf{do}
\end{flushleft}
\begin{itemize}
    \item \textbf{for} $k\in\mathcal{K}$, $t=1,\ldots,T$ \textbf{do}
    (forward pass)
    \begin{itemize}
        \item
        \textbf{solve} \eqref{eq:thm-recursive-formula} with current piecewise linear approximation
        of $V_{t}$, denoted by $V^{i}_{t}$,
        and
        trial decision
        $x_{t-1}^{k,i}$
        at stage $t-1$ to obtain
        trial decision
        $x_{t}^{k,i}$
        at stage $t$. Calculate a lower bound of the optimal value.
    \end{itemize}
        \item \textbf{for} $k\in\mathcal{K}$, $t=T,\ldots,1$ do (backward pass)
  \begin{itemize}
        \item
        solve \eqref{eq:thm-recursive-formula} with updated $V^{i+1}_{t}$  and
        trial
        decision $x_{t-1}^{k,i}$ to obtain the optimal value of dual variables.
        \item update $V^{i}_{t}$ by adding
        a cut
        constructed with the optimal values of the dual variables. Calculate an upper bound of the optimal value.
    \end{itemize}
    \item \textbf{terminate} when the gap between the upper and lower bounds
    falls within the prescribed precision.
\end{itemize}
\end{breakablealgorithm}
\vspace{2mm}

There are two ways to proceed. One is
to use a large scenario tree of the multistage decision making process in the sample space
and then find historical path-dependent optimal solutions
by solving
the optimization problem \eqref{eq-mpro-tc-nomial-dy-PLA} in Theorem \ref{theorem-dp-sddp}
embedded into each node on the large scenario tree.
This is known as the
NBD
algorithm.
The other is
to take some i.i.d. samples from all scenarios in the finite-support case or the continuous distribution in the infinite-support case
in solving \eqref{eq-mpro-tc-nomial-dy-PLA},
which is known as SDDP algorithm.
Here we adopt both and compare them with the scenario tree algorithm. We will report  comparative results in the next section.

Convergence of the two algorithms are guaranteed
under
some standard 
conditions. For instance,
when the MS-PRO problem has relatively complete recourse and the distribution of the process $\{\xi_t\}$ is known,
we can show that the
NBD 
algorithm converges to an optimal solution of MS-PRO-SD in finitely many iterations following a similar analysis to that of \cite{BiL11,FuR21}.
If, in addition,
$\xi_t$ is independent of the history $\xi_{[t-1]}$ of the process,
then we may follow \cite{Sha11,FuR21} to show that
the SDDP algorithm converges with probability $1$ to an optimal policy
of MS-PRO-SD in a finite number of iterations.
We skip the details as these are not the main focus of this paper.

}






\section{Numerical tests}

{To examine the performance of the proposed multistage MS-PRO-SD model (\ref{eq-mpro-tc}) and MS-PRO-SID model \eqref{eq-msp-indep-comp}, as well as numerical schemes,}
we carry out a number of 
numerical tests on
a multistage investment-consumption problem on the basis of \cite{Duf10,Fam70} with state-dependent utility functions.

\subsection{An investment-consumption problem}

Consider an investor who plans to use her/his wealth to purchase crude oil and make oil products over $T$ periods. At the beginning of each time period, the investor has two  options: (a) consume all of the wealth for the purchase, and (b) consume part of it and invest the remaining wealth in  $n$ risky assets of a  security market. The objective of the investor is to maximize the
overall expected utility of the oil products consumption.

Let $w_0=1$ denote the normalized initial wealth and
$q_t$ denote the quantity of crude oil that the investor plans to buy at beginning of time period $t$ at price $p_{t-1}$ which is the oil price at the end of time period $t-1$ (alternatively, at the beginning of period $t$). The total cost from the purchase is
$q_tp_{t-1}$ and the remaining wealth is $w_{t-1}-q_t p_{t-1}$,
where $w_{t-1}$ is the wealth at the end of period $t-1$. The remaining wealth is invested in $n$ risky assets with a portfolio $x_t$, where $x_t^i$ is the wealth invested in the $i$-th asset, $i=1,\ldots,n$,
whose 
random return rate, denoted by $r_t^i$, is calculated period-wise, i.e.,
a $\$ 1$ investment at the beginning of period $t$
will generate $\$ (1+r_t^i)$ at the end of the period.
Thus, the wealth of the investor at the end of period $t-1$ is
$w_{t-1}=(e+r_{t-1})^{\top} x_{t-1}$.
This wealth is divided into the consumption $q_t p_{t-1} $ and the further investment $e^{\top} x_{t}$, i.e.,
$w_{t-1}=q_t p_{t-1} + e^{\top} x_{t}$.
A combination of the two equations gives rise to
the following wealth balance equation
$$
e^{\top} x_{t}=(e+r_{t-1})^{\top} x_{t-1} - q_t p_{t-1},\ t=2,\ldots,T-1.
$$
At the initial period $t=1$, we have
$
e^\top x_1 =w_0 - q_1 p_0
$
and at the final period $T$, the investor
must consume all of the wealth on purchase of oil, thus
$
(e+r_{T-1})^{\top} x_{T-1} = q_T p_{T-1}.
$

The utility of the oil products is calculated at the end of each period as follows. We assume that all of the $q_t$ barrels of
oil purchased at the beginning of period $t$ is used to produce
$g_t(q_t)$  quantities of the oil products by the end of period $t$ with unit value $d_t$. Thus the total 
value from the production is  $g_t(q_t) d_t$ and the period-wise utility value is
$\mathfrak{u}_t( g_t(q_t) d_t,h_{[t-1]})$.
Here the investor's utility function depends on all the historical information $h_{[t-1]}$ $=\{p_0,\dots,p_{t-1},$ $d_1,\dots,$ $d_{t-1},r_1,\dots,$ $r_{t-1}\}$.  Based on the discussions above, we
formulate the multistage investment-consumption problem as
\begin{subequations}\label{ex-oil-msp}
\begin{eqnarray}
& \max\limits_{x_{[1,T\!-\! 1]},q_{[1,T]}} & \mathbb{E}\left[\mathfrak{u}_1(g_1(q_1)d_1,h_0)+\mathfrak{u}_2(g_2(q_2)d_2,h_{[1]})+\cdots+\mathfrak{u}_T(g_T(q_T)d_T,h_{[T-1]})\right]  \label{ex-oil-msp-1}\\
& {\rm s.t.} &
e^\top x_1 =w_0 - q_1 p_0,\ x_1\in \mathbb{R}^n_+,\ q_1\in\mathbb{R}_+, \label{ex-oil-msp-2}\\
& & e^{\top} x_{t}=(e\!+\! r_{t-1})^{\top} x_{t-1}-q_t p_{t-1}, x_t(\cdot)\! \in \mathbb{R}^n_+, q_t(\cdot)\! \in\mathbb{R}_+, t\! =\! 2,\ldots,T\!-\! 1, \label{ex-oil-msp-3}\\
& & (e+r_{T-1})^{\top} x_{T-1} = q_T p_{T-1},\ q_T(\cdot)\in\mathbb{R}_+. \label{ex-oil-msp-4}
\end{eqnarray}
\end{subequations}
In the setup, we assume that short sales of
the security assets and crude oil are forbidden, i.e.,
$x_t\in \mathbb{R}_+^n$ and $q_t\in\mathbb{R}_+$.
Assume that the investor is ambiguous about the true utility function at each stage, we then propose
a preference robust counterpart of the multistage 
investment-consumption problem to mitigate the risk arising from the ambiguity:
\begin{subequations}\label{ex-oil-pro}
\begin{eqnarray}
   & \max\limits_{x_{[1,T\!-\! 1]},q_{[1,T]}} & \inf_{\vec{\mathfrak{u}}\in\mathcal{U}}\mathbb{E}\left[\mathfrak{u}_1(g_1(q_1)d_1, h_0)+\mathfrak{u}_2(g_2(q_2) d_2,h_{[1]})+\! \cdots\! +\mathfrak{u}_T(g_T(q_T) d_T,h_{[T-1]})\right]  \label{ex-oil-pro-1} \\
& {\rm s.t.} &  \eqref{ex-oil-msp-2}-\eqref{ex-oil-msp-4}.  \label{ex-oil-pro-2}
\end{eqnarray}
\end{subequations}
{ We  carry out comparative numerical analysis
on the model by considering the utility functions being state-dependent (with the ambiguity set
being constructed via pairwise comparison and Kantorovich ball)  and state-independent, respectively.
}


\subsection{Setup of tests}
\label{sec:num-setup}

To ease the exposition,
we consider a simple case
where $g_t(x)=x$ and $d_t=p_t$, for $t=1,\ldots,T$. This is based on the understanding that
the productions of oil products are proportional to the purchased amount of the crude oil
and the value of oil products is proportional to the crude oil price.
We assume that the
true utility of oil products depends on the crude oil price in two regimes. In the usual regime when the crude oil price is less than or equal to $\$60$ per barrel, the investor has a linear utility function $\mathfrak{u}_{\inmat{lin}}(x)=x$ defined over $[0,1]$.
In the other regime when the crude oil price is greater than $\$60$ per barrel,
the investor has
a concave utility $\mathfrak{u}_{\inmat{exp}}(x)=(1-\exp(-3x))/(1-\exp(-3))$ defined over $[0,1]$.


We set the risky assets pool with 9 exchange-traded-funds (ETF)
in the US equity market corresponding to different industry sectors {
including Utilities (XLU), Energy (XLE), Finance 
(XLF),
Technology (XLK),
Health Care (XLV),
Consumer Staples (XLP),
Consumer Discretionary (XLY),
Industry 
(XLI), and
Materials (XLB) sectors.}
We collect  weekly data of crude oil price (OK Crude Oil Future Contract) and the ETF prices over the period 2007/1/1 - 2021/3/29. ETF data are downloaded from Yahoo Finance\footnote{\url{https://finance.yahoo.com}} and oil prices are downloaded from Energy Information Administration\footnote{\url{https://www.eia.gov/dnav/pet/hist/LeafHandler.ashx?n=PET&s=RCLC1&f=W}}.
Before generating the scenario tree, the price data are transformed into log-return rate to pass the stationary test of the data 
series. We adopt an {ARMA(0,1)-GARCH(1,1) model with Gaussian residuals to forecast the future return rate of oil and ETF prices and built a scenario tree with a symmetrical branching structure. The optimal orders for the ARMA and GARCH models were determined through maximum likelihood estimation.}
One can refer to \cite{YCC20} for detailed algorithms of the scenario tree generation. {To reduce the computational complexity of DP-type algorithms, we consider the stage independent case. }

The models to be tested in comparative 
analysis include:
 MSP-True:
 problem \eqref{ex-oil-msp} with the true utility functions,
 SP-PLN-{SD}:
 problem \eqref{ex-oil-msp} with piecewise linear nominal utility functions,
 MS-PRO-{SD}-Kan: 
 problem \eqref{ex-oil-pro} with
  the {state-dependent} ambiguity set $\mathcal{U}^K_t(\xi_{[t-1]}) $ constructed via the Kantorovich ball
  centered at a piecewise linear nominal utility function at each node.
 MS-PRO-{SD}-PC:
 problem \eqref{ex-oil-pro} based on the {state-dependent}
  pairwise comparison ambiguity set $\mathcal{U}^P_t(\xi_{[t-1]}) $ with
  randomly generated questionnaires and answers at each node.
  {
   MS-PRO-{SID}-Kan: 
problem \eqref{ex-oil-pro} with
  the {state-independent} ambiguity set $\mathcal{U}^K_t$ constructed via the Kantorovich ball
  centered at a piecewise linear nominal utility function at each stage.
Details of preference elicitation and
construction of the ambiguity sets are
deferred to \ref{EC:Elicit-Ambiguity-consct}.
  }
All optimization problems in the deterministic reformulations are solved
by Gurobi solver through CVX package in Matlab R2016a on a PC with 3.4GHz CPU and 16GB RAM.

{
\subsection{Numerical results: validation of three solution approaches}

In the first set of tests, we solve MS-PRO-SD-Kan with the scenario tree method, the NBD 
method, and the SDDP method for small instance problems with 2-6 stages.
In order to compare the three solution methods in a same problem,
we focus on
a 
scenario tree with stagewise independence (corresponding to a 
 recombining tree \cite{FuR21,SDR14}) 
and state-dependent utilities.
At each stage, we generate 5 samples of the oil price and return rates of the 9 ETF assets.
For the scenario tree method and the NBD method, we generate a 
tree with $5^{T}$ scenarios with the stagewise independent samples, where $T=2,\cdots,6$.
Table \ref{tab:tree_approach} displays the
optimal values and
CPU times of the three approaches.
From the table, we can see that
the scenario tree method and the NBD method generate the same optimal values
when
$T=2,\cdots,5$.
In the case that
$T=6$, the
lower and the upper bounds
generated by the
NBD method
do not match in the last two digits within
the specified algorithmic stopping criteria.
The SDDP method generates slightly
wider gaps between the lower bounds and upper bounds for $T=2,\cdots,6$ where the lower bounds are heuristic.
In terms of CPU time, the scenario tree method
is very efficient when $T\leq 4$ but its CPU time increases
rapidly when $T=5,6$ because the number of scenarios increases exponentially.
In contrast, the SDDP method displays a kind of
``linear''
increase of CPU time w.r.t. 
$T$.
The NBD 
method displays the longest
CPU time in all five cases ($T=2,\cdots,6$).

\begin{table}[h]
{
    \caption{{The Optimal values
    and CPU times of the scenario tree method, the NBD method and the SDDP method for MS-PRO-SD-Kan with 2-6 stages}}
    \label{tab:tree_approach}
    \centering
    \begin{tabular}{c|c|ccccccccc}
    \hline
   \multicolumn{2}{c|}{
   $T$
   }    &  2&3&4&5&6 \\
      \hline
     \multirow{2}*{Scenario tree method}  &  Opt. Val. & 1.1977 & 1.2281 & 1.3607 & 1.6832 & 1.9859 
     \\
     & CPU time (s) & 2.6905 & 13.2355&136.36&716.96&3738.8 
     \\\hline
 {NBD method} & Opt. Val. & 1.1977 & 1.2281 & 1.3607 & 1.6832 & [1.9808,1.9888] 
 \\
(tol=0.0001) & CPU time (s) & 46.128 & 224.34 & 780.1&  2411.4 &5519.4 
\\\hline
\multirow{2}*
{SDDP method}  &  Upper bound & 1.1992 & 1.2283 & 1.3618 & 1.6846 & 1.9868 
\\
& Lower bound & 1.1970 & 1.2287 & 1.3620 & 1.6838 & 1.9858 
\\
    (tol=$z_{1-0.99/2} \frac{\sigma_{\underline{v}}}{|\mathcal{K}|}$, $|\mathcal{K}|=20$) & CPU time (s) &  28.15 & 59.90 & 99.21 & 367.94 & 685.52 & 
    \\
     \hline
    \end{tabular}
}
\end{table}

In the second set of tests, we solve the same problem but for the case when $T$ ranges
from $10$ to $50$. Since the scenario tree method and the NBD method require unaffordable storage space and
unacceptably long CPU time, we concentrate on SDDP only.
Table \ref{tab:sddp} lists lower and upper bounds of the optimal values, the number of
iterations of forward-backward processes,
and CPU time
for $T=10, 30, 50$ with different numbers of scenarios ($|\mathscr{K}|=5, 10, 20$).
We can see that the lower bounds are close to upper bounds in all of the cases which means
the algorithm converges within the prescribed precision. The change of CPU time confirms our earlier observation that it increases at a ``linear'' rate w.r.t. the increase of $T$.



\begin{table}[h]
\centering
{
    \caption{{Upper bounds, lower bounds and CPU time of the SDDP method for MS-PRO-SD-Kan with 10-50 stages (tol=$z_{1-0.99/2} \frac{\sigma_{\underline{v}}}{ |\mathcal{K}|}$)}}\label{tab:sddp}
\begin{tabular}{ccccccc}
\hline
Stages & $|\mathcal{K}|$ & Iterations & Upper bound & Lower bound & CPU time (s) \\\hline
       & 5         & 13         & 2.4473      & 2.4513      & 333.44   \\
10     & 10        & 11         & 2.5307      & 2.5297      & 551.16    \\
       & 20        & 13         & 2.3907      & 2.3921      & $1.44\times 10^3$ \\\hline
       & 5         & 44         & 2.9791      & 2.9718      & $4.31\times 10^3$  \\
30     & 10        & 31         & 2.9207      & 2.9186      & $6.57\times 10^3$  \\
       & 20        & 13         & 2.915       & 2.9161      & $5.01\times 10^3$  \\\hline
       & 5         & 57         & 3.4359      & 3.4462      & $9.95\times 10^3$  \\
50     & 10        & 41         & 3.4213      & 3.4194      & $1.61\times 10^4$  \\
       & 20        & 23         & 3.3919      & 3.394       & $1.93\times 10^4$  \\
\hline
\end{tabular}
}
\end{table}

}



\subsection{Comparative analysis of the models}
\label{sec:num-model-val}
{
To examine the effects of different models,}
we have conducted
comparative
numerical analysis
from the following {four} perspectives:
   {
   (a) Compare the optimal values of  MSP-True, MS-PRO-SD-Kan and MS-PRO-SID-Kan}
   with respect to different numbers of time periods, $T=2,3,\ldots,6$.
   We set the radius of the Kantorovich ball to $R=0.001$ and the number of breakpoints to $N=40$ under all scenarios, see  { Figure \ref{fig-1}.}
  {
  (b) Compare} the optimal values of MSP-PLN and MS-PRO-SD-Kan with different numbers of breakpoints $N$, where the optimal value of MSP-True is chosen as the benchmark. Here, we set $T=4$, and $R=0.001$ under all scenarios,
  {see Figure \ref{fig-2}.}
{
(c) Compare the optimal values
of  MS-PRO-SD-Kan and MS-PRO-SID-Kan} with different radii of Kantorovich ball when $T=4$, and $N=40$, { see Figure \ref{fig-3}}.
{
(d) Compare} the optimal values of MS-PRO-PC with different number of questionnaires with $T=4$, see Figure \ref{fig-4}.


\begin{figure}[htbp]
	\centering
	\begin{minipage}[t]{0.44\textwidth}
		\centering
\includegraphics[width=7.5cm]{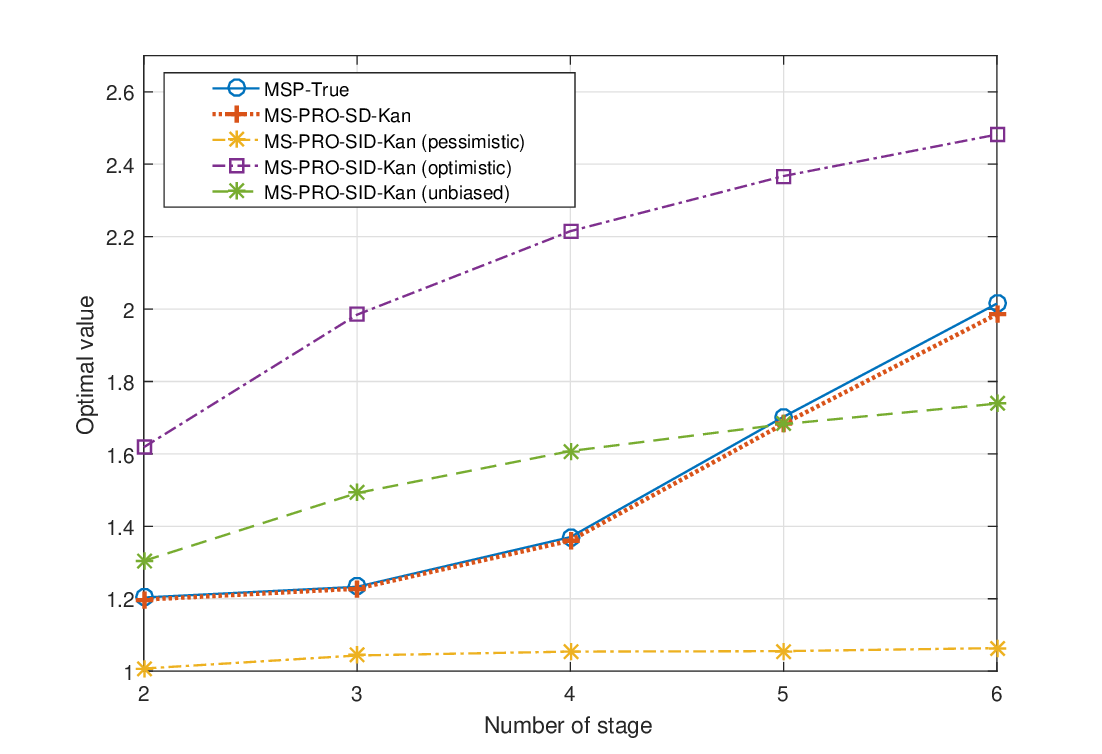}
		\caption{{
  Comparison of the optimal values
  of MSP-True, MS-PRO-SD-Kan and  MS-PRO-SID-Kan models
  with increasing number of stages
		($R=0.001$, $N=40$).} \label{fig-1}}
	\end{minipage}
	\quad
		\begin{minipage}[t]{0.44\textwidth}
		\centering
    \includegraphics[width=7.5cm]{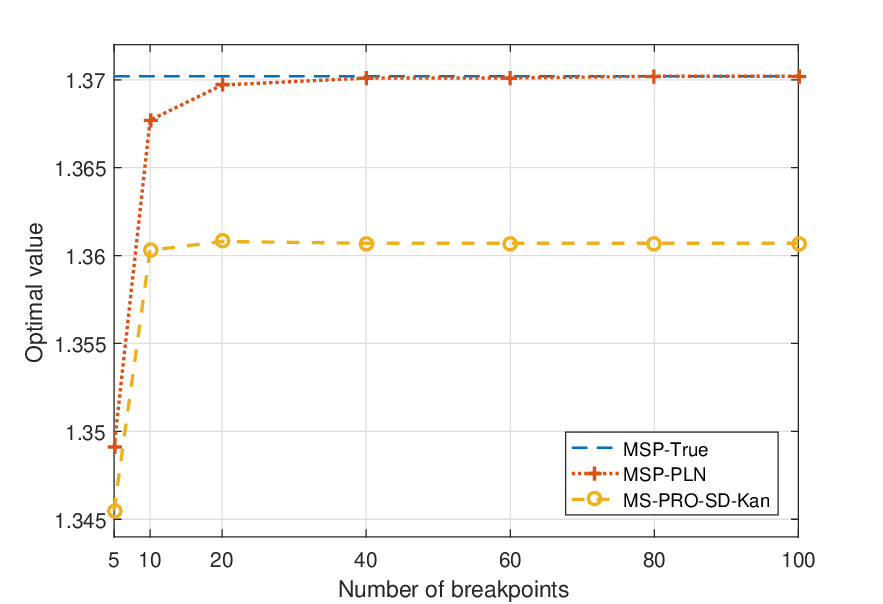}
\caption{Comparison of the optimal values of MSP-True, MSP-PLN and MS-PRO-SD-Kan with increasing number of breakpoints
($T=4$, $R=0.001$).\label{fig-2}}
	\end{minipage}
\end{figure}

\begin{figure}[htbp]
	\centering
	\begin{minipage}[t]{0.44\textwidth}
		\centering
		\includegraphics[width=7.5cm]{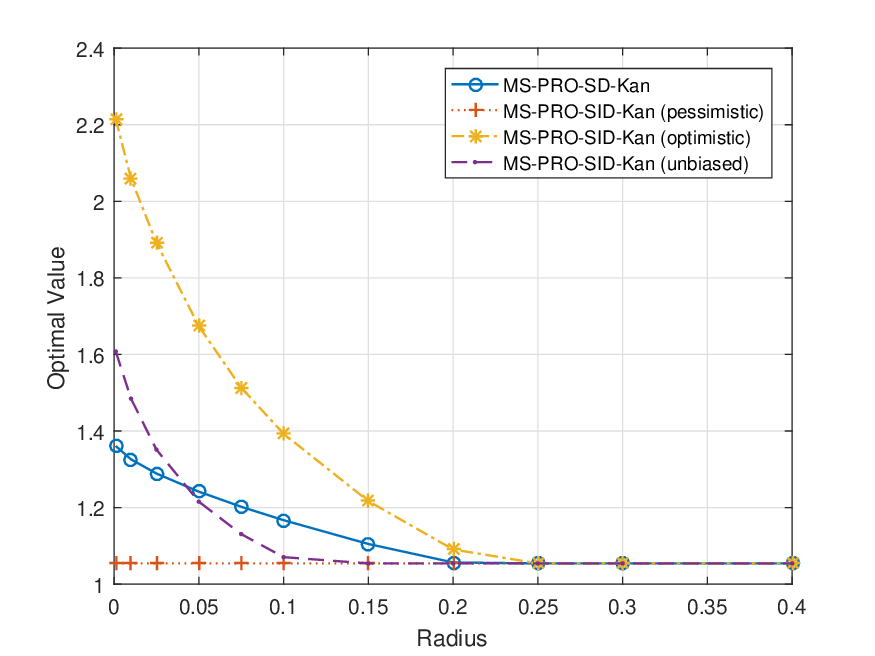}
		\caption{{
  Comparison of the optimal values
  of MS-PRO-SD-Kan and MS-PRO-SID-Kan with increasing radius of Kantorovich Ball
		($T=4$, $N=40$).} \label{fig-3}}
	\end{minipage}
	\quad	\begin{minipage}[t]{0.44\textwidth}
		\centering
   \includegraphics[width=7.5cm]{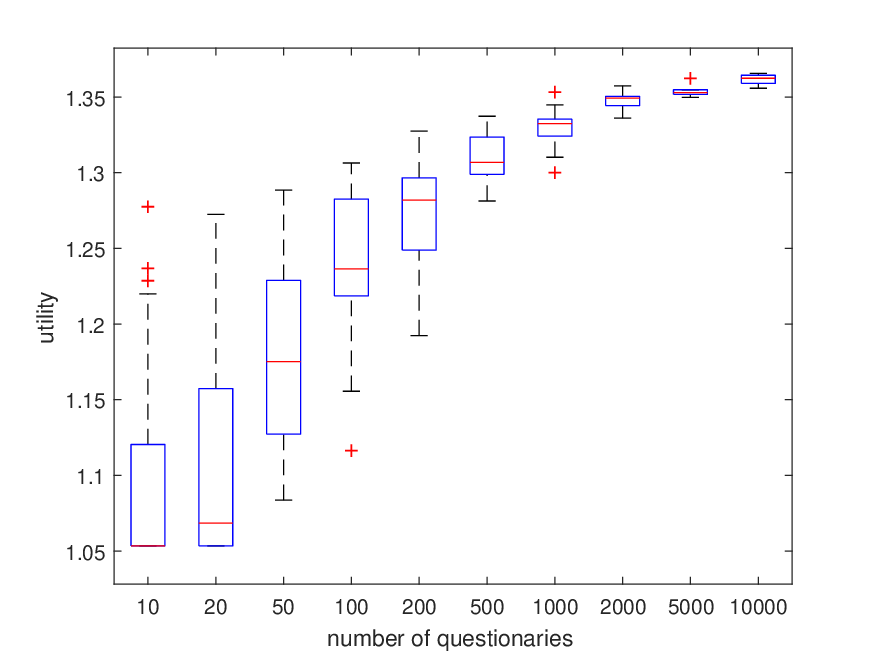}
		\caption{Boxplot of the optimal values of MS-PRO-SD-PC with different numbers of questionaries (for each number, we randomly generate 100 groups of questionaries and plot the mean, maximum, minimum and $25-75\%$ quantiles
		($T=4$, $N=40$).
  \label{fig-4}}
	\end{minipage}
\end{figure}


{
    From Figure \ref{fig-1},
    we can see that with
    more stages ($T$) to be included in the models,
    the investor has
    greater flexibility in
setting future consumption
and
consequently
obtaining higher optimal total expected utility values.
This phenomenon
    is observed
    for both MSP and MS-PRO models.
    Moreover, the optimal value
    of the MS-PRO model is smaller than that of MSP, which can be
    interpreted as the price of robustness.
    The gap narrows down
    as $T$ increases.
    The optimal values of MS-PRO-SID-Kan with optimistic estimations are the highest while that with pessimistic estimations are the lowest 
    and that with unbiased estimations are in the middle. This relationship
    reflects
    the nature of the estimations.
    Figure \ref{fig-2}
    depicts the variation trends of 
    the optimal values as the number of the breakpoints ($N$) of
    piecewise linear approximation
        increases.
        We can see that
        the optimal value of MSP-PLN 
        approaches
        that of MSP-True
        when $N$ reaches $20$;
        and the optimal value of MS-PRO-SD-Kan
    moves closer to that of MSP-True despite a gap exists due to $R>0$.
    This is consistent with our theoretical results.

    Figure \ref{fig-3} presents
    comparative analysis between
    state-dependent
    utility model and state-independent
    utility model under the framework of
 MS-PRO-Kan.
        We can
    observe that
    {with the decrease of the
    radius, }
    the optimal value of
    MS-PRO-SD-Kan
    approaches
    that of {MSP-PLN}  
    with the same piecewise linear nominal utility function.
    When the radius is
    greater than 0.25,
        MS-PRO-SD-Kan and MS-PRO-SID-Kan 
    generate almost the same solution.
    This is because
    the constraint
corresponding to the Kantorovich ball
becomes inactive 
(the worst-case utility function becomes linear
and the concavity constraint overrides the ball constraint, see Figure \ref{ECfig:worst-case-utility})
and
subsequently
only the bounds on Lipschitz modulus 
and the
    convexity constraints
    are effective.
    }
    Figure \ref{fig-4}
    shows that with the increase of the number of questionnaires, the optimal value of MS-PRO-SD-PC converges to that of MSP-True
    since the randomness of questionnaires recedes.

{
\subsection{Out-of-sample performance of different models with randomly generated true utilities}

We now turn to report our numerical test results on
the
out-of-sample performance of the proposed
MS-PRO models. Specifically, we
solve the MS-PRO models including
MS-PRO-SD-Kan 
and MS-PRO-SID-Kan, 
obtain an optimal
solution, and implement it in the out-of-sample tests
with the true utility function.
We begin by randomly generating
a set of non-decreasing, piecewise linear and concave utility functions which are within
an $\epsilon$-Kantorovich ball centered at
a state-dependent reference utility function,
see Figures \ref{fig-random-u-001}-
\ref{fig-random-u-02}.

The first set of tests is carried out as follows.
For the  MS-PRO-SID-Kan model, we use each of the four
estimation approaches outlined
in Section 6.3
to figure out
a  state-independent nominal utility function,
construct  respective Kantorovich balls with
three different radii ($R=0.01$, $0.1$ or $0.2$),
and solve the
resulting MS-PRO-SID-Kan models.
For the MS-PRO-SD-Kan model, we use the unbiased estimation method
to find a piecewise linear
nominal utility function at each state
and then construct a Kantorovich ball
 (with different radii $R=0.01,0.1,0.2$).
Here we assume that the number of states
is known 
but the
correspondence
between the elicited scores and the states is unknown.
For each of the optimal solutions,
 we calculate the returns in each scenario and then
evaluate the out-of-sample expected utility value 
with
one of the randomly generated
 utility  functions (we call one simulation).
 We 
 repeat the simulation 100 
 times and
 calculate the average of the expected utility values.
The rationale behind the simulations is that the true utility function is unknown and we presume that
each
of the 100 utility functions could be the true.
Table 
\ref{tab:out_of_sample_01}
displays
the average of the mean values, the
minimum value and the maximum value
of the 100 out-of-sample tests.

From Table~\ref{tab:out_of_sample_01},
we can see that
MSP-True performs best in terms of the mean value and the maximum value.
MS-PRO-SD-Kan ($R=0.01$)
gives the best of
the worst-case expected utility value, which highlights
the value of adopting the robust model.
The  MS-PRO-SID-Kan delivers the worst performance
in all aspects.
This is primarily because the true utility function
is state-dependent.
Figure \ref{compare_utility}
depicts the results in box-plots,
we can see that when $\epsilon$ increases, the difference of the performances
in terms of minimum values becomes smaller.
This is because the concavity constraint
overrides the Kantorovich ball
constraint (the worst-case utility function becomes linear).
Moreover, when $R$ matches $\epsilon$ ($0.1$),
the MS-PRO-SD-Kan  performs best in terms of the minimum value.
In this case,
the ambiguity set
in MS-PRO-SD-Kan
covers the set of randomly generated utility functions (for out-of-sample tests).

}


\begin{table}[]
\small
    \centering
    {
\caption{{
        Comparisons of out-of-sample performances
         of MS-PRO-SD-Kan and MS-PRO-SID-Kan
          with $T=4$, $N=40$ and $\epsilon=0.1$.
        }
        }
    \label{tab:out_of_sample_01}
    \begin{tabular}{c|c|ccc|cccc}
    \hline
  \multirow{2}*{\diagbox{Statistics}{Model}}  & \multirow{2}*{MSP-True} & \multicolumn{3}{c|}{MS-PRO-SD-Kan (Unbiased)
  }  & \multicolumn{4}{c}{MS-PRO-SID-Kan ($R=0.1$)} \\
& & $R=0.01$ & $R=0.1$ & $R=0.2$ & Pessimistic & Optimistic & Unbiased & Best-fit \\
\hline
 Mean  & \textbf{1.3660} & 1.3654 & 1.3587 & 1.1542 & 1.0860 & 1.1467 & 1.1418 & 1.1480\\
Min & 1.1660 & 1.1663 & \textbf{1.1670} & 1.0962 & 1.0791 & 1.0899 & 1.0996 & 1.1003 \\
Max & \textbf{1.5660} & 1.5645 & 1.5501 & 1.2113 & 1.0945 & 1.2028 & 1.1834 & 1.1946\\
\hline
    \end{tabular}
    }
\end{table}


\begin{figure}[htbp]
	\centering
\begin{minipage}[t]{0.32\textwidth}
		\centering
   \includegraphics[width=5.5cm]{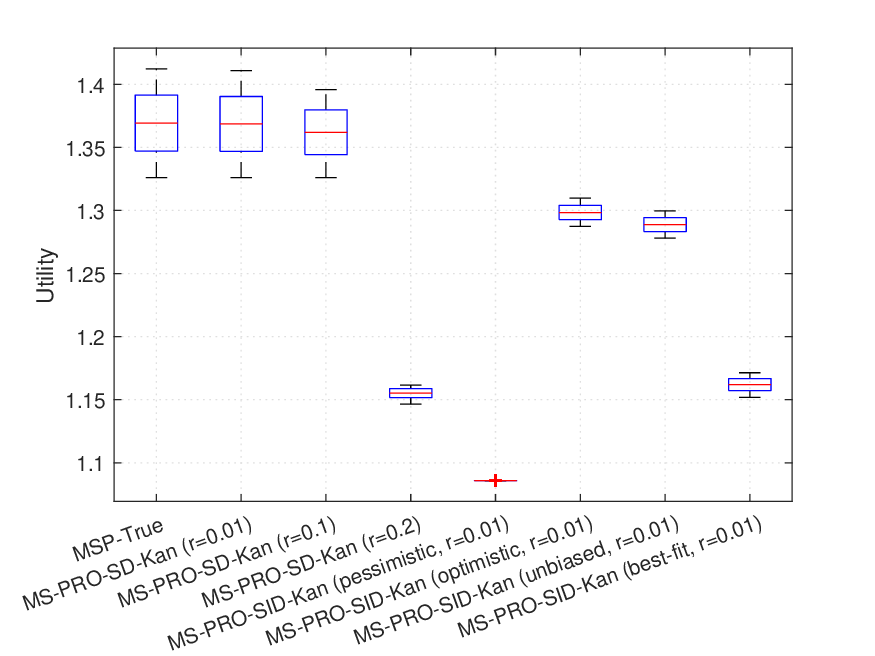}\\
  { $\epsilon=0.01$}
	\end{minipage}
 \begin{minipage}[t]{0.32\textwidth}
		\centering
   \includegraphics[width=5.5cm]{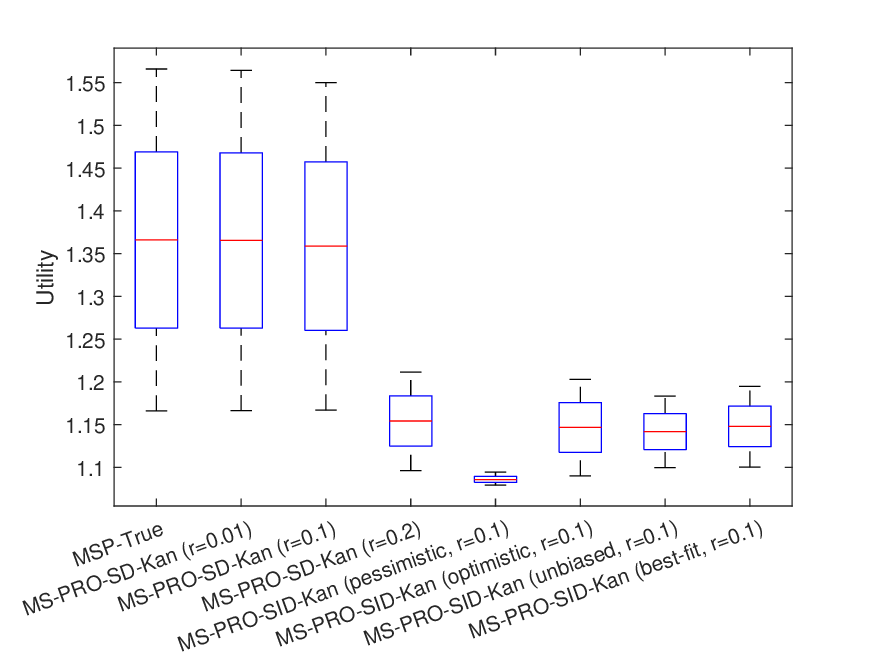}\\
	 	{ $\epsilon=0.1$}
	\end{minipage}
 \begin{minipage}[t]{0.32\textwidth}
		\centering
   \includegraphics[width=5.5cm]{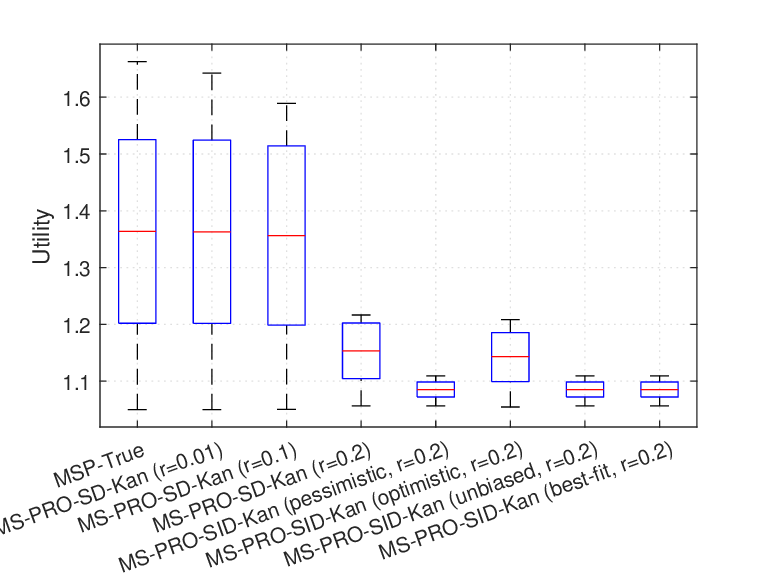}\\
   { $\epsilon=0.2$}
	\end{minipage}
 	\caption{
  Boxplots of out-of-sample utility values 
         of MS-PRO-SD-Kan and MS-PRO-SID-Kan under three sets of %
  randomly generated utility functions
  \label{compare_utility} }
\end{figure}

\section{Concluding remarks}

{
In this paper, we present a full investigation
of the PRO models for expected utility based
multistage  decision making.
We begin with holistic maximin models
(\ref{eq-mpro-tc}) for state-dependent utility case and
(\ref{eq-pro-intc}) for state-independent utility case,
demonstrate time consistency and time inconsistency for
them respectively, and derive the
dynamic recursive formulation (\ref{eq:thm-recursive-formula}) for the former.
We then use scenario-tree methods to solve both
(\ref{eq-mpro-tc}) and (\ref{eq-pro-intc})
with a given scenario tree structure of the underlying random process,
and the SDDP and the NBD 
methods to solve
(\ref{eq-mpro-tc}) via (\ref{eq:thm-recursive-formula}).
Finally, we
carry out comparative numerical tests
on state-dependent
model (\ref{eq-mpro-tc})
vs state-independent model
(\ref{eq-pro-intc}), and
scenario tree method
vs dynamic programming method for solving (\ref{eq-mpro-tc}).
To derive dynamic reformulation of (\ref{eq-mpro-tc}), we
derive a new  version of
the principle of interchangeability
in Banach space (Lemmas~\ref{lemma-interchange-expectation} and \ref{lem-int-change}).

A clear benefit of beginning the robust model
with  (\ref{eq-mpro-tc}) rather than (\ref{eq:thm-recursive-formula})
as some of
the distributionally robust MSP models do in the literature (see e.g.~\cite{PaB22,YuS22}) is that it allows us to apply both the
scenario tree
algorithm and
DP
algorithms for solving
the state-dependent MS-PRO model.
Moreover, since
(\ref{eq-pro-intc}) does not have a dynamic
reformulation, the presence of
(\ref{eq-mpro-tc}) facilitates us to
compare the performances of the two models
by solving them with the same scenario tree methods.
Establishing
a link between (\ref{eq-mpro-tc}) and (\ref{eq:thm-recursive-formula}) is
a key step given that our PRO model is non-parametric and
establishing an equivalence relation
requires
a new interchangebility result in Banach space.
While
our focus in the paper is
on the utility-based PRO models,
our approach on both (\ref{eq-mpro-tc}) and (\ref{eq:thm-recursive-formula}) may have
some ramifications on other nonparametric
 multistage maximin (res.~minimax)  optimization problems (\cite{XiG21,SaX20}).
To the best of our knowledge,
 the 
 existing research
only allows one to
establish
an equivalence relation
(analogous to (\ref{eq:thm-recursive-formula}) and (\ref{eq-mpro-tc}))
 for the multistage
 parametric robust optimization problems
where  both the outer maximization (res.~minimization) and the inner minimization (res.~maximization) problems
are essentially finite-dimensional or
nonstochastic, see \cite{Berts01,DeI15} and the references therein.
This is perhaps because
the existing interchangeability results
in the literature are
established in the
finite-dimensional spaces.
We 
hope that our new interchangeability result (Lemmas~\ref{lemma-interchange-expectation} and \ref{lem-int-change})
will help to make a breakthrough 
in these models.





Constructing a nominal utility function for the ambiguity set of Kantorovich ball is another important component of this work.
The 
 approaches outlined in \ref{EC:Elicit-Ambiguity-consct}
for state-dependent
and state-independent
utility
cases
provide a new avenue
for estimating an approximate utility function based on incomplete information of scoring and may provide
a new direction for general preference elicitation.
It remains to be an open question 
how the model will perform if the  number of states
is incorrectly preset, how to design consumption trajectories for
more
effective preference elicitation (a main departure from one stage),
and how to deal with errors occurring in scoring.
The
setting of the radius of the Kantorovich ball
can also be
improved by  comprehensively
considering the estimation error on the number of states, the estimation error on piecewise linear approximation and the
errors in scoring, with
appropriate statistical inference and guarantee.

Another
aspect of
our model
which could
be potentially strengthened
is
that instead of
separating
the preference elicitation/scoring
and the optimization process,
we may consider 
the dynamic interaction between the elicitation process and the optimization process on an online footing.
Online optimization, reinforcement learning or meta-learning approaches
may be further
incorporated to improve the intelligence of
our MS-PRO model.
We leave all these issues for future research.
}

\section*{Acknowledgements}
This work was funded by the National Key R\&D Program of China (No. 2022YFA1004000, 2022YFA1004001), National Natural Science Foundation of China (No. 11991023 and 11901449),
RGC grant (14204821) and CUHK startup grant.
\ECSwitch

\ECHead{Electronic Companion for ``Multistage Utility Preference Robust Optimization''}

\section{An example of multistage portfolio selection problem with utility maximization}
\label{ec-example-portfolio}

\begin{example}\label{ex-utility-max}
Consider a financial market with $n$ risky assets.
Suppose that an investor joins the market at time $0$ with a positive initial wealth $w_0$ and plans to
invest her/his wealth in the market for $T$ periods.
At each period, the investor gains a reward which could be her/his end-of-period wealth or
the increase of her/his wealth over this period, i.e., $h_t(r_t,x_t)=(e+r_{t})^{\top} x_{t}$ or $h_t(r_t,x_t)=r_{t}^{\top} x_{t}$,
where $x_t$ is the asset allocation vector, $r_t$ is the excess return rate vector of $n$ risky assets over period $t$,
$e=[1,\cdots,1]^\top$ is the vector with
all components being one.
The investor presets a utility $u_t(\cdot)$ which measures her/his preferences on the reward over that period.
Then the investor would like to maximize
the overall expected utility over all of the $T$
periods
by adjusting her/his portfolios at the beginning of each period.
If the investor's objective is to maximize
the overall expected utility of the wealth, the decision making problem can be reformulated as a
multistage expected utility maximization problem:
\begin{eqnarray*}
& \max\limits_{\x_{[
T]}} & \mathbb{E}\left[u_1((e+r_{1})^{\top} x_{1})+u_2((e+r_{2})^{\top} x_{2})+\cdots+u_T((e+r_{T})^{\top} x_{T})\right]  \\
& {\rm s.t.} & x_1\in \{x\in \mathbb{R}^n_+ \mid e^\top x=w_0 \}, x_t(r_{[t-1]})\in \{x\in \mathbb{R}^n_+\mid e^{\top} x_{t}=(e+r_{t-1})^{\top} x_{t-1}\}, t=2,\ldots,T,
\end{eqnarray*}
where $x_t$ is the asset allocation vector of the current wealth invested in the $n$ risky assets at the beginning of period $t$, $r_{t}^{\top} x_{t}$ is the
wealth at the end of period $t$.
$e^{\top} x_{t}=(e+r_{t-1})^{\top} x_{t-1}$ is the wealth balance equation, which together with the no-shorting constraint characterizes the feasible set $\mathcal{X}_t$ of portfolio $x_{t}$ at period $t$.
If the investor's utility
is valued over the state-wise return rates,
the objective could be set as
$$
\mathbb{E}\left[u_1\left(\frac{{r}_1^{\top} x_{1}}{e^\top x_1}\right)+u_2\left(\frac{{r}_2^{\top} x_{2}}{e^\top x_2}\right)+\cdots+u_T\left(\frac{{r}_T^{\top} x_{T}}{e^\top x_T}\right)\right].
$$
In this case, by normalizing $\tilde{x}_t=x_t/(e^\top x_t)$, 
we have an equivalent utility maximization problem:
\begin{equation*}
\begin{array}{cl}
 \max\limits_{\tilde{x}_{1},\{\tilde{x}_t(\cdot)\}} & 
 \mathbb{E}\left[u_1( {r}_1^\top \tilde{x}_{1})+
 u_2( {r}_2^\top \tilde{x}_{2}) +\cdots + u_T( {r}_T^\top \tilde{x}_{T}) \right] \\
 {\rm s.t. }& e^\top \tilde{x}_1=1,\ \tilde{x}_1\in [0,1]^n,
 \ e^\top \tilde{x}_{t}(r_{[t-1]})
 =1,\ \tilde{x}_t(\cdot)\in \mathcal{L}^0([0,1]^n),\ t=2,\ldots,T.
\end{array}
\end{equation*}
\end{example}

 \section{Proofs}

{
 \subsection{Proof of Lemma~\ref{lemma-interchange-expectation}.}
\label{EC-interchangibility}

For any $\mathfrak{z} \in \mathfrak{M}_Z$, we have that $\mathfrak{z}(\omega) \in Z(\omega)
\subseteq \mathbb{Z}$ and hence
$\inf _{z \in {Z}(\omega)
} f(z, \omega)\leq   f(\mathfrak{z}(\omega), \omega)$
a.s..
Since $Z$ is measurable and $f$ is continuous in $z$, it follows by
\cite[Theorem 8.2.11]{AuF09-EC} that
$\inf _{z \in {Z}(\omega)
} f(z, \omega)$ is measurable.
By taking mathematical expectation on both sides of the inequality, we obtain
\bgeqn
\mathbb{E}\left[\inf _{z \in {Z}(\omega)
} f(z, \omega)\right]\leq   \mathbb{E}\left[ f(\mathfrak{z}(\omega), \omega)\right] = \mathbb{E}\left[F_{\mathfrak{z}}\right].
\edeqn
Moreover, by taking infimum w.r.t. $\mathfrak{z}$ over $\mathfrak{M}_Z$
on both sides of the inequality, we have
\bgeqn \label{eq-proof-le1}
\mathbb{E}\left[\inf _{z \in {Z}(\omega)
} f(z, \omega)\right]\leq  \inf_{\mathfrak{z} \in \mathfrak{M}_Z} \mathbb{E}\left[F_{\mathfrak{z}}\right].
\edeqn
Next, we show
the
inequality holds in the opposite direction.
We first
consider the case when $\inf _{z \in {Z}(\omega)} f(z, \omega)$ is finite valued a.s..
For $k=1,2,\cdots$, we consider the level-set mapping $S_k:\Omega \rightrightarrows \mathbb{Z}$
where
$$
S_k(\omega) = \left\{ z\in \mathbb{Z}:
f(z, \omega) \leq \inf_{z \in {Z}(\omega)} f(z, \omega) + \frac{1}{k} \right\}.
$$
Since $f$ is a Carath\'edory function,
we know
from \cite[Lemma 8.2.6]{AuF09-EC} that
$f(z, \omega)$ is 
$\B(\mathbb{Z})\bigotimes\F$-measurable
and for every $\omega\in\Omega$,
the function $z \rightarrow f(z,\omega)$ is continuous.
Thus $S_k(\omega)$ is
a closed set for
every given $\omega$.
By \cite[Theorem 8.1.4]{AuF09-EC},
the measurability of $f$ ensures the measurability of
$S_k(\omega)$
w.r.t.~$\F$.
Let
$$
Z_k(\omega) 
:= S_k(\omega)\bigcap Z(\omega), \forall \omega\in \Omega.
$$
Since both $S_k$ and  $Z$ are $\F$-measurable,
by \cite[Theorem 8.2.4]{AuF09-EC}, $Z_k$ is also $\F$-measurable.
Moreover, since $\inf_{z \in {Z}(\omega)} f(z, \omega)$
is finite-valued
a.s.,
then
$Z_k(\omega)$ is non-empty
and
$$
f(z,\omega) \leq
\inf_{z \in {Z}(\omega)} f(z, \omega) + \frac{1}{k}, \; \forall z\in Z_k(\omega), \; \text{a.s.}.
$$
Together with the closedness and $\F$-measurablity of $Z_k$,
we know
by virtue of Theorem 8.1.3 in \cite{AuF09-EC}
that there exists a $\F$-measurable selection $\mathfrak{z}_k$ of $Z_k$ such that
\begin{equation}\label{eq-proof-fz}
    f(\mathfrak{z}_k(\omega), \omega) \leq \inf_{z \in {Z}(\omega)} f(z, \omega) + \frac{1}{k},\ \text{a.e.}\ \omega\in \Omega.
    \end{equation}
By
taking expectation
on both sides of inequality \eqref{eq-proof-fz}, we have
$$
\mathbb{E}\left[F_{\mathfrak{z}_k}\right]=\mathbb{E}[f(\mathfrak{z}_k(\omega), \omega)] \leq \mathbb{E}\left[\inf_{z \in {Z}(\omega)} f(z, \omega)\right] + \frac{1}{k}.
$$
Since $\mathfrak{z}_k$ is a $\F$-measurable selection from $Z_k=S_k\bigcap Z$,
then $\mathfrak{z}_k\in \mathfrak{M}_Z$.
Letting $k\rightarrow +\infty$ gives us that
\begin{equation}\label{eq-proof-fz-1}
\inf_{\mathfrak{z}\in \mathfrak{M}_Z}\mathbb{E}\left[F_{\mathfrak{z}}\right] \leq \mathbb{E}\left[\inf_{z \in {Z}(\omega)} f(z, \omega)\right].
\end{equation}
Combining with inequality (\ref{eq-proof-le1}), we arrive at
(\ref{eq:interchange}) as desired.

Next, we move on to consider
two extreme cases:
(a) the event $\{\omega \mid \inf _{z \in {Z}(\omega)
} f(z, \omega)=-\infty \}$
has a positive probability $p_{-\infty}$
and (b) the event
 $\{\omega \mid \inf _{z \in {Z}(\omega)
} f(z, \omega)=+\infty \}$
has a positive probability $p_{+\infty}$.
We first consider case (a). In this case,
$\mathbb{E}\left[\left(-\inf_{z \in {Z}(\omega)} f(z, \omega)\right)_+\right]= +\infty$, which,
by
the
assumption of the lemma, implies
$\mathbb{E}\left[ \left(\inf_{z \in {Z}(\omega) } f(z, \omega)\right)_+\right]< +\infty$.
This gives rise to
$
\mathbb{E}\left[\inf _{z \in {Z}(\omega)
} f(z, \omega)\right] =-\infty.
$
We can use the same approach
as that in the finite-valued case to
show that the right hand side of \eqref{eq-proof-le1}
is also equal to $-\infty$.
Specifically,
for any $k\in \{1,2,3\ldots\}$,
we consider the level set mapping
\begin{equation*}
S_k(\omega) := \left\{ z\in \mathbb{Z}\ \Big| \
\begin{array}{ll}
f(z, \omega) \leq \inf_{z \in {Z}(\omega)} f(z, \omega) + \frac{1}{k}, & \text{if } \inf_{z \in {Z}(\omega)} f(z, \omega) >-\infty \\
f(z, \omega) \leq -k, & \text{if } \inf_{z \in {Z}(\omega)} f(z, \omega) =-\infty
\end{array}\right\}.
\end{equation*}
We can show that
$S_k(\omega)\bigcap Z(\omega)$ is
$\F$-measurable
and
there exists a $\F$-measurable
selection of $\mathfrak{z}_k$ from $Z_k :=S_k\bigcap Z$
such that
\begin{eqnarray*}
\inf_{\mathfrak{z} \in \mathfrak{M}_Z} \mathbb{E}\left[F_{\mathfrak{z}}\right]
&\leq&
 \mathbb{E}\left[F_{\mathfrak{z}_k}\right] \\
&\leq &
  \int_{\inf_{z \in {Z}(\omega)} f(z, \omega)=-\infty} (-k) \mathbb{P}(d\omega)
 + \int_{\inf_{z \in {Z}(\omega)} f(z, \omega)>-\infty} \left(\inf_{z \in {Z}(\omega)} f(z, \omega) + \frac{1}{k}\right) \mathbb{P}(d\omega) \\
&
=& -p_{-\infty} k
+ \int_{\inf_{z \in {Z}(\omega)} f(z, \omega)>-\infty} \left( \left[\inf_{z \in {Z}(\omega)} f(z, \omega)\right]_+ - \left[-\inf_{z \in {Z}(\omega)} f(z, \omega)\right]_+ \right) \mathbb{P}(d\omega)\\
&&
+ (1-p_{-\infty}) \frac{1}{k} \\
&\leq & -p_{-\infty} k
+ \int_{\Omega} \left[\inf_{z \in {Z}(\omega)} f(z, \omega)\right]_+ \mathbb{P} (d\omega)
+ (1-p_{-\infty}) \frac{1}{k},
\end{eqnarray*}
where $p_{-\infty} = \mathbb{P}\left(\inf_{z \in {Z}(\omega)} f(z, \omega)=-\infty\right)>0$.
Since $\int_{\Omega} [\inf_{z \in {Z}(\omega)} f(z, \omega)]_+ \mathbb{P}( d\omega) <+\infty$,
by letting $k\rightarrow +\infty$, we 
arrive at
$
\inf_{\mathfrak{z} \in \mathfrak{M}_Z} \mathbb{E}\left[F_{\mathfrak{z}}\right]=-\infty
$
as desired.

Consider now case (b). In this case,
$\mathbb{E}\left[\left(\inf_{z \in {Z}(\omega)} f(z, \omega)\right)_+\right]= +\infty$, which implies by
the assumption of the lemma that
$\mathbb{E}\left[ \left[-\inf_{z \in {Z}(\omega) } f(z, \omega)\right]_+\right]<+ \infty$.
This gives rise to
$$
\mathbb{E}\left[\inf _{z \in {Z}(\omega)
} f(z, \omega)\right] =+\infty.
$$
By \eqref{eq-proof-le1}, we know that both sides of \eqref{eq:interchange} are equal to $+\infty$.
Note that
the cases  (a) and (b)  cannot occur
simultaneously
due to the assumption on the positive/negative part of the expectation. 
\hfill $\Box$

 }

 \subsection{Proof of Lemma \ref{lem-int-change}}
 \label{EC-Lemma2}
{

 The thrust of the proof is to
 fit (\ref{eq:lemma2-condi-expt-F-xi}) in the framework of Lemma 1 so that the principle of the interchangeability established in Lemma 1
 can be readily applied. To this effect,
we introduce
a new random function $\mathfrak{\hat{u}}:\R \times \Omega \rightarrow \R$ 
such that
$\mathfrak{\hat{u}}(x,\omega)=\mathfrak{{u}}(x,\xi(\omega))$, $\forall x\in \R$ and a.e. $\omega\in\Omega$,
where $\mathfrak{{u}}\in\mathfrak{M}_\mathcal{U}$.
{Define $\mathbb{Z}:=\mathcal{L}^p(\R \rightarrow\R)$ as a 
functional space.}
Let $\mathfrak{U}$  denote
 the 
 space of measurable functions $ \mathfrak{\hat{u}}: \Omega \rightarrow \mathbb{Z}$ with finite $p$-th order moments and define
$$
\hat{\mathfrak{M}}_{\mathcal{U}} :=\left\{ \mathfrak{\hat{u}}\in \mathfrak{U}\ \Big|\ \begin{array}{ll}
       \mathfrak{\hat{u}}(x,\omega)=\mathfrak{{u}}(x,\xi(\omega)),\ \forall x\in \R, \text{for a.e.}\; \omega, \\
      \mathfrak{u}(\cdot,\xi) \in \mathcal{U}(\xi), 
      \text{ for any } \xi
\end{array}   \right\}.$$
By letting $\hat{\mathcal{U}}:={\mathcal{U}}(\xi)$,
we have
$
\hat{\mathfrak{M}}_{\mathcal{U}} = \left\{ \mathfrak{\hat{u}}\in \mathfrak{U} \mid \mathfrak{\hat{u}}(x,\omega)\in \hat{\mathcal{U}}(\omega),\ \text{for a.e.}\ \omega\in\Omega
  \right\}.
  $
By changing the variable from $\mathfrak{u}$ to $\mathfrak{\hat{u}}$,
 we obtain 
 $$
 {\inf\limits_{\mathfrak{u}\in{{\mathfrak{M}}_{\mathcal{U}}}}\mathbb{E}\left[\mathfrak{u}(\eta,\xi)\right] = \inf\limits_{\mathfrak{u}\in{{\mathfrak{M}}_{\mathcal{U}}}} \int_{\Omega} \mathfrak{u}(\eta(\omega),\xi(\omega))d\mathbb{P}(\omega)=
 \inf\limits_{\mathfrak{{\hat{u}}}\in{\hat{\mathfrak{M}}_{\mathcal{U}}}} \int_{\Omega} \mathfrak{\hat{u}}(\eta(\omega),\omega)d\mathbb{P}(\omega)}
 = \inf\limits_{\mathfrak{{\hat{u}}}\in{\hat{\mathfrak{M}}_{\mathcal{U}}}} \mathbb{E}\left[\mathfrak{\hat{u}}(\eta(\omega),\omega)\right].
 $$
Moreover, by the tower property of the expectation operator, we have
\bgeq
\begin{array}{ll}
&
\mathbb{E}\left[\mathfrak{\hat{u}}(\eta(\omega),\omega)\right] =
\mathbb{E}\left[ \mathbb{E}_{\mid \F_\xi}\left[\mathfrak{\hat{u}}(\eta(\omega),\omega)
\right]   \right],
\end{array}
\edeq
where $\mathbb{E}_{\mid \F_\xi}$ denotes the conditional expectation with respect to $\F_\xi$,
and hence
\bgeqn
\inf\limits_{\mathfrak{u}\in{{\mathfrak{M}}_{\mathcal{U}}}}\mathbb{E}\left[\mathfrak{u}(\eta,\xi)\right]
=\inf\limits_{\mathfrak{{\hat{u}}}\in{\hat{\mathfrak{M}}_{\mathcal{U}}}}
\mathbb{E}\left[ \mathbb{E}_{\mid \F_\xi}\left[\mathfrak{\hat{u}}(\eta(\omega),\omega)
\right]\right].
\label{eq:Lema2-prf-3}
\edeqn
To apply Lemma 1, we define
\begin{eqnarray}
f(\mathfrak{\hat{u}}(\cdot,\omega),\omega) := \mathbb{E}_{\mid \F_\xi}\left[\mathfrak{\hat{u}}(\eta(\omega),\omega)
\right].
\label{eq-lemma2-p1}
\end{eqnarray}
Here, $f$: $\mathbb{Z}\times \Omega
\rightarrow \mathcal{L}^p(\Omega,\F_{\xi},\mathbb{P};\R)$ is a 
functional with
$f(z,\omega)
=
\mathbb{E}_{\mid \F_\xi}\left[z(\eta(\omega))
\right]
]$ for each $z\in \mathbb{Z}$.
By the definition, we can see that
for each fixed $\omega$, $f(z,\omega)$ is continuous
in $z$ since the conditional expectation is a linear operator. For each fixed $z$, $f(z,\cdot)$ is measurable.
By the continuity of $z(\cdot)\in
\mathbb{Z}$, and 
the measurability of $\mathbb{E}_{\mid \F_\xi}$ and $\eta$, we
know
by virtue of \cite[Corollary  8.2.3]{AuF09-EC}
that $f(z,\omega)$ is $\F$-measurable.
Thus, $f(z,\omega)$ is a Carath\'edory function.
Thus,
$$ \mathbb{E}\left[ \mathbb{E}_{\mid \F_\xi}\left[\mathfrak{\hat{u}}(\eta(\omega),\omega)
\right]   \right] = \mathbb{E}\left[f(\mathfrak{\hat{u}}(\cdot,\omega),\omega)\right].
 $$
By the nonemptyness of
$\hat{\mathfrak{M}}_{\mathcal{U}}$
and the boundedness of $\mathfrak{u} \in \hat{\mathfrak{M}}_\mathcal{U}$,   $\inf\limits_{\mathfrak{\hat{u}}\in{\hat{\mathfrak{M}}_{\mathcal{U}}}}\mathbb{E}\left[f(\mathfrak{\hat{u}}(\cdot,\omega),\omega)\right]
$ is bounded.
By Lemma \ref{lemma-interchange-expectation}
(here we
 require $\mathfrak{u}$ to
 have finite $p$-th moment which is equivalent to
 $\mathfrak{z}$ having finite $p$-th moment. This additional condition does not affect the result in
 Lemma~\ref{lemma-interchange-expectation}), we have  that
\bgeqn
\inf\limits_{\mathfrak{\hat{u}}\in{\hat{\mathfrak{M}}_{\mathcal{U}}}}\mathbb{E}\left[f(\mathfrak{\hat{u}}(\cdot,\omega),\omega)\right]
=\mathbb{E}\left[\inf_{{u} \in {\hat{\mathcal{U}}}(\omega)} f({u}, \omega)\right],
\label{eq:Lema2-prf-5}
\edeqn
where
$ \hat{\mathcal{U}}(\omega) ={\mathcal{U}}(\xi(\omega))$,
for a.e. $\omega\in \Omega$.
Thus
\bgeqn
\mathbb{E}\left[\inf_{{u} \in \hat{\mathcal{U}}(\omega)} f({u}, \omega)\right]
&=&\mathbb{E}\left[  \inf\limits_{{u}\in \hat{\mathcal{U}}(\omega)} \mathbb{E}\left[ {u}(\eta(\omega))\mid \F_{\xi} \right]   \right]
=\mathbb{E}\left[  \inf\limits_{{u}\in \mathcal{U}(\xi(\omega))} \mathbb{E}\left[ {{u}}(\eta(\omega)) \mid \F_{\xi} \right]   \right]
\nonumber\\
&:=& \mathbb{E}\left[  \inf\limits_{{u}\in \mathcal{U}(\xi)} \mathbb{E}\left[ {{u}}(\eta) \mid \F_{\xi} \right]   \right].
\label{eq:Lema2-prf-6}
\edeqn
Combining (\ref{eq:Lema2-prf-3})-(\ref{eq:Lema2-prf-6}), we obtain
(\ref{eq:lemma2-condi-expt-F-xi}) as desired.
}

\hfill $\Box$
%

 \subsection{Proof of Proposition \ref{prop-rectangular}}
 {
By Lemma \ref{lem-int-change},
\begin{equation}\label{eq-15}
\begin{array}{ll}
&\inf\limits_{\mathfrak{u}_t\in{\mathcal{U}_t}}\mathbb{E}_{|\mathcal{F}_0}\left[\mathfrak{u}_t(Z_t(\xi_t),\xi_{[t-1]})\right]
= \mathbb{E}_{|\mathcal{F}_0}\left[  \inf\limits_{u_t\in \mathcal{U}_t(\xi_{[t-1]})} \mathbb{E}_{|\mathcal{F}_{t-1}}\left[ u_t(Z_t(\xi_t))\right]   \right],
\end{array}
\end{equation}
where $\mathcal{U}_t=\{ {\mathfrak{u}_t}\mid \exists \vec{\mathfrak{u}}_{[1,t-1]} \inmat{ and } \vec{\mathfrak{u}}_{[t+1,T]} \inmat{ such that } [\vec{\mathfrak{u}}_{[1,t-1]},\mathfrak{u}_t, \vec{\mathfrak{u}}_{[t+1,T]}]^\top \in \mathcal{U}  \}$.
From Definition 1, we can
see that $\mathcal{U}_t(\xi_{[t-1]})$ is a
decomposition of $\mathcal{U}$.
By the decomposability of the objective function and feasible set ${\cal U}$, the
tower property and the translation invariance property of the expectation operator, we have
\begin{align}
&\inf\limits_{\vec{\mathfrak{u}}\in \mathcal{U}}\mathbb{E}_{|\mathcal{F}_0}\left[\mathfrak{u}_1(Z_1)+\mathfrak{u}_2(Z_2,\xi_1)+\cdots+\mathfrak{u}_T(Z_T,\xi_{[T-1]})\right] \nonumber\\
=&\inf\limits_{\vec{\mathfrak{u}}\in \mathcal{U}} \sum_{t=1}^T \mathbb{E}_{|\mathcal{F}_0}\left[\mathfrak{u}_t(Z_t,\xi_{[t-1]})\right] \nonumber\\
=&\inf\limits_{\mathfrak{u}_t\in \mathcal{U}_t, t=1,\ldots,T} \sum_{t=1}^T \mathbb{E}_{|\mathcal{F}_0}\left[\mathfrak{u}_t(Z_t,\xi_{[t-1]})\right] \nonumber\\
=& \sum\limits_{t=1}^{T} \inf\limits_{\mathfrak{u}_t\in \mathcal{U}_t} \mathbb{E}_{|\mathcal{F}_0}\left[ \mathfrak{u}_t(Z_t,\xi_{[t-1]})   \right]\nonumber\\
=& \sum\limits_{t=1}^{T} \inf\limits_{\mathfrak{u}_t\in \mathcal{U}_t} \mathbb{E}_{|\mathcal{F}_0}\left[   \mathbb{E}_{|\mathcal{F}_{t-1}}\left[ \mathfrak{u}_t(Z_t,\xi_{[t-1]})\right]   \right]\nonumber\\
=& \sum\limits_{t=1}^{T} \mathbb{E}_{|
\mathcal{F}_0}\left[  \inf\limits_{u_t\in \mathcal{U}_t(\xi_{[t-1]})} \mathbb{E}_{|\mathcal{F}_{t-1}}\left[ u_t(Z_t)\right]   \right] \label{eq-18}\\
=& \mathbb{E}_{|\mathcal{F}_0}\left[\inf\limits_{ u_1\in \mathcal{U}_1} \mathbb{E}_{|\mathcal{F}_0}\left[ u_1(Z_1)\right]\right]+\mathbb{E}_{|\mathcal{F}_0}\left[\inf\limits_{u_2\in \mathcal{U}_2(\xi_{[1]})} \mathbb{E}_{|\mathcal{F}_1}\left[ u_2(Z_2)\right]\right]+\cdots \nonumber \\ &\qquad +\mathbb{E}_{|\mathcal{F}_0}\left[\mathbb{E}_{|\mathcal{F}_1}\left[\cdots\mathbb{E}_{|\mathcal{F}_{T-2}
}\left[\inf\limits_{u_T\in \mathcal{U}_T(\xi_{[T-1]})} \mathbb{E}_{|\mathcal{F}_{T-1}}\left[ u_T(Z_T)\right]\right]\cdots\right]\right] \\ 
=&\inf\limits_{ u_1\in \mathcal{U}_1} \mathbb{E}_{|\mathcal{F}_0}\left[ u_1(Z_1)+\inf\limits_{u_2\in \mathcal{U}_2(\xi_{[1]})} \mathbb{E}_{|\mathcal{F}_1}\left[ u_2(Z_2)+\cdots  +\inf\limits_{u_T\in \mathcal{U}_T(\xi_{[T-1]})} \mathbb{E}_{|\mathcal{F}_{T-1}}\left[ u_T(Z_T)\right]\cdots\right]\right],\label{eq-19}
\end{align}
which gives rise to (\ref{eq:Rectangl}).
\hfill $\Box$
}

\subsection{Proof of Theorem \ref{theorem-dp-reform}}
\label{EC-theorem-dp-reform}
{
We divide the proof into three main steps.

\underline{Step 1.} We begin by decomposing problem
(MS-PRO-SD) into a stagewise
maxmin problem.
By Proposition 1, for any
fixed decision sequence  $
x_1,\cdots,x_T$ and random process $\xi_1,\ldots,\xi_T$, we have
\bgeqn
&&\inf\limits_{\vec{\mathfrak{u}}\in\mathcal{U}}\mathbb{E}\left[{u}_1(h_1\left(x_{1},\xi_1\right))
+\mathfrak{u}_2(h_{2}\left(
{x_{2}(\xi_1)}, \xi_{2}\right),\xi_{1})+\cdots+\mathfrak{u}_T(h_T\left(
{x_{T}(\xi_{[T-1]})}
, \xi_{T}\right),\xi_{[T-1]})\right]\nonumber\\
&=&
\sum_{t=1}^T {{\mathbb{E}_{\F_0}}} \left[ \inf\limits_{ u_t \in \mathcal{U}_t(\xi_{[t-1]})}\mathbb{E}_{|\mathcal{F}_{t-1}}\left[ u_{t}\left(h_t(
{x_{t}(\xi_{[t-1]})}
, \xi_{t})\right)\right]\right].
\label{eq:lfh-thm1-recursive}
\edeqn
Consequently
\bgeqn
V_1 &:=&
\max\limits_{\x_{[1,T]}\in\mathscr{X}_{[1,T]}}\inf\limits_{\vec{\mathfrak{u}}\in\mathcal{U}}\mathbb{E}\left[{u}_1(h_1\left(x_{1},\xi_1\right))
+\mathfrak{u}_2(h_{2}\left(
{x_{2}(\xi_{1})}, \xi_{2}\right),\xi_{1})+\cdots+\mathfrak{u}_T(h_T\left(
{x_{T}(\xi_{[T-1]})}, \xi_{T}\right),\xi_{[T-1]})\right]\nonumber\\
&=&
\max\limits_{\x_{[1,T]}\in \mathscr{X}_{[1,T]}}
{{\mathbb{E}_{\F_0}}}
\left[\sum_{t=1}^T\inf\limits_{ u_t \in \mathcal{U}_t(\xi_{[t-1]})}\mathbb{E}_{|\mathcal{F}_{t-1}}\left[ u_{t}\left(h_t(
{x_{t}(\xi_{[t-1]})}, \xi_{t})\right)\right]\right].\nonumber
\edeqn
Here, we denote $\mathscr{X}_{[t,T]}:=\{ \x_{[t,T]}\mid \x_s(\xi_{[s-1]})\in \mathscr{X}_s(x_{[s-1]},\xi_{[s-1]}), s=t,\ldots,T\} $, $t=1,\ldots,T$, for short.
At stage $t=1,\ldots,T$, for given $x_{[t-1]}$ and $\xi_{[t-1]}$,
we denote the
optimal value of the
sub-optimization problem at remaining stages by
\begin{align}
&V_{t}\left(x_{[t-1]}, \xi_{[t-1]}\right):=\max\limits_{
{\x_{[t,T]}\in \mathscr{X}_{[t,T]}  }}
\mathbb{E}_{\F_{t-1}}
\left[\sum_{s=t}^T\inf\limits_{ u_s \in \mathcal{U}_s(\xi_{[s-1]})}\mathbb{E}_{|\mathcal{F}_{s-1}}\left[ u_{s}\left(h_s(
{x_{s}(\xi_{[s-1]})}
, \xi_{s})\right)\right]\right].&
\label{eq:lfh-thm1-recursive-a}
\end{align}
Let $V_{T+1}(\cdot, \cdot) := 0$. At the first stage, $V_1$ is the optimal value of problem (MS-PRO-
SD). We then prove the dynamic equations \eqref{eq:thm-recursive-formula} between $V_t$ and $V_{t+1}$ by induction.
At stage $T$, we have   \eqref{eq:thm-recursive-formula} directly by the definition above.
We then prove
\eqref{eq:thm-recursive-formula} at
stage $T-1$ in Step 2,
and then prove that
the equation at stage $T-1$ implies the equation at stage $T-2$ in Step 3.
As the
induction relationship between adjacent two stages holds, we can
establish 
the results by induction.

\underline{Step 2.}
We
consider the sub-optimization problem at the last two stages.
On the basis of the right-hand side of (\ref{eq:lfh-thm1-recursive-a}),
we  prove the recursive formula for $t=T-1$.
Observe first that
{
\begin{eqnarray}
\label{eq:thm1-prf-1}
&&
V_{T-1}\left(x_{[T-2]}, \xi_{[T-2]}\right):=
\max\limits_{{x_{[T-1,T]}\in \mathscr{X}_{[T-1,T]}}}
\mathbb{E}_{\F_{T-2}}\left[
\sum_{t=T-1}^T\inf\limits_{ u_t \in \mathcal{U}_t(\xi_{[t-1]})}\mathbb{E}_{|\mathcal{F}_{t-1}}\left[ u_{t}\left(h_t(
{\x_{t}(\xi_{[t-1]})}, \xi_{t})\right)\right]\right].\nonumber\\
&=&\max\limits_{x_{T-1}\in \mathscr{X}_{T-1}\left(x_{[T-2]}, \xi_{[T-2]}\right)\atop
{\x_{T}(\xi_{[T-1]})}\in \mathscr{X}_{T}\left(
{\x_{T-1}}
, \xi_{[T-1]}\right)} \Bigg[
\mathbb{E}_{|\mathcal{F}_{T-2}}\bigg[
\inf\limits_{ u_{T-1}\in \mathcal{U}_{T-1}(\xi_{[T-2]})} \mathbb{E}_{|\mathcal{F}_{T-2}}\Big[ u_{T-1}(h_{T-1}\left(x_{T-1},\xi_{T-1}\right))\Big]\bigg]\nonumber\\ &
&+\mathbb{E}_{|\mathcal{F}_{T-2}}\bigg[ \inf\limits_{u_T\in \mathcal{U}_T(\xi_{[T-1]})} \mathbb{E}_{|\mathcal{F}_{T-1}}\left[ {u}_T(h_T\left(
{\x_{T}(\xi_{[T-1]})}, \xi_{{T}}\right))\right]\bigg]\Bigg]
\nonumber\\ &
= & \max\limits_{x_{T-1}\in \mathscr{X}_{T-1}\left(x_{[T-2]}, \xi_{[T-2]}\right)} \Bigg[  \inf\limits_{ u_{T-1}\in \mathcal{U}_{T-1}(\xi_{[T-2]})} \mathbb{E}_{|\mathcal{F}_{T-2}}\left[ u_{T-1}(h_{T-1}\left(x_{T-1},\xi_{T-1}\right))\right]\nonumber\\ &
&+  \max\limits_{
{\x_{T}(\xi_{[T-1]})}
\in \mathscr{X}_{T}\left(x_{[T-1]}, \xi_{[T-1]}\right)}  \mathbb{E}_{|\mathcal{F}_{T-2}}\bigg[\inf\limits_{u_T\in \mathcal{U}_T(\xi_{[T-1]})} \mathbb{E}_{|\mathcal{F}_{T-1}}\left[ u_T(h_T\left(
{\x_{T}(\xi_{[T-1]})}
, \xi_{{T}}\right))\right]\bigg]\Bigg].
\label{eq:thm1-prf-1-b}
\end{eqnarray}
}
This is because the objective in the square brackets 
is separable and the first term is independent of $x_T$. Let
\bgeqn
f_T(x_T,\xi_{[T-1]}) :=\inf\limits_{u_T\in \mathcal{U}_T(\xi_{[T-1]})} \mathbb{E}_{|\mathcal{F}_{T-1}}\left[ u_T(h_T\left(x_{T}, \xi_{T}\right))\right].
\label{eq:f_T-defi}
\edeqn
For fixed $\xi_{[T-1]}$, since
$\mathcal{U}_T(\xi_{[T-1]})$ is a compact set and $ \mathbb{E}_{|\mathcal{F}_{T-1}}\left[ u_T(h_T\left(x_{T}, \xi_{T}\right))\right]$ is continuous in $x_T$ under conditions (a) and  (b), then
$f_T(x_T,\xi_{[T-1]})$ is finite-valued. Moreover, for any $\hat{x}_T, \tilde{x}_T\in \mathscr{X}_{T}\left(x_{[T-1]}, \xi_{[T-1]}\right)$,
\begin{align}
\label{eq:f--cont}
|f_T(\hat{x}_T,\xi_{[T-1]})
-f_T(\tilde{x}_T,\xi_{[T-1]})|
\leq \sup\limits_{u_T\in \mathcal{U}_T(\xi_{[T-1]})} \mathbb{E}_{|\mathcal{F}_{T-1}}\big[ |u_T(h_T\left(\hat{x}_{T}, \xi_{T}\right))
-u_T(h_T\left(\tilde{x}_{T}, \xi_{T}\right))|
\big].
\end{align}
Since any $u_T\in \mathcal{U}_T(\xi_{[T-1]})$ is
globally Lipschitz continuous
under condition (a),
\bgeqn
\label{eq:u-equi-cont}
|u_T(h_T\left(\hat{x}_{T}, \xi_{T}\right))
-u_T(h_T\left(\tilde{x}_{T}, \xi_{T}\right))| \leq
\kappa(\xi_{[T-1]})
|h_T\left(\hat{x}_{T}, \xi_{T}\right)
-h_T\left(\tilde{x}_{T}, \xi_{T}\right)|,\nonumber\\
\qquad \forall \; \hat{x}_T,\tilde{x}_T\in \mathscr{X}_T(x_{[T-1]},\xi_{[T-1]}),
\edeqn

Under condition (b),
\bgeqn
\label{eq:h-Lip-cont}
|h_T\left(\hat{x}_{T}, \xi_{T}\right)
-h_T\left(\tilde{x}_{T}, \xi_{T}\right)
|\leq \sigma(\xi_T)\|\hat{x}_T-\tilde{x}_T\|,\; \forall \; \hat{x}_T,\tilde{x}_T\in \mathscr{X}_T(x_{[T-1]},\xi_{[T-1]}),
\edeqn
where $\bbe_{|\F_{[T-1]}}[\sigma(\xi_T)]<+\infty$.
Combining (\ref{eq:f--cont})-(\ref{eq:h-Lip-cont}), we obtain
\bgeqn
\label{eq:f--cont-a}
|f_T(\hat{x}_T,\xi_{[T-1]})
-f_T(\tilde{x}_T,\xi_{[T-1]})|
&\leq&
\sup\limits_{u_T\in \mathcal{U}_T(\xi_{[T-1]})}
\mathbb{E}_{|\mathcal{F}_{T-1}}\big[
\kappa(\xi_{[T-1]})
\sigma(\xi_T)\|\hat{x}_T-\tilde{x}_T\|\big]
\nonumber\\
& = & \kappa(\xi_{[T-1]}) \mathbb{E}_{|\mathcal{F}_{T-1}}\big[
\sigma(\xi_T)\big]\|\hat{x}_T-\tilde{x}_T\|.
\edeqn
Hence we obtain the
continuity of
$f_T$ in $x_T$ for fixed $\xi_{T-1}$ and $\xi_T$.
Next, we can show that
$f_T(x_T,\xi_{[T-1]})$ is a Carath\'edory function, 
that is,
 for fixed $x_T$,  $f_T(x_T,\xi_{[T-1]})$
 is $\F_{[T-1]}$-measurable.
To see this, we note that
$ \mathbb{E}_{|\mathcal{F}_{T-1}}\left[ u_T(h_T\left(x_{T}, \xi_{T}\right))\right]$
is continuous in $u_T$ and is $\F_{[T-1]}$-measurable
for fixed $u_T$,
and
 $\mathcal{U}_T(\xi_{[T-1]})$ is $\F_{[T-1]}$-measurable,
 by the marginal map theorem
\cite[Theorem 8.2.11]{AuF09-EC},
$f_T$ is $\F_{[T-1]}$-measurable for fixed $x_T$.

By Lemma \ref{lemma-interchange-expectation}, we have
\bgeqn
\max\limits_{
\bm{x}_{T}(\xi_{[T-1]})\in \mathscr{X}_{T}(\bm{x}_{[T-1]}
, \xi_{[T-1]})
}
\mathbb{E}_{|\mathcal{F}_{T-2}}\bigg[
f_T(\bm{x}_T{(\xi_{[T-1]})},\xi_{[T-1]})
\bigg]
&=& \mathbb{E}_{|\mathcal{F}_{T-2}}\bigg[ \max\limits_{x_{T}\in \mathscr{X}_{T}\left(x_{[T-1]}, \xi_{[T-1]}\right)}
f_T(x_T,\xi_{[T-1]})
\bigg]
\nonumber\\
&=:&
\mathbb{E}_{|\mathcal{F}_{T-2}}\bigg[V_{T}\left(x_{[T-1]}, \xi_{[T-1]}\right)\bigg].
\label{eq:thm1-prf-1-V}
\edeqn
Note that $V_{T}\left(x_{[T-1]}, \xi_{[T-1]}\right)$ is well-defined since
$f_T(x_T,\xi_{[T-1]})$ is uniformly bounded
under the uniform boundedness condition of
$u_T$ and the fact that
$\max\limits_{x_{T}\in \mathscr{X}_{T}\left(x_{[T-1]}, \xi_{[T-1]}\right)}
f_T(x_T,\xi_{[T-1]})$ is $\F_{[T-1]}$-measurable.
Combining (\ref{eq:thm1-prf-1-b}) and (\ref{eq:thm1-prf-1-V}) gives us that 
\begin{eqnarray}
&&{V_{T-1}\left(x_{[T-2]}, \xi_{[T-2]}\right)} \nonumber\\
&
= & \max\limits_{x_{T-1}\in \mathscr{X}_{T-1}\left(x_{[T-2]}, \xi_{[T-2]}\right)}  \Bigg[ \inf\limits_{ u_{T-1}\in \mathcal{U}_{T-1}(\xi_{[T-2]})} \mathbb{E}_{|\mathcal{F}_{T-2}}\left[ u_{T-1}(h_{T-1}\left(x_{T-1},\xi_{T-1}\right))\right]\nonumber\\ &
&+  \mathbb{E}_{|\mathcal{F}_{T-2}}\bigg[ \max\limits_{x_{T}\in \mathscr{X}_{T}\left(x_{[T-1]}, \xi_{[T-1]}\right)}\inf\limits_{u_T\in \mathcal{U}_T(\xi_{[T-1]})} \mathbb{E}_{|\mathcal{F}_{T-1}}\left[ u_T(h_T\left(x_{T}, \xi_{{T}}\right))\right]\bigg]\Bigg] \nonumber\\ &
= & \max\limits_{x_{T-1}\in \mathscr{X}_{T-1}\left(x_{[T-2]}, \xi_{[T-2]}\right)}  \Bigg[ \inf\limits_{ u_{T-1}\in \mathcal{U}_{T-1}(\xi_{[T-2]})} \mathbb{E}_{|\mathcal{F}_{T-2}}\bigg[ u_{T-1}(h_{T-1}\left(x_{T-1},\xi_{T-1}\right))\nonumber\\ &
&+ V_{T}\left(x_{[T-1]}, \xi_{[T-1]}\right) \Big] \bigg].\label{eq-mpro-tc-dy1}
\end{eqnarray}

\underline{Step 3.}  We show the recursive formula for $t=T-2$.
Let
\bgeqn
f_{T-1}(x_{T-1},\xi_{[T-2]}) &:=&
 \inf\limits_{ u_{T-1}\in \mathcal{U}_{T-1}(\xi_{[T-2]})} \mathbb{E}_{|\mathcal{F}_{T-2}}\bigg[ u_{T-1}(h_{T-1}\left(x_{T-1},\xi_{T-1}\right)) + V_{T}\left(x_{[T-1]}, \xi_{[T-1]}\right)\bigg]\nonumber\\
 &=&
 \inf\limits_{ u_{T-1}\in \mathcal{U}_{T-1}(\xi_{[T-2]})} \mathbb{E}_{|\mathcal{F}_{T-2}}\bigg[ u_{T-1}(h_{T-1}\left(x_{T-1},\xi_{T-1}\right)) \bigg]+  \mathbb{E}_{|\mathcal{F}_{T-2}}\bigg[V_{T}\left(x_{[T-1]}, \xi_{[T-1]}\right)\bigg].\nonumber\\
\edeqn
Observe first that
$V_{T}\left(x_{[T-1]}, \xi_{[T-1]}\right)
:=\max\limits_{x_{T}\in \mathscr{X}_{T}\left(x_{[T-1]}, \xi_{[T-1]}\right)}
f_T(x_T,\xi_{[T-1]})$ is a Carath\'eodory function.
To see this, we note that by assumption (c), the feasible set
$\mathscr{X}_{T}\left(x_{[T-1]}, \xi_{[T-1]}\right)
$
is
Lipschitz continuous w.r.t.~$x_{T-1}$.
Together with the continuity of
$f_T(x_T,\xi_{[T-1]})$ in $x_T$, we obtain by
virtue of \cite[Theorem 1]{Kla87-EC} that $ V_{T}\left(x_{[T-1]}, \xi_{[T-1]}\right)$
is continuous in $x_{[T-1]}$.
The measurability
follows from \cite[Theorem 8.2.11]{AuF09-EC}
since $\mathscr{X}_{T}\left(x_{[T-1]}, \xi_{[T-1]}\right)$ is $\F_{[T-1]}$-measurable
and $f_T(x_T,\xi_{[T-1]})$ is a Carath\'eodory
function.
The conclusion follows since the conditional expectation preserves the above-mentioned continuity and  measurability.
We now show that $
 \inf\limits_{ u_{T-1}\in \mathcal{U}_{T-1}(\xi_{[T-2]})} \mathbb{E}_{|\mathcal{F}_{T-2}}\bigg[ u_{T-1}(h_{T-1}\left(x_{T-1},\xi_{T-1}\right))\bigg] $
 is also a Carath\'eodory function. This can be established following a proof analogous to that of $f_{T}$. Summarizing the discussions above, we
 conclude that $f_{T-1}(x_{T-1},\xi_{[T-2]})$
is a
 Carath\'eodory
function.
Thus the optimization problem at stage $T-1$ can be written as
$$
V_{T-1}\left(x_{[T-2]}, \xi_{[T-2]}\right) :=
\max\limits_{x_{T-1}\in \mathscr{X}_{T-1}\left(x_{[T-2]}, \xi_{[T-2]}\right)}
\mathbb{E}_{|\mathcal{F}_{T-2}}\left[ f_{T-1}(x_{T-1},\xi_{[T-2]})\right ].
$$
Since the feasible set is assumed to be compact and the objective
function is continuous in $x_{T-1}$, the optimal solution exists.
Moreover, the Lipschitz continuity of
$\mathscr{X}_{T-1}\left(x_{[T-2]}, \xi_{[T-2]}\right)$
in $x_{[T-2]}$ and the continuity of
$
\mathbb{E}_{|\mathcal{F}_{T-2}}\left[ f_{T-1}(x_{T-1},\xi_{[T-2]})\right ]$ ensures that
$V_{T-1}\left(x_{[T-2]}, \xi_{[T-2]}\right)$ is continuous
in $x_{[T-2]}$ and for fixed $x_{[T-2]}$,
we can show by \cite[Theorem 8.2.11]{AuF09-EC} that
$V_{T-1}\left(x_{[T-2]}, \xi_{[T-2]}\right)$ is $\F_{[T-2]}$-measurable.

Summarizing from the discussions above, 
the
continuity and measurability can be established in the recursive manner. This shows that the recursive formula (\ref{eq:thm-recursive-formula}) holds.

Since the optimal solutions exist at individual
stages and the recursive formula (\ref{eq:thm-recursive-formula}) holds,
the global optimal solution is also
optimal
to the local problems,
i.e., the time consistency of the policy holds.
\hfill $\Box$
}

\subsection{Proof of Proposition \ref{prop-compact}}

{
Part (i). By definition, $u$ is nondecreasing
over $[a,b]$ with $u(a)=0$, $u(b)=1$,
and both $u$ is globally Lipschitz continuous with a uniformly bounded Lipschitz modulus.
The monotonic increasing property and the normalization condition ensure the
boundedness of the utility functions in the set,
the globally Lipschitz continuity guarantees  equicontinuity
of the class of functions. By  Arzel{\` a}-Ascoli Theorem
(see e.g.~\cite[Theorem 2.3]{Bro14-EC}), $\mathcal{U}^{\mathbb{B}}_t(\xi_{[t-1]})$ is a weakly compact set, that is, it is contained by a compact set
in the space of continuous functions.
To show the compactness of the set, it suffices to show that the set is closed. Let $\{u_k\}\subset \mathcal{U}^{\mathbb{B}}_t(\xi_{[t-1]})$ be a sequence converging to $u$ under some norm topology in the $\mathscr{L}^p$ space. The uniform convergence ensures continuity of $u$.
For any fixed
 points $x,y\in [a,b]$,
$$
 |u_k(x)-u_k(y)|\leq L(\xi_{[t-1]})|x-y|,\ \forall k.
$$
By driving $k$ to infinity, we obtain
 $$
 |u(x)-u(y)|\leq L(\xi_{[t-1]})|x-y|,
$$
 which means that $u$ is also Lipschitz continuous
 with modulus being bounded by $L(\xi_{[t-1]})$.
 Moreover, since $u_k$ is a concave function,
 its limit is also a concave function. This shows
 $u\in \mathcal{U}^{\mathbb{B}}_t(\xi_{[t-1]})$.


Part (ii).
Let us show first that
\begin{eqnarray}
\label{eq:U-normial-dynamic-L}
\mathcal{U}^{L}(\xi_{[t-1]}) :=\left\{u\in \mathscr{U}^c\big|
 \inmat{Lip}(u) \leq L(\xi_{[t-1]})\right\}
\end{eqnarray}
is ${\cal F}_{[t-1]}$-measurable.
To this end, we can rewrite $\mathcal{U}^{L}(\xi_{[t-1]})$ as
\begin{eqnarray*}
&\mathcal{U}^{L}(\xi_{[t-1]})&= \left\{u\in \mathscr{U}^c \ \Big|\ \frac{u(x)-u(y)}{x-y} \leq L(\xi_{[t-1]}),\ \forall x,y\in[a,b], x\neq y \right\}\\
&& = \left\{u\in \mathscr{U}^c \mid \sup_{x,y\in[a,b] } \left({u(x)-u(y)} - L(\xi_{[t-1]})({x-y}) \right)\leq 0 \right\}\\
&& = \left\{u \in \mathscr{U}^c \mid g(u,\xi_{[t-1]})\leq 0 \right\},
\end{eqnarray*}
where
$$
g(u,\xi_{[t-1]}):=\sup_{x,y\in[a,b] } \left({u(x)-u(y)} - L(\xi_{[t-1]})({x-y}) \right).
$$
Since $L(\cdot)$ is assumed to be continuous
and
the function $(u,L)\rightarrow {u(x)-u(y)} - L({x-y})$
is linear in $u$ and $L$, then
$g(u,\xi_{[t-1]})$ is continuous
jointly in $\xi_{[t-1]}$
and $u$.
Let 
$$
\tilde{g}(u,\omega):={g}(u,\xi_{[t-1]}(\omega)).
$$
Then
$\tilde{g}:\Z\times \Omega\rightarrow \R$ is
$\F_{t-1}$ measurable for every $z\in \Z$, and $\tilde{g}(\cdot,\omega)$ is continuous for fixed $\omega$. This shows that  $\tilde{g}:\mathbb{Z}\times \Omega \to \bar{R}$ is a Carath\'eodory function.
By Lemma \ref{lemma-level-meaura},
$\mathcal{U}^{L}(\xi_{[t-1](\cdot)})$
is measurable in $\F_{[t-1]}$.

Next, let
$$
  f_k(u,\omega):= z_{k}(\xi_{[t-1]}(\omega)) \mathbb{E}\left[u\left(Y_{k}\right)
  \right]- z_{k}(\xi_{[t-1]}(\omega))\mathbb{E}\left[u\left(W_{k}\right)
  \right],
  k=1,\ldots,K.\
  $$
  Define level sets
  ${\cal L}_{f_k\leq 0}(\omega):=\{u\in \mathscr{U}^c\mid f_k(u,\omega)\leq 0\}$ and we can rewrite the pairwise comparison ambiguity set as
 $$
 \mathcal{U}^P_t(\xi_{[t-1]})(\omega)
=
 \bigcap_{k=1,\ldots,K} {\cal L}_{f_k\leq 0}(\omega) \bigcap \mathcal{U}^{L}(\xi_{[t-1]})(\omega).
 $$

  Since
  $f_k$ is linear in
  $u$ and measurable w.r.t. $\omega$, then it is a Carath\'edory function.  By Lemma \ref{lemma-level-meaura}, we have that
  ${\cal L}_{f_k\leq 0}$
  is closed-valued and measurable.
  By \cite[Theorem 8.2.4]{AuF09-EC}, the intersection of those sets in
  $\mathcal{U}^P_t(\xi_{[t-1]})(\omega)$
  is closed-valued and measurable.

 The measurability of
 $\mathbb{B}(\tilde{\mathfrak{u}}_t(\cdot,\xi_{[t-1]}),r_t(\xi_{[t-1]})) $
 can be observed by
 \cite[Corollary  8.2.13]{AuF09-EC} 
  given that the center $\tilde{\mathfrak{u}}_t(\cdot,\xi_{[t-1]})$ and the radium $r_t(\xi_{[t-1]})$ are $\F_{t-1}$-measurable.
Thus
$$
\mathcal{U}^{\mathbb{B}}_t(\xi_{[t-1]}) = \mathbb{B}(\tilde{\mathfrak{u}}_t(\cdot,\xi_{[t-1]}),r_t(\xi_{[t-1]})) \bigcap \mathcal{U}^{L}(\xi_{[t-1]})
$$
is measurable.

Part (iii). To show the rectangularity of ${\cal U}$, we recall that
 in Proposition~\ref{prop-rectangular}, we have demonstrated that
the ambiguity set ${\cal U}$ defined
in Definition \ref{def-pro-set}
satisfies the rectangularity.
The key underlying reason is that
${\cal U}$ is
constructed by
a series of conditional ambiguity sets $\mathcal{U}_t(\xi_{[t-1]})$ which satisfies the following three properties:
\begin{itemize}
    \item
$\mathcal{U}_t(\xi_{[t-1]})$ is $\F_{t-1}$-measurable;
  \item
  $\mathcal{U}_t(\xi_{[t-1]})$
  comprises continuous, bounded and
  monotonically increasing utility functions;
    \item for
    given $\xi_{[t-1]}$, $\mathcal{U}_t(\xi_{[t-1]})$ is a compact set.
\end{itemize}
Thus, it suffices to show here
that the ambiguity set
${\cal U}$ constructed through conditional ambiguity sets
$\mathcal{U}^P_t(\xi_{[t-1]})$ and
$\mathcal{U}^{\mathbb{B}}_t(\xi_{[t-1]}) $
satisfies the three properties.
The 
 measurability
is 
addressed in Part (ii) and
the compactnesss is addressed in Part (i).
Continuity, boundedness and monotonicity follow from the definition of
$\mathscr{U}^c$.
Thus, the ambiguity $\cal{U}$ constructed by
$\mathcal{U}^P_t(\xi_{[t-1]})$ or
$\mathcal{U}^{\mathbb{B}}_t(\xi_{[t-1]})$ in the form of \eqref{eq:Ambiguityset-U-SD} satisfy the conditions in Definition \ref{def-pro-set} and is thus 
rectangular. \hfill $\Box$
}




 {
   \begin{lemma}[Measurability of level set-mapping]\label{lemma-level-meaura}
   Let $\mathbb{Z}$ be a Polish space and $(\Omega,\F,\mathbb{P})$ a probability space.
   For a random function $f:\mathbb{Z}\times \Omega \rightarrow \bar{\R}$, i.e., $f(\cdot,\omega)$ is lsc for any fixed $\omega\in\Omega$ and $f(z,\cdot)$ is measurable for any fixed $z\in \mathbb{Z}$,
   the level set mapping ${\cal L}_{f\leq \alpha}: \Omega\rightrightarrows \mathbb{Z}$, defined by ${\cal L}_{f\leq \alpha}(\omega) := \{z\in \mathbb{Z} \mid f(z,\omega)\leq \alpha\}$, is closed-valued and measurable.
   \end{lemma}

  \begin{proof}{Proof:}
  For fixed $\omega$, the closedness of
  ${\cal L}_{f\leq \alpha}(\omega)$
  follows directly from the lsc of $f$ in $z$.
  For any given closed set $Z\in\mathbb{Z}$, we consider a closed-valued mapping $R:\omega\rightarrow Z\times [-\infty,\alpha] $. Thus a constant-valued mapping is naturally measurable. By the
  lsc  of $f(\cdot,\omega)$ for any $\omega$, the epi-mapping ${\rm epi} \; f(\omega):\Omega\rightrightarrows \mathbb{Z}\times \R $
  is closed-valued and measurable (see \cite[Definition 7.35]{SDR14} when
  $z$ is finite dimensional).
  Then we have that
  $$ {\cal L}_{f\leq \alpha}^{-1}(Z)=\left\{\omega\mid {\rm epi} \;f(
  \omega) \bigcap (Z\times [-\infty,\alpha]) \neq \emptyset \right\}=\dom({\rm epi} f \bigcap (Z\times [-\infty,\alpha]) ).
  $$
  By \cite[Theorem 8.2.4]{AuF09-EC},
$\omega\rightarrow {\rm epi}\; f(\omega) \bigcap (Z\times [-\infty,\alpha])$ is a closed-valued and measurable set-valued mapping. Thus, its domain $\dom({\rm epi}\; f \bigcap (Z\times [-\infty,\alpha]) )\in \F$. 
This shows that ${\cal L}_{f\leq\alpha}^{-1}(Z)\in \F$. By \cite[Theorem 8.1.4 (iii)]{AuF09-EC}, we know that the ${\cal L}_{f\leq\alpha}$, as an inverse of
  ${\cal L}_{f\leq \alpha}^{-1}(Z)$,
  is measurable.
  \hfill $\Box$
  \end{proof}
}

\subsection{Proof of Lemma \ref{l-dist-B-B_N-zeta}}

Before presenting a proof for the lemma,
we need the following technical result which is drawn from \cite[Proposition 4.1]{GuX21}  and the proof of \cite[Theorem 4.1]{GuX21}.

\begin{proposition}
\label{p-PC-appr}
Let $u\in \mathcal{U}$.
Assume that
$u(\cdot)$ is
 Lipschitz continuous over an interval $[a,b]$ with modulus $L$
 and $u_N$ is its piecewise linear approximation, that is,
\bgeqn
u_N(y) := u(y_{i-1}) + \frac{u(y_{i})- u(y_{i-1})}{y_i-y_{i-1}} (y-y_{i-1}), \; \inmat{for} \; y\in [y_{i-1},y_i], \; i=2,\cdots,N,
\label{eq:u-N}
\edeqn
where $y_1=a, y_N=b$. Let  $\beta_N:= \max_{i=2,\cdots,N} (y_i-y_{i-1})$.
Then the following assertions hold.

\begin{itemize}
    \item[(i)]
$\|u_N-u\|_{\infty} :=\sup_{y\in [a,b]} |u_N(y)-u(y)| \leq L\beta_N$.

\item[(ii)] For $\mathscr{G}=\mathscr{G}_{L}$,
$
\dd_\mathscr{G}(u,u_N)\leq 2\beta_N.
$
    \item[(iii)] For $\mathscr{G}=\mathscr{G}_{I}$,
$
\dd_\mathscr{G}(u,u_N)\leq L\beta_N.
$

\end{itemize}

\end{proposition}

\noindent
\textbf{Proof.}
We call $u_N$ defined in (\ref{eq:u-N}) a {\em projection} of $u$ on $\mathscr{U}_N$.
Using the proposition, we
are able to derive an upper bound for the Hausdorff distance between $\mathbb{B}_N(u,r)$ and $\mathbb{B}(v,r)$.

It suffices to show that
\bgeqn
\mathbb{D}(\mathbb{B}(u,r_1), \mathbb{B}(v,r_2);   \dd_\mathscr{G}) \leq \dd_\mathscr{G}(u,v) + |r_2-r_1|
\label{eq:H-two-balls-U-1}
\edeqn
and
\bgeqn
\mathbb{D}(\mathbb{B}(v,r_2), \mathbb{B}(u,r_1);   \dd_\mathscr{G}) \leq \dd_\mathscr{G}(u,v) + |r_2-r_1|.
\label{eq:H-two-balls-U-2}
\edeqn
We prove (\ref{eq:H-two-balls-U-1}). The conclusion is trivial if $\mathbb{B}(u,r_1)\subset \mathbb{B}(v,r_2)$, so we consider the case that $\mathbb{B}(u,r_1)\not\subset \mathbb{B}(v,r_2)$. Let
$\tilde{v}\in \mathbb{B}(u,r_1)\backslash \mathbb{B}(v,r_2)$ and $\lambda=r_2/\dd_\mathscr{G}(\tilde{v},v)$. Then $\lambda\in (0,1)$.
Let $v_\lambda = \lambda v+ (1-\lambda)\tilde{v}$. Then
$$
\dd_\mathscr{G}(v_\lambda,v) =\dd_\mathscr{G}((1-\lambda) v+ \lambda \tilde{v},v)
\leq \lambda\dd_\mathscr{G}(\tilde{v},v)=r_2.
$$
This shows $v_\lambda \in \mathbb{B}(v,r_2)$. Thus
\bgeq
\dd_\mathscr{G}(\tilde{v},\mathbb{B}(v,r_2))\leq \dd_\mathscr{G}(\tilde{v},v_\lambda)=(1-\lambda)\dd_\mathscr{G}(\tilde{v},v)=\dd_\mathscr{G}(\tilde{v},v)-r_2 \leq \dd_\mathscr{G}(u,v)+r_1-r_2.
\edeq
Swapping the positions of the two balls in the above discussions, we obtain (\ref{eq:H-two-balls-U-2}). \hfill $\Box$

\textbf{Proof of Lemma \ref{l-dist-B-B_N-zeta}.}
Inequality (\ref{eq:H-two-balls-U-U_N-u-proj-u}) follows straightforwardly from (\ref{eq:H-two-balls-U-U_N}) and
(\ref{eq:u-N-appr-u-zeta-metric}), so we only prove (\ref{eq:H-two-balls-U-U_N}).
 By the triangle inequality,
\bgeq
\mathbb{H}(\mathbb{B}_N(u,r), \mathbb{B}(v,r);   \dd_\mathscr{G})
\leq
\mathbb{H}(\mathbb{B}_N(u,r), \mathbb{B}(u,r);   \dd_\mathscr{G})
+ \dd_\mathscr{G}(u,v).
\edeq
So it suffices to show that
\bgeq
\mathbb{H}(\mathbb{B}_N(u,r), \mathbb{B}(u,r);   \dd_\mathscr{G}) \leq  4\max(2,L)\beta_N.
\edeq
By definition, $\mathbb{B}_N(u,r)\subset \mathbb{B}(u,r)$, so it is enough  to show that
\bgeqn
\mathbb{D}(\mathbb{B}(u,r), \mathbb{B}_N(u,r);   \dd_\mathscr{G}) \leq 4\max(2,L)\beta_N.
\label{eq:D-two-balls-U-U_N}
\edeqn
Let $\epsilon$ be a small positive number and
$u^\epsilon\in \mathbb{B}(u,r)\backslash \mathbb{B}_N(u,r)$ be such that
$$
\dd_\mathscr{G}(u^\epsilon, \mathbb{B}_N(u,r)) \geq \mathbb{D}(\mathbb{B}(u,r), \mathbb{B}_N(u,r); \dd_\mathscr{G}) -\epsilon.
$$
For the given $u^\epsilon$, we may find $u^\epsilon_N\in \mathscr{U}_N$ as that in Proposition \ref{p-PC-appr} such that
\bgeq
\label{eq:u-N-appr-u-zeta-metric}
 \dd_{\mathscr{G}}(u^\epsilon,u^\epsilon_N) \leq \max(2,L)\beta_N.
\edeq
If $u^\epsilon_N\in \mathbb{B}_N(u,r)$, then
$$
\mathbb{D}(\mathbb{B}(u,r), \mathbb{B}_N(u,r); \dd_\mathscr{G})
\leq \dd_\mathscr{G}(u^\epsilon, \mathbb{B}_N(u,r)) +\epsilon\leq
\dd_{\mathscr{G}}(u^\epsilon,u^\epsilon_N)+\epsilon
 \leq \max(2,L)\beta_N+\epsilon
$$
and hence (\ref{eq:D-two-balls-U-U_N}) because $\epsilon$ can be driven to zero.
So we are left with the case that  $u^\epsilon_N\not\in \mathbb{B}_N(u,r)$.
Let $\lambda = \frac{r}{\dd_\mathscr{G}(u^\epsilon_N,u)}$. Then $\lambda\in (0,1)$. Let
$
u_\lambda = \lambda u^\epsilon_N+ (1-\lambda)u.
$
Then $u_\lambda\in\mathscr{U}_N$ and
$
\dd_\mathscr{G}(u_\lambda,u) = \lambda \dd_\mathscr{G}(u^\epsilon_N,u)=r.
$
This shows $u_\lambda\in \mathbb{B}_N(u,r)$. Thus
\bgeq
\dd_\mathscr{G}(u^\epsilon, \mathbb{B}_N(u,r)) &\leq& \dd_\mathscr{G}(u^\epsilon, u_\lambda)
\leq \dd_\mathscr{G}(u^\epsilon, u^\epsilon_N) + \dd_\mathscr{G}(u^\epsilon_N, u_\lambda)\\
&\leq&
2\max(2,L)\beta_N + (1-\lambda)\dd_\mathscr{G}(u^\epsilon_N, u)=
2\max(2,L)\beta_N + \dd_\mathscr{G}(u^\epsilon_N, u)-r\\
&\leq& 2\max(2,L)\beta_N + \dd_\mathscr{G}(u^\epsilon_N, u^\epsilon)+\dd_\mathscr{G}(u^\epsilon, u)-r\leq
4\max(2,L)\beta_N +r-r\\
&=& 4\max(2,L)\beta_N,
\edeq
and hence
$$
\mathbb{D}(\mathbb{B}(u,r), \mathbb{B}_N(u,r); \dd_\mathscr{G}) \leq \dd_\mathscr{G}(u^\epsilon, \mathbb{B}_N(u,r))+\epsilon
$$
which gives (\ref{eq:D-two-balls-U-U_N}) by driving $\epsilon$ to zero. \hfill $\Box$

\subsection{Proof of Theorem \ref{theorem-error-bound}}
 We prove by induction.
Observe that for any $x_{t-1}$ and $\xi_{[t-1]}$, $t=2,\dots,T$,
\bgeq
&&\left|V_{t}\left(x_{[t-1]}, \xi_{[t-1]}\right)-\tilde{V}_{t}\left(x_{[t-1]}, \xi_{[t-1]}\right)\right|\nonumber\\
&\leq&
\max\limits_{x_{t} \in \mathscr{X}_{t}\left(x_{[t-1]}, \xi_{[t-1]}\right)}
\left|\inf
\limits_{ u_t \in \mathcal{U}^{\mathbb{B}}_t(\xi_{[t-1]})}\mathbb{E}_{|\mathcal{F}_{t-1}}\left[ u_{t}\left(h_t(x_{t}, \xi_{t})\right)+V_{t+1}\left(x_{[t]}, \xi_{[t]}\right)\right]\right.\nonumber\\
&&\qquad\qquad
-\left.\inf
\limits_{ u_t \in {\mathcal{U}^{\mathbb{B}_N}_t}(\xi_{[t-1]})}\mathbb{E}_{|\mathcal{F}_{t-1}}
\left[ u_{t}\left(h_t(x_{t}, \xi_{t})\right)+\tilde{V}_{t+1}\left(x_{[t]}, \xi_{[t]}\right)\right]
\right|\nonumber\\
&\leq& \mathbb{H} \left(
\mathbb{B}(  \tilde{\mathfrak{u}}_t(\cdot,\xi_{[t-1]}),r_t(\xi_{[t-1]})) ,
\mathbb{B}_N(\tilde{\mathfrak{u}}^N_t(\cdot,\xi_{[t-1]}),r_t(\xi_{[t-1]})); \dd_{\mathscr{G}}
\right)\nonumber\\
&&+\max\limits_{x_{t} \in \mathscr{X}_{t}\left(x_{[t-1]}, \xi_{[t-1]}\right)} \mathbb{E}_{|\mathcal{F}_{t-1}} \left[\left|V_{t+1}\left(x_{[t]}, \xi_{[t]}\right)
-\tilde{V}_{t+1}\left(x_{[t]}, \xi_{[t]}\right)
\right|\right]
\nonumber\\
&\leq& 
6\max(2,L(\xi_{[t-1]}))\beta_N(\xi_{[t-1]})+
\max\limits_{x_{t} \in \mathscr{X}_{t}\left(x_{[t-1]}, \xi_{[t-1]}\right)}\mathbb{E}_{|\mathcal{F}_{t-1}}\left[\left|
V_{t+1}\left(x_{[t]}, \xi_{[t]}\right)
-\tilde{V}_{t+1}\left(x_{[t]}, \xi_{[t]}\right)
\right|\right],\nonumber\\
\edeq
where
$V_{T+1}=0$ and $\widetilde{V}_{T+1}=0$ and the last inequality follows
from Lemma \ref{l-dist-B-B_N-zeta}.
Assume for stage $t+1$ that
\bgeq
\left|V_{t+1}\left(x_{[t]}, \xi_{[t]}\right)-\tilde{V}_{t+1}\left(x_{[t]}, \xi_{[t]}\right)\right|\leq
\sum_{s=t+1}^T 6\mathbb{E}\left[
\max(2,L(\xi_{[s-1]}))\beta_N(\xi_{[s-1]})
\mid\mathcal{F}_{t}\right],
\edeq
for any fixed $x_t$ and $\xi_{[t]}$.
Then
\bgeq
&&\left|V_{t}\left(x_{[t-1]}, \xi_{[t-1]}\right)-\tilde{V}_{t}\left(x_{[t-1]}, \xi_{[t-1]}\right)\right|\nonumber\\
&\leq& 6
\max(2,L(\xi_{[t-1]}))\beta_N(\xi_{[t-1]})
+ \mathbb{E}_{|\mathcal{F}_{t-1}}\left[\sum_{s=t+1}^T 6\mathbb{E}\left[
\max(2,L(\xi_{[s-1]}))\beta_N(\xi_{[s-1]})
\mid\mathcal{F}_{t}\right]\right]\\
&\leq &\sum_{s=t}^T 6\mathbb{E}\left[
\max(2,L(\xi_{[s-1]}))\beta_N(\xi_{[s-1]})
\mid\mathcal{F}_{t-1}\right],
\edeq
which gives rise to (\ref{eq:errr-bnd-v}).
\hfill $\Box$

\subsection{Proof of Theorem \ref{theorem-dp-sddp}}
The resulting robust dynamic programming equation can be written as
\begin{equation}\label{eq-dynamic-aprox}
\begin{array}{l}
\widetilde{V}_{t}\left(x_{[t-1]}, \xi_{[t-1]}\right) =\max\limits_{x_{t} \in \mathscr{X}_{t}\left(x_{[t-1]}, \xi_{[t-1]}\right)}\inf\limits_{ u_t \in \mathcal{U}^K_t(\xi_{[t-1]}) }\mathbb{E}_{|\mathcal{F}_{t-1}}\left[ u_{t}\left(h_t(x_{t}, \xi_{t})\right)+\widetilde{V}_{t+1}\left(x_{[t]}, \xi_{[t]}\right)\right].
\end{array}
\end{equation}
We can separate the maximin operations by writing
$\widetilde{V}_{t}\left(x_{[t-1]}, \xi_{[t-1]}\right)$ as
\bgeq
\widetilde{V}_{t}\left(x_{[t-1]}, \xi_{[t-1]}\right) = \max\limits_{x_{t} \in \mathscr{X}_{t}\left(x_{[t-1]}, \xi_{[t-1]}\right)} \hat{V}_t\left(x_{[t]}, \xi_{[t]}\right) +\mathbb{E}_{|\mathcal{F}_{t-1}}\left[ \widetilde{V}_{t+1}\left(x_{[t]}, \xi_{[t]}\right)\right],
\edeq
where
\begin{subequations}\label{eq-vt}
\begin{eqnarray}
\hat{V}_t\left(x_{[t]}, \xi_{[t]}\right) := &\inf\limits_{u_t} & \mathbb{E}_{|\mathcal{F}_{t-1}}\left[ u_t\left(h_t(x_{t}, \xi_{t})\right)\right],\\
 &\inmat{s.t.} &  \dd_K(u_t, \tilde{\mathfrak{u}}_t^N(\cdot,\xi_{[t-1]}) )\leq r_t(\xi_{[t-1]}),  \label{eq-vt-kan-ball}\\
 && u_t\in \mathscr{U}_N,   \label{eq-vt-kan-approx}\\
 && \inmat{Lip}(u_t) \leq L,  \label{eq-vt-kan-lip}\\
 && u''_t\leq 0.  \label{eq-vt-concave}
\end{eqnarray}
\end{subequations}
By utilizing the piecewise linear structure of $u$
and setting $\alpha_j=u_t(y_j)$ and $\beta_j=u'_t(y_j)$
at the breakpoints  $y_j$, $j=1,\ldots,N$, we
 can effectively write \eqref{eq-vt} as
\begin{subequations}\label{eq-ec-29}
\begin{eqnarray}
&\hat{V}_t\left(x_{[t]}, \xi_{[t]}\right) := \\
& \inf\limits_{\lambda,\mu,\rho,\phi, \alpha, \beta,\varepsilon, \varphi} & \sum_{i=1}^{S} \mathbb{P}(\xi_t=\xi^i_t|\xi_{[t-1]})
\left(\varepsilon_{i} h_t(x_{t}, \xi^i_{t})+ \varphi_{i} \right)\\
& \inmat{s.t.}  & \frac{1}{2}\sum_{j=2}^N  (\lambda_j + \mu_j + \rho_j + \phi_j)(y_{j}- y_{j-1})^2  \leq r_t(\xi_{[t-1]}) \label{eq-vt2-kan-ball-1} \\
&&  \tilde{\beta}_j -\beta_j + \lambda_j-\mu_j + \rho_j - \phi_j =0,\  j=2,\cdots,N, \label{eq-vt2-kan-ball-2}\\
&& (\mu_{2}-\lambda_{2})(y_{2}-y_1) =0, \label{eq-vt2-kan-ball-3} \\
&& (\mu_{j+1}\!-\! \lambda_{j+1})(y_{j+1}\!-\! y_j)\! +\! (\rho_{j}\!-\! \phi_{j})(y_{j}\!-\! y_{j-1})\! =\! 0, j=2,\! \cdots\!,N\! -\! 1,\label{eq-vt2-kan-ball-4}  \\
&& (\rho_{N}-\phi_{N})(y_{N}-y_{N-1})   =0,  \label{eq-vt2-kan-ball-5}\\
&& \mu_j,\lambda_j,\rho_j,\phi_j \geq 0,\ j=2,\cdots,N. \label{eq-vt2-kan-ball-6}\\
& &{y}_{j} \varepsilon_{i}+\varphi_{i} \geq \alpha_{j}, \ i=1,\ldots,S,\ j =1, \ldots, N, \label{eq-vt2-concave-upper} \\
& & \alpha_{j+1}-\alpha_j= \beta_{j+1}\left({y}_{j+1}-{y}_{j}\right),\  j =1,\ldots,N-1, \label{eq-vt2-dev1} \\
&& \alpha_{j+1}-\alpha_j \geq \beta_{j+2}\left({y}_{j+1}-{y}_{j}\right),\ j =1,\ldots,N-2, \label{eq-vt2-concave} \\
&& 0 \leq \beta_{j+1} \leq L(\xi_{[t-1]}),\  j =1,\ldots,N-1, \label{eq-vt2-lip1} \\
&& \alpha_{1}=0,\ \alpha_{N}=1,\  \varepsilon_{i} \geq 0,\  i =1,\ldots,S, \label{eq-vt2-norm}
\end{eqnarray}
\end{subequations}
where constraints
\eqref{eq-vt2-kan-ball-1}-\eqref{eq-vt2-kan-ball-6}
characterize the Kantorovich ball
\eqref{eq-vt-kan-ball}
as we described in \eqref{eq:Kant-u-v-LP-dual}.
Constraint \eqref{eq-vt2-dev1} characterizes
the piecewise linear structure of $u_t$ in \eqref{eq-vt-kan-approx}
and constraints \eqref{eq-vt2-dev1}-\eqref{eq-vt2-concave}
imply that $\beta_j\geq \beta_{j+1}$ and hence
the concavity of the piecewise linear utility function $u_t$.
Constraint \eqref{eq-vt2-lip1} is concerned with the non-decreasing
property and
Lipschitz continuity of $u_t$
with modules bounded by $L(\xi_{[t-1]})$ as in \eqref{eq-vt-kan-lip}.
As in the literature of PRO models
in one-stage decision making,
the evaluation of the utility function at point $h_t(x_{t}, \xi^i_{t})$
in the objective is carried out by
a linear function passing through
point $(h_t(x_{t}, \xi^i_{t}), u_t(h_t(x_{t}, \xi^i_{t}))$,
with slope $\varepsilon_{i}$ and intercept $\varphi_{i}$.
Constraint \eqref{eq-vt2-concave-upper} requires that all those linear pieces upper bound $u_t$ at those breakpoints. $\tilde{\beta}_j=\frac{\tilde{\mathfrak{u}}_t(y_{j},\xi_{[t-1]})- \tilde{\mathfrak{u}}_t(y_{j-1},\xi_{[t-1]})}{y_j-y_{j-1}}$ is the slope of nominal utility at those breakpoints.
By taking the duality of the linear program \eqref{eq-ec-29}, we obtain
\begin{small}
\begin{eqnarray*}
&\max \quad & \theta_{N-1}+\sum_{i=1}^{S}\mu_{i,N}-L(\xi_{[t-1]}) \sum_{j=1}^{N-1} \eta_{j}
- \sum_{j=2}^N \tilde{\beta}_j w_j -r_t(\xi_{[t-1]}) \varsigma\nonumber\\ &
\text { s.t. } & \sum_{j=1}^{N} {y}_{j} \mu_{i, j} \leq \mathbb{P}(\xi_t=\xi^i_t|\xi_{[t-1]})
 h_t(x_{t}, \xi^i_{t}) ,\ i=1,\ldots,S, \nonumber\\ &
& \frac{p_{i}}{p_{s}}-\sum_{j=1}^{N} \mu_{i,j}=0,\ i=1,\ldots,S,\nonumber\\ &
& \theta_{j-1} {y}_{j-1}-\theta_{j-1} {y}_{j}+v_{j-2}\left({y}_{j-1}-{y}_{j-2}\right)+
w_j+\eta_{j-1} 
\geq 0,\nonumber\\ &
&\ j=3, \cdots, N-1, \nonumber\\ &
& \theta_{1} {y}_{1}-\theta_{1} {y}_{2}+
w_2+\eta_{1} 
\geq 0,\nonumber\\ &
&\theta_{N-1} {y}_{N-1}-\theta_{N-1} {y}_{N}+v_{N-2}\left({y}_{N-1}-{y}_{N-2}\right)+
w_N+\eta_{N-1} 
\geq 0, \nonumber\\ &
&\theta_{j-1}-\theta_{j}+\sum_{i=1}^{S} \mu_{i,j}-v_{j-1}+v_{j}=0,\  j=2, \cdots, N-2, \nonumber\\ &
&\theta_{N-2}-\theta_{N-1}+\sum_{i=1}^{S} \mu_{i,N-1}-v_{N-2}=0, \nonumber\\ &
&  w_j\leq z_{j-1}(y_{j}- y_{j-1})+\frac{1}{2}(y_{j}- y_{j-1})^2 \varsigma ,  j=2,\cdots,N, \nonumber\\ &
&-w_j\leq -z_{j-1}(y_{j}- y_{j-1})+\frac{1}{2}(y_{j}- y_{j-1})^2 \varsigma, j=2,\cdots,N, \nonumber\\ &
&  w_j\leq z_{j}(y_{j}- y_{j-1})+\frac{1}{2}(y_{j}- y_{j-1})^2 \varsigma ,  j=2,\cdots,N, \nonumber\\ &
&-w_j\leq -z_{j}(y_{j}- y_{j-1})+\frac{1}{2}(y_{j}- y_{j-1})^2 \varsigma, j=2,\cdots,N, \nonumber\\ &
&\theta \in \mathbb{R}^{N-1}, v \in \mathbb{R}_+^{N-2}, \eta \in \mathbb{R}_+^{N-1},
\mu \in \mathbb{R}_+^{S \times  N},\varsigma\in \mathbb{R}_+, w\in \mathbb{R}^{N-1}, z\in \mathbb{R}^{N},
\end{eqnarray*}
\end{small}
Taking this duality form back to \eqref{eq-dynamic-aprox} gives the results.
\hfill $\Box$

 \section{An example of time inconsistency}\label{sec-example}

In the Section \ref{sec-robust}, we have demonstrated the rectanglarity of the ambiguity set ${\cal U}$ and subsequently time consistency of problem (MS-PRO-SD).
It is natural  to ask whether the same property is retained
by the ambiguity set of state-independent utility functions and the
robust model (MS-PRO-SID). The answer is no. In this section, we use a counter example to illustrate this fact.

Consider a preference robust counterpart of the stage-wise return rate utility maximization problem in Example \ref{ex-utility-max} with three time points $0,1,2$ and two investment stages $1,2$ between the time points. At each time point, there are two branches 
from the current state with probability $50\%$ each. Thus, we have a two-stage scenario tree with an initial node at time point 0, two nodes at the end of the first stage and four leaf nodes at the end of the second stage.
We assume that there are two risky assets with random excess
return rates ${r}_t=[{r}_t^1,{r}_t^2]$ in range $[0,1]$ at the two stages $t=1,2$. We denote the realization of ${r}_t$ on the $k$-th node at stage $t$ by ${r}_{t,k}$.
We mark the return rates ${r}_{t,k}$ around
the nodes of
each scenario on the scenario tree, see Figure \ref{fig-scenario-tree}.

\begin{figure}[h]
\centering
\begin{tikzpicture}[scale=0.7]
\node [fill=black,circle,,scale=0.8] (0) at (-0.5,0){};
\node [fill=black,circle,,scale=0.8] (1) at (4,1.25) {};
\node [fill=black,circle,,scale=0.8] (2) at (4,-1.25) {};
\node [fill=black,circle,,scale=0.8] (3) at (10, 2.2) {};
\node [fill=black,circle,,scale=0.8] (4) at (10, 0.8) {};
\node [fill=black,circle,,scale=0.8] (5) at (10, -0.8) {};
\node [fill=black,circle,,scale=0.8] (6) at (10, -2.2) {};
\draw[-latex, thick] (0) to (1);
\draw[-latex,thick] (0) to (2);
\draw[-latex,thick] (1) to (3);
\draw[-latex,thick] (1) to (4);
\draw[-latex,thick] (2) to (5);
\draw[-latex,thick] (2) to (6);

\node [left] at (2,1) {$0.5 $ };
\node [left] at (2,-1) {$0.5 $ };
\node [left] at (8,2.2) {$0.5 $ };
\node [left] at (8,0.65) {$0.5 $ };
\node [left] at (8,-0.65) {$0.5 $ };
\node [left] at (8,-2.2) {$0.5 $ };
\node [above] at (0) {$x_1$};
\node [below] at (2) { ${r}_{1,2}=[0.8, 0.2]$ };
\node [above] at (2) {$x_{2,2}$};
\node [below] at (1) { ${r}_{1,1}=[0, 0]$ };
\node [above] at (1) {$x_{2,1}$};
\node [right] at (6) { $r_{2,4}=[1, 0.6]$ };
\node [right] at (5) { $r_{2,3}=[0.4, 0.6]$ };
\node [right] at (4) { $r_{2,2}=[0.6, 0.8]$ };
\node [right] at (3) { $r_{2,1}=[0.6, 0.2]$ };
\end{tikzpicture}
\caption{Branching probability, realizations of two risky assets' return rates and predictable portfolios on the two-stage scenario tree}\label{fig-scenario-tree}
\end{figure}
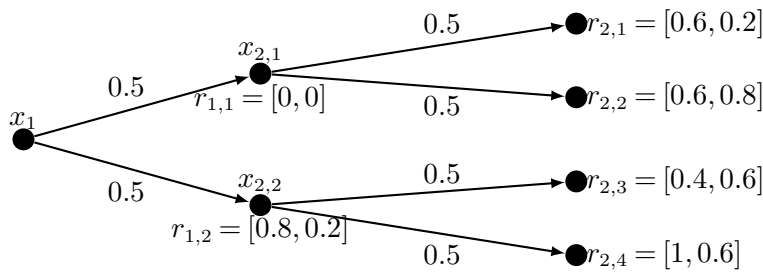

At the beginning of each stage, the investor may reallocate the wealth among the two risky assets. We assume that the portfolio at stage $t$ is $x_t=[x^1_t,x^2_t]$ with $x^1_t+x^2_t=1$, where $x^i_t, t=1, 2, i=1,2$ is the proportion of wealth invested in the $i$-th asset at stage $t$. The first stage portfolio is deterministic while the second stage portfolio is random and scenario dependent.

We assume that the DM is ambiguous about the true utility function
which lies in the ambiguity set:
$$
U=\{ u^1(y):=\min\{3y,0.5y+0.5\},\; u^2(y)=2y-y^2\},
$$
where $u^i(0)=0$, $u^i(1)=1$ for $i=1,2$.
It is easy to see that $u^i$ is
strictly increasing and concave over $[0,1]$
and $U$ is independent of state and stage.

\begin{figure}[h]
\begin{center}
\includegraphics[width=0.5\linewidth]{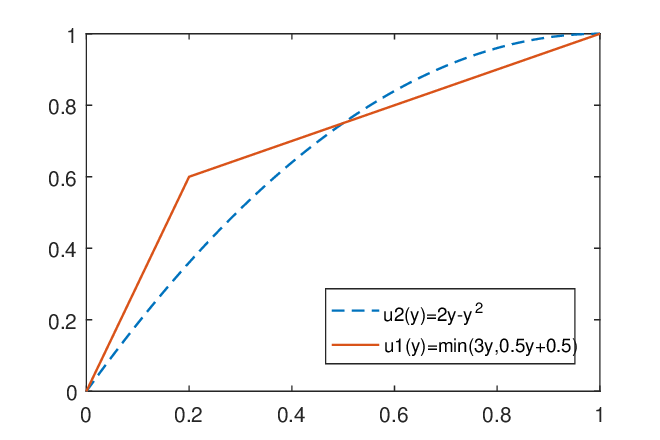}
\caption{\small Plot of $u^1(y)$ and $u^2(y)$.} \label{Figure-proset}
\end{center}
\end{figure}

We consider a simple two-stage portfolio selection problem under the state-independent preference robust expected utility model \eqref{eq-pro-intc}:
\begin{equation}
\label{eq-pro-example}
\begin{array}{cl}
 \max\limits_{x_{1},x_2(\cdot)} & \inf\limits_{{u}_1\in{U},u_2\in U}\mathbb{E}\left[u_1( x_{1}^\top {r}_1 )+
 \bbe\left[u_2( x_{2}^\top {r}_2 )|r_1\right] \right] \\
 {\rm s.t. }& e^\top x_{1}=1,\ x_1\in \mathbb{R}^2_+,
 \ e^\top x_{2}({r}_1)=1,\ x_2(\cdot)\in \mathcal{L}^0(\mathbb{R}^2_+),
\end{array}
\end{equation}
where $\mathcal{L}^0(\mathbb{R}^2_+)$ denotes
the space of measurable functions taking
finite values in $\mathbb{R}^2_+$,
and discuss how the worst-case utility function is identified at each 
investment stage.

\subsection{Non-rectangularity of the preference robust counterpart}
We begin by investigating rectangularity of the ambiguity set in problem \eqref{eq-pro-example}, which is essentially about the consistency between the global preference robust counterpart
\begin{equation}\label{eq-ex-glo}
\begin{array}{cl}
f^*(x) :=\inf\limits_{{u}_1\in{U},u_2\in U}\mathbb{E}\left[u_1( x_{1}^\top {r}_1 )+
 \bbe\left[u_2( x_{2}^\top {r}_2 )|{r}_1\right] \right]
\end{array}
\end{equation}
with global worst-case utility functions $u_1^*,u_2^*$ and the nested local preference robust counterpart
\begin{equation}\label{eq-ex-loc}
\begin{array}{cl}
\hat{f}^*(x) :=\inf\limits_{{u}_1\in{U}}\mathbb{E}\left[u_1( x_{1}^\top {r}_1 )+ \inf\limits_{u_2\in U}
 \bbe\left[u_2( x_{2}^\top {r}_2 )|{r}_1\right] \right],
\end{array}
\end{equation}
 with local worst-case utility functions $\hat{u}_1^*,\hat{u}_2^*(\cdot)$.
Note that in both problems \eqref{eq-ex-glo} and \eqref{eq-ex-loc},
the decision variables are fixed. Here we set
$x_1=[1,0]$, $x_{2,1}=[1,0]$ and
$x_{2,2}=[1,0]$
and demonstrate that the
worst-case utility functions of the two problems
are different at some state in the second stage.
Since both problems
are decomposable,
we may solve them by solving
$f^*_1=\inf\limits_{{u}_1\in{U}}\mathbb{E}\left[u_1( x_{1}^\top {r}_1 )\right] $,
$f^*_2=\inf\limits_{u_2\in U}\mathbb{E}\left[
 \bbe\left[u_2( x_{2}^\top {r}_2 )|r_1\right] \right] $,
$\hat{f}^*_2({r}_1)= \inf\limits_{u_2\in U}
 \bbe\left[u_2( x_{2}^\top {r}_2 )|{r}_1\right] $
and then setting
$f^*=f^*_1+f^*_2$ and
\begin{equation*}
\hat{f}^*=f^*_1+\bbe[\hat{f}^*_2({r}_1)].
\end{equation*}
For instance,
\bgeq
f^*_1 & =&\inf\limits_{{u}_1\in{U}}\frac{1}{2}\left[u_1( [1,0] \times [0,0]^\top )+u_1( [1,0] \times [0.8,0.2]^\top )\right]\\
&=&\inf\limits_{{u}_1\in\{u^{1},u^{2}\}}\frac{1}{2}\left[u_1( 0 )+u_1( 0.8 )\right]\\
& =& 0 + \min\{ \frac{1}{2}  \min\{3*0.8 ,0.5*0.8+0.5 \} , \frac{1}{2} (2*0.8-0.8^2 )  \}\\
&=&\min \{ 0.45,0.48 \}=0.45.
\edeq
The worst-case utility value is attained by $u^1(\cdot)$.
Likewise
\bgeq
f^*_2 &= & \inf\limits_{u_2\in U}\frac{1}{2}\Big[
\frac{1}{2} \left[u_2( [1,0]\times [0.6,0.2]^\top ) + u_2( [1,0]\times [0.6,0.8]^\top ) \right]\\
&& + \frac{1}{2} \left[u_2( [1,0]\times [0.4,0.6]^\top ) + u_2( [1,0]\times [1,0.6]^\top ) \right] \Big]\\
&=& \inf\limits_{u_2\in \{u^{1},u^{2}\}}\frac{1}{2}\left[
\frac{1}{2} \left[u_2( 0.6 ) + u_2( 0.6 ) \right] +\frac{1}{2} \left[u_2( 0.4 ) + u_2( 1 ) \right] \right]\\
&=& \min\left\{\frac{1}{2}(0.8+0.85),\frac{1}{2}(0.84+0.82)\right\} =\min\{0.825,0.83\}=0.825.
\edeq
The worst-case utility value is attained by $u^1(\cdot)$.
Summing them up, we obtain
$f^*= f^*_1+ f^*_2=1.275$.
The analysis is depicted at the left-hand side of
Figure \ref{fig:rectangular} where ``PLU''  denotes
the piecewise linear utility function.
We now move on to calculate $\hat{f}^*$.
\bgeq
\hat{f}^*_2({r}_{1,1}) &=& \inf\limits_{u_2\in U}
\frac{1}{2} \left[u_2( [1,0]\times [0.6,0.2]^\top ) + u_2( [1,0]\times [0.6,0.8]^\top ) \right]\\
&=& \inf\limits_{u_2\in \{u^{1},u^{2}\}} \frac{1}{2} \left[u_2( 0.6 ) + u_2( 0.6 ) \right]\\
&=& \min\{0.8,0.84\}=0.8.
\edeq
The worst-case utility (locally) at the second-stage is
attained by $u^1(\cdot)$
in the first node at stage 1. 
\bgeq
\hat{f}^*_2({r}_{1,2}) &=& \inf\limits_{u_2\in U}
 \frac{1}{2} \left[u_2( [1,0]\times [0.4,0.6]^\top ) + u_2( [1,0]\times [1,0.6]^\top ) \right]\\
&=& \inf\limits_{u_2\in \{u^{1},u^{2}\}} \frac{1}{2} \left[u_2( 0.4 ) + u_2( 1 ) \right]\\
&=&\min\{0.85,0.82\}=0.82.
\edeq
 The worst-case utility (locally)
 is attained by  $u^2(\cdot)$ in the second node at stage 1. 
Consequently
$$
\hat{f}^*_2= \frac{1}{2}(0.8+0.82)=0.81<0.825=f^*_2.
$$

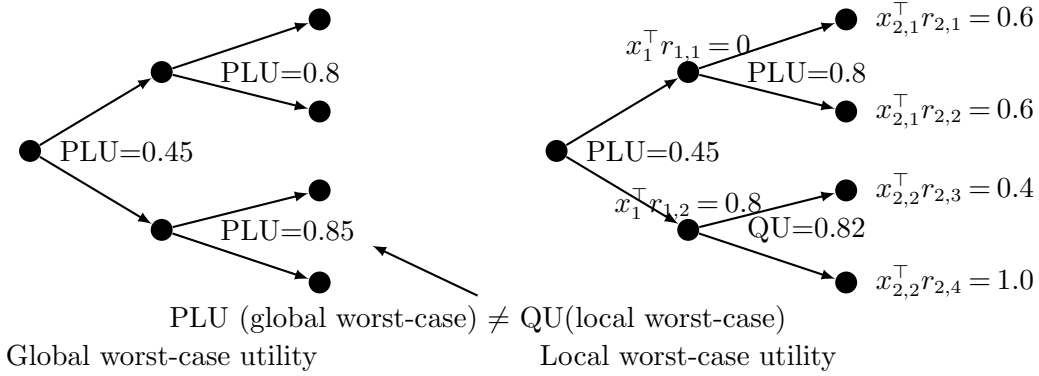
\begin{figure}[h]
\centering
\begin{tikzpicture}[scale=0.7]
\node [fill=black,circle,,scale=0.8] (0) at (-0.5,0){};
\node [fill=black,circle,,scale=0.8] (1) at (2,1.5) {};
\node [fill=black,circle,,scale=0.8] (2) at (2,-1.5) {};
\node [fill=black,circle,,scale=0.8] (3) at (5, 2.5) {};
\node [fill=black,circle,,scale=0.8] (4) at (5, 0.75) {};
\node [fill=black,circle,,scale=0.8] (5) at (5, -0.75) {};
\node [fill=black,circle,,scale=0.8] (6) at (5, -2.5) {};
\draw[-latex, thick] (0) to (1);
\draw[-latex,thick] (0) to (2);
\draw[-latex,thick] (1) to (3);
\draw[-latex,thick] (1) to (4);
\draw[-latex,thick] (2) to (5);
\draw[-latex,thick] (2) to (6);
\node (9) at (8,-3.2) {\color{black}{PLU (global worst-case)
$\neq$
QU(local worst-case)}};
\draw[-latex,thick,black] (8,-2.75) to (6,-1.8);

\node [right] at (0) {$\ \ $PLU=0.45};
\node [right] at (1) {$\ \quad\ $PLU=0.8};
\node [right] at (2) {$\ \quad\ $PLU=0.85};

\node [fill=black,circle,,scale=0.8] (10) at (10-0.5,0){};
\node [fill=black,circle,,scale=0.8] (11) at (10+2,1.5) {};
\node [fill=black,circle,,scale=0.8] (12) at (10+2,-1.5) {};
\node [fill=black,circle,,scale=0.8] (13) at (10+5, 2.5) {};
\node [fill=black,circle,,scale=0.8] (14) at (10+5, 0.75) {};
\node [fill=black,circle,,scale=0.8] (15) at (10+5, -0.75) {};
\node [fill=black,circle,,scale=0.8] (16) at (10+5, -2.5) {};
\draw[-latex, thick] (10) to (11);
\draw[-latex,thick] (10) to (12);
\draw[-latex,thick] (11) to (13);
\draw[-latex,thick] (11) to (14);
\draw[-latex,thick] (12) to (15);
\draw[-latex,thick] (12) to (16);

\node [above] at (11) {$x_{1}^\top {r}_{1,1}=0$ };
\node [above] at (12) {$x_{1}^\top {r}_{1,2}=0.8$ };
\node [right]  at (13) {$\ \ x_{2,1}^\top {r}_{2,1}=0.6 $};
\node [right]  at (14) {$\ \ x_{2,1}^\top {r}_{2,2}=0.6 $};
\node [right]  at (15) {$\ \ x_{2,2}^\top {r}_{2,3}=0.4 $};
\node [right]  at (16) {$\ \ x_{2,2}^\top {r}_{2,4}=1.0 $};
\node [right] at (10) {$\ \ $PLU=0.45};
\node [right] at (11) {$\ \quad\ $PLU=0.8};
\node [right] at (12) {$\ \quad\ $QU=0.82};
\node [below] at (2,-3.5) {Global worst-case utility};
\node [below] at (12,-3.5) {Local worst-case utility};
\end{tikzpicture}
\caption{
Worst-case utilities on a two-stage scenario tree.
}
\label{fig:rectangular}
\end{figure}

From the analysis above, we can see that the DM would
adopt a quadratic utility (QU) function $u^2(\cdot)$ which is more risk-averse
after she/he has earned some money (second node at stage 1)
but would take a piecewise linear utility function $u^1(\cdot)$
after she/he has failed to earn anything (first node at stage 1).
The analysis is depicted at the right-hand side of Figure \ref{fig:rectangular}.
The overall worst-case expected utility
value in the two stages
is $ \hat{f}^* = f^*_1 +\hat{f}^*_2 = f^*_1 + \frac{1}{2}(\hat{f}^*_{2,1}+\hat{f}^*_{2,2}) = 1.26$.

Summarizing the calculations of both problems  \eqref{eq-ex-glo} and \eqref{eq-ex-loc}, we conclude that
$
\hat{f}^* = 1.26<1.275=f^*.
$
This is because  model \eqref{eq-ex-glo}
chooses the worst-case utility function independent of scenarios
in the second stage
whereas model \eqref{eq-ex-loc} chooses the worst-case utility
function after observing the scenarios and hence is more conservative.
The underlying reason is that the utility functions in the ambiguity
set are state-independent.

\subsection{Time inconsistency of the preference robust optimization model}

We now turn to discuss time consistency of the preference robust optimization problem \eqref{eq-pro-example}. Let
\begin{equation}
\label{eq-ex-tc-glo}
\{x_{1}^*,x_2^*(\cdot)\}=\mathop{\arg\max}\limits_{x_{1}\in X_1,x_2(\cdot)\in X_2}\inf\limits_{{u}_1\in{U},u_2\in U}\mathbb{E}\left[u_1( x_{1}^\top {r}_1 )+
 \bbe\left[u_2( x_{2}^\top {r}_2 )|r_1\right] \right]
\end{equation}
and
\begin{equation}\label{eq-ex-tc-loc}
\begin{array}{cl}
\hat{x}_2^*(\cdot)=\mathop{\arg\max}\limits_{x_2\in X_2}\inf\limits_{u_2\in U}
 \bbe\left[u_2( x_{2}^\top {r}_2 )|r_1\right],
\end{array}
\end{equation}
where $X_1=\{ x_1\in \mathbb{R}^2_+ | x^1_{1}+x^2_1=1 \}$,
$X_2= \{ x_2(\cdot)\in \mathcal{L}^0(\mathbb{R}^2_+) | x^1_{2}(r_1)+x^2_2(r_1)=1\}$.
We want to show that $x_2^*(\cdot)\neq \hat{x}_2^*(\cdot)$.
Observe that due to the decomposable structure of problem \eqref{eq-ex-tc-glo},
\begin{equation}
\label{eq-ex-tc-glo-b}
\{x_{1}^*,x_2^*(\cdot)\}=
\left\{ \mathop{\arg\max}\limits_{x_{1}\in X_1}\inf\limits_{{u}_1\in{U}}\mathbb{E}\left[u_1( x_{1}^\top {r}_1 )\right],
  \mathop{\arg\max}\limits_{x_2(\cdot)\in X_2}\inf\limits_{u_2\in U} \bbe\left[\bbe\left[u_2( x_{2}^\top {r}_2 )|r_1\right]\right] \right\}.
\end{equation}

For the first-stage optimization problem in \eqref{eq-ex-tc-glo-b}, the optimal portfolio is always ${x}_1^*=[1,0]$
as $r^1\geq r^2$ in both scenarios and the utility function does not
affect the optimal choice, given that the utility function is increasing.
To identify the worst-case utility,  let us compare
the optimal values $u_1( x_{1}^\top {r}_1 )$ under the two utility functions.
It is easy to obtain that
the optimal value is $0.48$ under the quadratic utility function
and $0.45$ under the piecewise linear utility function.
Thus the worst-case utility function $u^*_1$
at the first-stage is $u^1(\cdot)$.

Let us now look at the second-stage optimization problem in \eqref{eq-ex-tc-glo-b}.
By the finiteness of the preference robust set $U$ and the scenario tree structure of the random return ${r}$,
the second-stage optimization problem in \eqref{eq-ex-tc-glo-b} can be formulated as
\begin{equation}\label{eq-ex-tc-r1}
\begin{array}{ccl}
v^*_2&=\max\limits_{z,x_2(\cdot)\in X_2} & z\\
&{\rm s.t. } & z\leq \bbe\left[\bbe\left[ \min\{ 3 x_{2}^\top {r}_2 , 0.5x_{2}^\top {r}_2+0.5\}  |r_1\right]\right], \\
&& z\leq \bbe\left[\bbe\left[ 2x_{2}^\top {r}_2 - (x_{2}^\top {r}_2)^2 \}  |r_1\right]\right]\\
&=\max\limits_{x_2,y,z} & z\\
&{\rm s.t. } & z\leq \frac{1}{4}\sum_{i=1}^4 y_i, \\
&& y_i \leq  3 x_{2,1}^\top {r}_{2,i} ,\ i=1,2,\\
&& y_i \leq 0.5 x_{2,1}^\top {r}_{2,i}  + 0.5,\ i=1,2,\\
&& y_i \leq  3 x_{2,2}^\top {r}_{2,i} ,\ i=3,4,\\
&& y_i \leq 0.5 x_{2,2}^\top {r}_{2,i}  + 0.5,\ i=3,4,\\
&& z\leq  \frac{1}{4}\sum_{i=1}^2[ 2 x_{2,1}^\top {r}_{2,i} - ( x_{2,1}^\top {r}_{2,i})^2 ]
+ \frac{1}{4}\sum_{i=3}^4[ 2 x_{2,2}^\top {r}_{2,i} - ( x_{2,2}^\top {r}_{2,i})^2 ],\\
&& z\in \mathbb{R},y\in \mathbb{R}^4,x_{2}\in \mathbb{R}_+^{2\times 2}, x^1_{2,i}+x^2_{2,i}=1,\ i=1,2,\\
\end{array}
\end{equation}
by adding some auxiliary variables.
Problem \eqref{eq-ex-tc-r1} is a convex quadratic constrained quadratic programming problem which can be solved efficiently by CVX in Matlab.

Alternatively, we can solve \eqref{eq-ex-tc-r1} in a closed-form.
As the return rates in all scenarios are larger than 0.2,
thus $3 x_{2}^\top {r}_2 \geq 0.5x_{2}^\top {r}_2+0.5$
in all scenarios.
Then we can reformulate
\eqref{eq-ex-tc-r1} as
\begin{small}
\begin{equation*}
\begin{array}{cl}
v^*_2=\max\limits_{x_{2,1},x_{2,2}} &\min \Bigg\{\frac{1}{2}\bigg( 0.5 x_{2,1}^\top \left[\begin{array}{c} 0.6 \\ 0.5 \end{array}\right] +0.5 +  0.5 x_{2,2}^\top \left[\begin{array}{c} 0.7 \\ 0.6 \end{array}\right] +0.5 \bigg) ,\\
& \frac{1}{2}\Bigg( \frac{1}{2}\bigg( 2 x_{2,1}^\top \left[\begin{array}{c} 0.6 \\ 0.2 \end{array}\right] -\left(x_{2,1}^\top \left[\begin{array}{c} 0.6 \\ 0.2 \end{array}\right]\right)^2\bigg) + \frac{1}{2}\bigg( 2 x_{2,1}^\top \left[\begin{array}{c} 0.6 \\ 0.8 \end{array}\right] -\left(x_{2,1}^\top \left[\begin{array}{c} 0.6 \\ 0.8 \end{array}\right]\right)^2 \bigg) \Bigg) \\
& +\frac{1}{2}\Bigg( \frac{1}{2}\bigg( 2 x_{2,2}^\top \left[\begin{array}{c} 0.4\\ 0.6 \end{array}\right] -\left(x_{2,2}^\top \left[\begin{array}{c} 0.4 \\ 0.6 \end{array}\right]\right)^2\bigg) + \frac{1}{2}\bigg( 2 x_{2,2}^\top \left[\begin{array}{c} 1 \\ 0.6 \end{array}\right] -\left(x_{2,2}^\top \left[\begin{array}{c} 1 \\ 0.6 \end{array}\right]\right)^2 \bigg) \Bigg) \Bigg\} \\
=\max\limits_{x_{2,1},x_{2,2}} &\min \bigg\{\frac{1}{4}\Big( x_{2,1}^\top \left[\begin{array}{c} 0.6 \\ 0.5 \end{array}\right] +  x_{2,2}^\top \left[\begin{array}{c} 0.7 \\ 0.6 \end{array}\right] \Big) + 0.5,\\
& x_{2,1}^\top \left[\begin{array}{c} 0.6 \\ 0.5 \end{array}\right] - x_{2,1}^\top \left[\begin{array}{cc} 0.18 &  0.15 \\ 0.15 & 0.17   \end{array}\right] x_{2,1}  +
x_{2,2}^\top \left[\begin{array}{c} 0.7 \\ 0.6 \end{array}\right] - x_{2,2}^\top \left[\begin{array}{cc}  0.29 & 0.21\\ 0.21  & 0.18   \end{array}\right]  x_{2,1} \bigg\} \\
{\rm s.t. } & x^1_{2,i}+x^2_{2,i}=1,\ x_{2,i}\in [0,1]^2,\ i=1,2.
\end{array}
\end{equation*}
\end{small}
By eliminating
variables $x^2_{2,i}$, $i=1,2$, we can obtain a reduced maximin problem
\begin{small}
\begin{equation*}
\begin{array}{cl}
v^*_2=\max\limits_{x^1_{2,1},x^1_{2,2}} &\min \Big\{ 0.025(x^1_{2,1}+x^1_{2,2})+0.775,  - 0.05(x^1_{2,1})^2 + 0.14 x^1_{2,1}  - 0.05(x^1_{2,2})^2 + 0.04x^1_{2,2} +0.75 \Big\}\\
{\rm s.t. } & x^1_{2,1}\in [0,1], x^1_{2,2}\in [0,1],
\end{array}
\end{equation*}
\end{small}
where
$x^2_{2,i}=1-x^1_{2,i}$, $i=1,2$.
This is a maximization problem with a piecewise quadratic objective function. The optimum is attained
potentially at two sets of points:
the global maximizers of each piece,
and the set of points where the two pieces intersect, that is,
$$
v^*_2= \max\{ \min\{ v^*_{linear},v^*_{quad}\}, v^*_{int}  \},
$$
where
\begin{equation}\label{ex-tc-final-linear}
\begin{array}{ccl}
v^*_{linear}=&\max\limits_{x^1_{2,1},x^1_{2,2}} & 0.025(x^1_{2,1}+x^1_{2,2})+0.775,  \\
& {\rm s.t.} & x^1_{2,1}\in [0,1], x^1_{2,2}\in [0,1],
\end{array}
\end{equation}
\begin{equation}\label{ex-tc-final-quad}
\begin{array}{ccl}
v^*_{quad}= & \max\limits_{x^1_{2,1},x^1_{2,2}} & - 0.05(x^1_{2,1})^2 + 0.14 x^1_{2,1}  - 0.05(x^1_{2,2})^2 + 0.04x^1_{2,2} +0.75 \\\
& {\rm s.t.} &  x^1_{2,1}\in [0,1], x^1_{2,2}\in [0,1],
\end{array}
\end{equation}
and
\bgeqn
v^*_{int} = \max\limits_{x^1_{2,1},x^1_{2,2}} &&  0.025(x^1_{2,1}+x^1_{2,2})+0.775 \nonumber\\
 {\rm s.t. } &&  0.025(x^1_{2,1}+x^1_{2,2})+0.775 \nonumber\\
 &&\quad =  - 0.05(x^1_{2,1})^2 + 0.14 x^1_{2,1}  - 0.05(x^1_{2,2})^2 + 0.04x^1_{2,2} +0.75,\nonumber\\
&& x^1_{2,1}\in [0,1], x^1_{2,2}\in [0,1].
\label{ex-tc-final-int}
\edeqn
Problem \eqref{ex-tc-final-linear}
achieves its maximum
at the boundary $x^1_{2,1}=1$,  $x^1_{2,2}=1$ with $v^*_{linear}=0.825$.
Problem \eqref{ex-tc-final-quad}
achieves its maximum at the boundary of $x^1_{2,1}=1$ and stationary point of $x^1_{2,2}=0.4$
with $v^*_{quad}=0.848$.
Problem \eqref{ex-tc-final-int}
attains the maximum at the intersection point
$x^1_{2,1}=0.8$,  $x^1_{2,2}=1$
with $v^*_{int}=0.82$.
Thus $v^*_2= \max\{ \min\{ v^*_{linear},v^*_{quad}\}, v^*_{int}  \}=0.825$ with the
optimal solution $x^*_{2,1}=[1,0]$,  $x^*_{2,2}=[1,0]$.
The PRO model has a piecewise linear worst-case
utility function at its optimum $u^*_2$.
The analysis is depicted at the left-hand side of Figure \ref{fig:tc}.



We now turn to discuss solution of PRO problem \eqref{eq-ex-tc-loc}.
Suppose that at the beginning of the second-stage, the DM
can predict different scenarios
that would occur at the end of the second stage. Then the DM may consider the sub-PRO problem \eqref{eq-ex-tc-loc} at the second stage, which may have different optimal solutions and
corresponding worst-case utility functions at the two different nodes.

As there are two nodes
at the end of the first stage,
we have to solve the two sub-optimization problems conditional on the historical information on the two nodes at stage 1,
i.e.,
\begin{small}
\begin{equation}\label{ex-eq-tc-opt-1}
\begin{array}{cl}
&\hat{v}_2^*(r_{1,1})\\
=&\max\limits_{x_2\in X_2}\inf\limits_{u_2\in U}
 \bbe\left[u_2( x_{2}^\top {r}_2 )|r_{11}\right]\\
=& \max\limits_{
[x_{2,1},x_{2,2}]\in X_2 }
\min \bigg\{ 0.5 x_{2,1}^\top \left[\begin{array}{c} 0.6 \\ 0.5 \end{array}\right] +0.5,
2 x_{2,1}^\top \left[\begin{array}{c} 0.6 \\ 0.5 \end{array}\right]
- \frac{1}{2}\left(x_{2,1}^\top \left[\begin{array}{c} 0.6 \\ 0.2 \end{array}\right]\right)^2
-\frac{1}{2}\left(x_{2,1}^\top \left[\begin{array}{c} 0.6 \\ 0.8 \end{array}\right]\right)^2 \bigg\} \\
=&  \max\limits_{x^1_{2,1}\in [0,1]} \min \Big\{ 0.05 x^1_{2,1} +0.75,  - 0.1(x^1_{2,1})^2 + 0.28 x^1_{2,1}  +0.66 \Big\},
\end{array}
\end{equation}
\end{small}
where $x_{2,1}^2=1-x_{2,1}^1$
and
\begin{small}
\begin{equation}\label{ex-eq-tc-opt-2}
\begin{array}{cl}
&\hat{v}_2^*(r_{1,2})\\
=&\max\limits_{x_2\in X_2}\inf\limits_{u_2\in U}
 \bbe\left[u_2( x_{2}^\top {r}_2 )|r_{12}\right]\\
=&  \max\limits_{[x_{2,1},x_{2,2}]\in X_2} \min
\Big\{ 0.5 x_{2,2}^\top \left[\begin{array}{c} 0.7 \\ 0.6 \end{array}\right] +0.5,
2x_{2,2}^\top \left[\begin{array}{c} 0.7 \\ 0.6 \end{array}\right]
- \frac{1}{2}\left(x_{2,2}^\top \left[\begin{array}{c} 0.4 \\ 0.6 \end{array}\right]\right)^2
-\frac{1}{2}\left(x_{2,2}^\top \left[\begin{array}{c} 1 \\ 0.6 \end{array}\right]\right)^2  \Big\} \\
=&\max\limits_{x^1_{2,2}\in [0,1]} \min \Big\{ 0.05 x^1_{2,2} + 0.8,  - 0.1(x^1_{2,2})^2 + 0.08x^1_{2,2} +0.84 \Big\}.
\end{array}
\end{equation}
\end{small}
Problem \eqref{ex-eq-tc-opt-1} has an optimal solution $\hat{x}_{2,1}^*=[1,0]$
with $\hat{v}_2^*(r_{1,1})=0.8$.
The corresponding worst-case utility $u^*_{2,1}$ is $u^1(\cdot)$.
Problem \eqref{ex-eq-tc-opt-2} has an optimal solution
$\hat{x}_{2,2}^*=[0.8,0.2]$
with $\hat{v}_2^*(r_{1,2})=0.84$.
At the optimum, the expected utility values of $u^1(\cdot)$ and
$u^2(\cdot)$ are the same.
The analysis is depicted at the right-hand side of Figure \ref{fig:tc}.


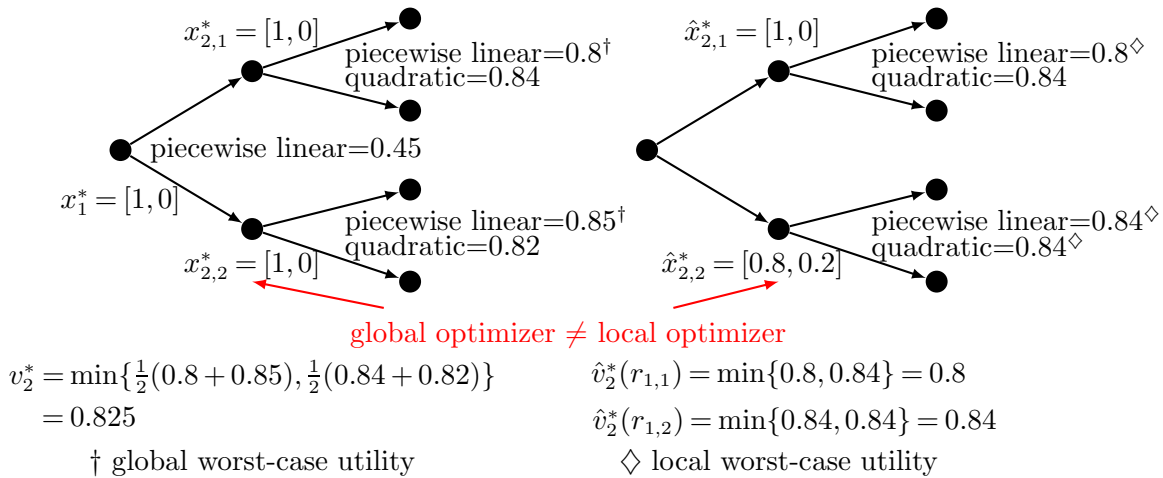
\begin{figure}[h]
\centering
\begin{tikzpicture}[scale=0.7]
\node [fill=black,circle,,scale=0.8] (0) at (-0.5,0){};
\node [fill=black,circle,,scale=0.8] (1) at (2,1.5) {};
\node [fill=black,circle,,scale=0.8] (2) at (2,-1.5) {};
\node [fill=black,circle,,scale=0.8] (3) at (5, 2.5) {};
\node [fill=black,circle,,scale=0.8] (4) at (5, 0.75) {};
\node [fill=black,circle,,scale=0.8] (5) at (5, -0.75) {};
\node [fill=black,circle,,scale=0.8] (6) at (5, -2.5) {};
\draw[-latex, thick] (0) to (1);
\draw[-latex,thick] (0) to (2);
\draw[-latex,thick] (1) to (3);
\draw[-latex,thick] (1) to (4);
\draw[-latex,thick] (2) to (5);
\draw[-latex,thick] (2) to (6);
\node (9) at (8,-3.5) {\color{red}{global optimizer $\neq$ local optimizer }};
\draw[-latex,thick,red] (4.5,-3) to (2,-2.5);
\draw[-latex,thick,red] (10,-3) to (12,-2.5);

\node [below] at (-0.5,-0.5) {$x^*_1=[1,0]$};
\node [above] at (2,1.7) {$x^*_{2,1}=[1,0]$ };
\node [below] at (2,-1.7) {$x^*_{2,2}=[1,0]$ };
\node [right] at (0) {$\ \ $piecewise linear=0.45};
\node [right] at (3.2,1.85) {$\ \ $piecewise linear=0.8$^\dagger$};
\node [right] at (3.2,1.35) {$\ \ $quadratic=0.84};

\node [right] at (3.2,-1.35) {$\ \ $piecewise linear=0.85$^\dagger$};
\node [right] at (3.2,-1.85) {$\ \ $quadratic=0.82};

\node [fill=black,circle,,scale=0.8] (10) at (10-0.5,0){};
\node [fill=black,circle,,scale=0.8] (11) at (10+2,1.5) {};
\node [fill=black,circle,,scale=0.8] (12) at (10+2,-1.5) {};
\node [fill=black,circle,,scale=0.8] (13) at (10+5, 2.5) {};
\node [fill=black,circle,,scale=0.8] (14) at (10+5, 0.75) {};
\node [fill=black,circle,,scale=0.8] (15) at (10+5, -0.75) {};
\node [fill=black,circle,,scale=0.8] (16) at (10+5, -2.5) {};
\draw[-latex, thick] (10) to (11);
\draw[-latex,thick] (10) to (12);
\draw[-latex,thick] (11) to (13);
\draw[-latex,thick] (11) to (14);
\draw[-latex,thick] (12) to (15);
\draw[-latex,thick] (12) to (16);

\node [above] at (11.5,1.7) {$\hat{x}^*_{2,1}=[1,0]$  };
\node [below] at (11.5,-1.7) {$\hat{x}^*_{2,2}=[0.8,0.2]$  };
\node [right] at (13.2,1.85) {$\ \ $piecewise linear=0.8$^\diamondsuit$};
\node [right] at (13.2,1.35) {$\ \ $quadratic=0.84};
\node [right] at (13.2,-1.35) {$\ \ $piecewise linear=0.84$^\diamondsuit$};
\node [right] at (13.2,-1.85) {$\ \ $quadratic=0.84$^\diamondsuit$};

\node [below] at (2,-3.8) {
$v_2^*=\min\{\frac{1}{2}(0.8+0.85),\frac{1}{2}(0.84+0.82)\}$};
\node [below] at (-1.1,-4.7) {
$=0.825$};
\node [below] at (12,-3.8) {$\hat{v}_2^*(r_{1,1})=\min\{0.8,0.84\}=0.8$};
\node [below] at (12.3,-4.7) { $\hat{v}_2^*(r_{1,2})=\min\{0.84,0.84\}=0.84$};
\node [below] at (2,-5.5) {$\dagger$ global worst-case utility};
\node [below] at (12,-5.5) {$\diamondsuit$ local worst-case utility};

\end{tikzpicture}
\caption{
Left: optimal solutions of (\ref{eq-ex-tc-glo}). \quad
Right:  optimal solutions of (\ref{eq-ex-tc-loc}).
}
\label{fig:tc}
\end{figure}

By comparing the solutions shown in Figure \ref{fig:tc}, we can see that the global optimal solution at the left-hand side and the local optimal solution at the right-hand side are not the same. Thus, the optimal solution of the state-independent PRO model \eqref{eq-pro-intc} is not time
consistent.
This is because in model (\ref{eq-ex-tc-glo}) the worst-case utility $u_{2,1}^*$ must be equal to $u^*_{2,2}$ regardless of the reward at the end of stage one. In contrast, model (\ref{eq-ex-tc-loc}) allows one to choose worst-case utility $u_{2,1}^*$ or $u_{2,2}^*$ after viewing the outcome of reward at the end of stage one.
The fundamental reason is that the worst-case utilities in sub-horizon
model \eqref{eq-pro-intc-t}
are scenario ($\mathcal{F}_{t-1}$-adapted $\xi_{[t-1]}$) dependent
whereas the worst-case utilities in \eqref{eq-pro-intc}
are all deterministic (independent of the stochastic process $\{\xi_t\}$).

\section{Reformulations of the multistage PRO models under a scenario tree structure}\label{sec-app-3-tree}

\subsection{Time-consistent model with pairwise comparison-based ambiguity set}\label{sec-sce-pariwise}


If 
the state-dependent pairwise comparisons ambiguity set $\mathcal{U}^P_t(\xi_{[t-1]})$ defined in Section 4.1 is adopted, we can apply the tractable
reformulation of the one-stage PRO model with pairwise comparison proposed in \cite{AmD15-EC} to each non-leaf node of problem \eqref{eq-mpro-scenario} and get the following reformulation of the time consistent model.

\begin{proposition}\label{prop-pairwise-tree}
Given the scenario tree structure of $\{\xi_{t}\}$ and a series of pairwise comparisons ambiguity sets $\mathcal{U}^{P}(s)=\mathcal{U}^{P}_{t(s)}(\xi[s])$,
problem \eqref{eq-mpro-scenario} is equivalent to
\begin{eqnarray}
&\max &  \sum_{s\in S^-}p_s \bigg( \theta_{N-1}(s)+\sum_{i\in s^+}\mu_{i,N} -L(s) \sum_{j=1}^{N-1} \eta_{j}(s)
+ \sum_{k=1}^{K}  z_k(s)\left(\mathbb{P}\left[Y_{k}={y}_{N}\right] -  \mathbb{P}\left[W_{k}={y}_{N}\right] \right) \lambda_k(s) \bigg) \nonumber\\ &
\inmat{s.t.} & \sum_{j=1}^{N} {y}_{j} \mu_{i, j} \leq \frac{p_i}{p_{i^-}}
 h_{t(i)}(x(i^-), \xi(i)),\ i\in S\setminus\{1\}, \nonumber\\ &
& \frac{p_{i}}{p_{i^-}}-\sum_{j=1}^{N} \mu_{i,j}=0,\ i\in S\setminus\{1\},\nonumber\\ &
& \theta_{j-1}(s) {y}_{j-1}-\theta_{j-1}(s) {y}_{j}+v_{j-2}(s)\left({y}_{j-1}-{y}_{j-2}\right)+\eta_{j-1}(s)
\geq 0,\ j=3, \cdots, N-1,\ s\in S^-,\nonumber\\ &
& \theta_{1}(s) {y}_{1}-\theta_{1}(s) {y}_{2}
+\eta_{1}(s) 
\geq 0,\ s\in S^-,\nonumber\\ &
&\theta_{N-1}(s) {y}_{N-1}-\theta_{N-1}(s) {y}_{N}+v_{N-2}(s)\left({y}_{N-1}-{y}_{N-2}\right)
+\eta_{N-1}(s) 
\geq 0, \ s\in S^-, \label{eq-sce-reformu2}\\ &
&\theta_{j-1}(s)-\theta_{j}(s)+\sum_{i\in s^+} \mu_{i,j}-v_{j-1}(s)+v_{j}(s) \nonumber\\ &
&\qquad + \sum_{k=1}^{K}  z_k(s)\left(\mathbb{P}\left[Y_{k}={y}_{j}\right] -  \mathbb{P}\left[W_{k}={y}_{j}\right] \right)\lambda_k(s) =0,\ j=2, \cdots, N-2,\ s\in S^-, \nonumber\\ &
&\theta_{N-2}(s)-\theta_{N-1}(s)+\sum_{i\in s^+} \mu_{i,N-1}-v_{N-2}(s) \nonumber\\ &
&\qquad + \sum_{k=1}^{K}  z_k(s)\left(\mathbb{P}\left[Y_{k}={y}_{N-1}\right] -  \mathbb{P}\left[W_{k}={y}_{N-1}\right] \right)\lambda_k(s) =0,\ s\in S^-, \nonumber\\ &
&x({1}) \in \mathscr{X}_{1}, x(s) \in \mathscr{X}_{t(s)}\left(x[{s^-}],\xi[{s}]\right),\ s\in S^-\setminus\{1\},\nonumber\\ &
&\theta(s) \in \mathbb{R}^{N-1}, v(s) \in \mathbb{R}_+^{N-2}, \eta(s) \in \mathbb{R}_+^{N-1},
\lambda(s)\in \mathbb{R}_+^{K},\ s\in S^-,  
\mu(s) \in \mathbb{R}_+^{N}, s\in S\setminus\{1\}.\nonumber
\end{eqnarray}
\end{proposition}

Given the concavity of $h_{t}\left(\cdot, \cdot\right)$, $t=1,\ldots,T$, and the
convexity of $\mathscr{X}_{1}(\cdot)$,  $\mathscr{X}_{t}\left(\cdot, \cdot\right)$, $t=2,\ldots,T$,
problem \eqref{eq-sce-reformu2} is a convex programming problem.

{\it Proof of Proposition \ref{prop-pairwise-tree}: }
Analogous to the proof of Theorem 1 in \cite{AmD15-EC},
we can show that the worst-case utility function
is in a piecewise linear form with at most $N$ breakpoints.
Let $S(s)=|s^{+}|$, here $s^+$ stands for the set of all son nodes of $s$.
By taking the piecewise linear form in the functional infimum problem, we have
\begin{eqnarray*}
&\inf\limits_{u_{s} \in \mathcal{U}^P(s)}& \sum\limits_{i \in s^{+}} \frac{p_{i}}{p_{s}} u_{s}\left(h_{t(i)}\left(x\left(s\right), \xi(i)\right)\right)\\
=& \inf\limits_{\alpha, \beta,\varepsilon, \varphi} & \sum_{i=1}^{S(s)} \mathbb{P}(\xi_t=
{\xi(i)}
|\xi_{[t-1]})
\left(\varepsilon_{i} h_{t(i)}\left(x(s)
, \xi(i)
+ \varphi_{i} \right) \right)\\
& \inmat{s.t.}  & z_k(s)\sum_{j=2}^{N} \mathbb{P}\left[W_{k}={y}_{j}\right] \alpha_{j} \geq z_k(s)\sum_{j=2}^{N} \mathbb{P}\left[Y_{k}={y}_{j}\right] \alpha_{j},\ k=1, \ldots, K, \\
&& \inmat{constraints} \; \eqref{eq-vt2-concave-upper}-\eqref{eq-vt2-norm}.
\end{eqnarray*}
The only difference between the studied model and the model in Theorem 1 of \cite{AmD15-EC} is that, we replace the normalization constraint in \cite{AmD15-EC} by bounded support constraints.

Taking the duality to the minimization LP problem gives an equivalent
maximization LP reformulation.
Applying the maximization LP reformulation to each inner infimum problem at node $s$ in \eqref{eq-mpro-scenario}, we obtain the deterministic reformulation of \eqref{eq-mpro-scenario}.
\hfill $\Box$

\subsection{Time-consistent model with Kantorovich ball based ambiguity set}

If 
the state-dependent Kantorovich ball-based ambiguity set studied in Section \ref{sec-tractable-sddp} is adopted,
we can 
apply the tractable reformulation of the dynamic programming equation obtained in Theorem~\ref{theorem-dp-sddp} 
to the scenario tree recursively.

\begin{proposition}\label{theorem-scenario-tree-ball}
Given the scenario tree structure of $\{\xi_{t}\}$ and
a series of Kantorovich ball based ambiguity sets $\mathcal{U}^{K}(s)=\mathcal{U}^{K}_{t(s)}(\xi[s])$ on each node $s\in S^-$ of the scenario tree, program \eqref{eq-mpro-scenario}
can be reformulated as
\begin{eqnarray}
&\max &  \sum_{s\in S^-}p_s \bigg( \theta_{N-1}(s)+\sum_{i\in s^+}\mu_{i,N} -L(s) \sum_{j=1}^{N-1} \eta_{j}(s) 
-\sum_{j=2}^N \tilde{\beta}_j(s) w_j(s) -r(s)\varsigma(s) \bigg) \nonumber\\ &
\inmat{s.t.} & \sum_{j=1}^{N} {y}_{j} \mu_{i, j} \leq \frac{p_i}{p_{i^-}}
 h_{t(i)}(x(i^-), \xi(i)),\ i\in S\setminus\{1\}, \nonumber\\ &
& \frac{p_{i}}{p_{i^-}}-\sum_{j=1}^{N} \mu_{i,j}=0,\ i\in S\setminus\{1\},\nonumber\\ &
& \theta_{j-1}(s) {y}_{j-1}-\theta_{j-1}(s) {y}_{j}+v_{j-2}(s)\left({y}_{j-1}-{y}_{j-2}\right)+
w_j(s) +\eta_{j-1}(s)
\geq 0,\ j=3, \cdots, N-1,\ s\in S^-,\nonumber\\ &
& \theta_{1}(s) {y}_{1}-\theta_{1}(s) {y}_{2}+
w_2(s)+\eta_{1}(s)
\geq 0,\ s\in S^-,\nonumber\\ &
&\theta_{N-1}(s) {y}_{N-1}-\theta_{N-1}(s) {y}_{N}+v_{N-2}(s)\left({y}_{N-1}-{y}_{N-2}\right)+
w_N(s)+\eta_{N-1}(s) 
\geq 0, \ s\in S^-, \nonumber\\ &
&\theta_{j-1}(s)-\theta_{j}(s)+\sum_{i\in s^+} \mu_{i,j}-v_{j-1}(s)+v_{j}(s)=0,\ j=2, \cdots, N-2,\ s\in S^-, \nonumber\\ &
&\theta_{N-2}(s)-\theta_{N-1}(s)+\sum_{i\in s^+} \mu_{i,N-1}-v_{N-2}(s)=0,\ s\in S^-, \nonumber\\ &
&  w_j(s)\leq z_{j-1}(s)(y_{j}- y_{j-1})+\frac{1}{2}(y_{j}- y_{j-1})^2 \varsigma(s) , \ j=2,\cdots,N,\ s\in S^-, \nonumber\\ &
&-w_j(s)\leq -z_{j-1}(s)(y_{j}- y_{j-1})+\frac{1}{2}(y_{j}- y_{j-1})^2 \varsigma(s),\ j=2,\cdots,N,\ s\in S^-, \label{reformu-sce-kantorovich}\\ &
&  w_j(s)\leq z_{j}(s)(y_{j}- y_{j-1})+\frac{1}{2}(y_{j}- y_{j-1})^2 \varsigma(s) ,\  j=2,\cdots,N,\ s\in S^-,\nonumber\\ &
&-w_j(s)\leq -z_{j}(s)(y_{j}- y_{j-1})+\frac{1}{2}(y_{j}- y_{j-1})^2 \varsigma(s),\ j=2,\cdots,N,\ s\in S^-, \nonumber\\ &
&x({1}) \in \mathscr{X}_{1}, x(s) \in \mathscr{X}_{t(s)}\left(x[{s^-}],\xi[{s}]\right),\ s\in S^-\setminus\{1\},\nonumber\\ &
&\theta(s) \in \mathbb{R}^{N-1},\ v(s) \in \mathbb{R}_+^{N-2},\ \eta(s) \in \mathbb{R}_+^{N-1},\
s\in S^-,  \nonumber\\ &
&\varsigma(s)\in \mathbb{R}_+,\ w(s)\in \mathbb{R}^{N-1},\ z(s)\in \mathbb{R}^{N},\ s\in S^-,\ \mu(s) \in \mathbb{R}_+^{N},\ s\in S\setminus\{1\}.\nonumber
\end{eqnarray}
\end{proposition}

Under the concavity of $h_{t}\left(\cdot, \cdot\right)$, $t=1,\ldots,T$, and the
convexity of $\mathscr{X}_{1}(\cdot)$,  $\mathscr{X}_{t}\left(\cdot, \cdot\right)$, $t=2,\ldots,T$,
problem \eqref{reformu-sce-kantorovich} is a convex programming problem.


{\it Proof of Proposition \ref{theorem-scenario-tree-ball}: }
There are two ways to obtain the reformulation.
One is to use the duality technique in the proof
of Theorem~\ref{theorem-dp-sddp}, on the basis of the finite expansion reformulation \eqref{eq-mpro-scenario},
to reformulate the inner minimization
over $\mathcal{U}(s)$ as the sum of $|S^-|$ maximization problems.
The other is to use Theorem \ref{theorem-dp-sddp} recursively from the last stage to the first stage, and then we derive the desired conclusion.
\hfill $\Box$

{
\subsection{Time-inconsistent model with pairwise comparison-based ambiguity set }

If we consider the state-independent pairwise comparison based ambiguity set ${U}^{P}_{t}:=\mathcal{U}^{P}_{t}(\xi_0)$ which is fixed
at each stage $t=1,\ldots,T$, we can apply the tractable
reformulation of the one-stage PRO model to each stage of problem \eqref{eq-pro-intc} and get the following reformulation.

\begin{proposition}[Pairwise comparison based ambiguity set]
\label{theorem-scenario-tree-pairwise}
Given the scenario tree structure of $\{\xi_{t}\}$ and $T$ pairwise comparison based state-independent ambiguity sets ${U}^{P}_{t}:=\mathcal{U}^{P}_{t}(\xi_0)$ at each stage $t=1,\ldots,T$, program \eqref{eq-pro-intc}
can be reformulated as
\begin{eqnarray*}
&\max &  \sum_{t=1}^{T} \bigg( \theta_{N-1}(t)+\sum_{i\in S(t)}\mu_{i,N} -L(t) \sum_{j=1}^{N-1} \eta_{j}(t)
+ \sum_{k=1}^{K}  z_k(t)\left(\mathbb{P}\left[Y_{k}={y}_{N}\right] -  \mathbb{P}\left[W_{k}={y}_{N}\right] \right) \lambda_k(t) \bigg) \nonumber\\ &
\inmat{s.t.} & \sum_{j=1}^{N} {y}_{j} \mu_{i, j} \leq {p_i}
 h_{t(i)}(x(i^-), \xi(i)),\ i\in S\setminus\{1\}, \nonumber\\ &
& {p_{i}}-\sum_{j=1}^{N} \mu_{i,j}=0,\ i\in S\setminus\{1\},\nonumber\\ &
& \theta_{j-1}(t) {y}_{j-1}-\theta_{j-1}(t) {y}_{j}+v_{j-2}(t)\left({y}_{j-1}-{y}_{j-2}\right)+\eta_{j-1}(t)
\geq 0,\ j=3, \cdots, N-1,\ t=1,\ldots,T,\nonumber\\ &
& \theta_{1}(t) {y}_{1}-\theta_{1}(t) {y}_{2}
+\eta_{1}(t) \geq 0,\ t=1,\ldots,T,\nonumber\\ &
&\theta_{N-1}(t) {y}_{N-1}-\theta_{N-1}(t) {y}_{N}+v_{N-2}(t)\left({y}_{N-1}-{y}_{N-2}\right)
+\eta_{N-1}(t) \geq 0, \ t=1,\ldots,T, \nonumber\\ &
&\theta_{j-1}(t)-\theta_{j}(t)+\sum_{
i\in S(t)} \mu_{i,j}-v_{j-1}(t)+v_{j}(t) \nonumber\\ &
&\qquad + \sum_{k=1}^{K}  z_k(t)\left(\mathbb{P}\left[Y_{k}={y}_{j}\right] -  \mathbb{P}\left[W_{k}={y}_{j}\right] \right)\lambda_k(t) =0,\ j=2, \cdots, N-2,\ t=1,\ldots,T, \nonumber\\ &
&\theta_{N-2}(t)-\theta_{N-1}(t)+\sum_{
i\in S(t)} \mu_{i,N-1}-v_{N-2}(t) \nonumber\\ &
&\qquad + \sum_{k=1}^{K}  z_k(t)\left(\mathbb{P}\left[Y_{k}={y}_{N-1}\right] -  \mathbb{P}\left[W_{k}={y}_{N-1}\right] \right)\lambda_k(t) =0,\ t=1,\ldots,T, \nonumber\\ &
&x({1}) \in \mathscr{X}_{1}, x(s) \in \mathscr{X}_{t(s)}\left(x[{s^-}],\xi[{s}]\right),\ s\in S^-\setminus\{1\},\nonumber\\ &
&\theta(t) \in \mathbb{R}^{N-1}, v(t) \in \mathbb{R}_+^{N-2}, \eta(t) \in \mathbb{R}_+^{N-1},
\lambda(t)\in \mathbb{R}_+^{K},\ t=1,\ldots,T,\
\mu(s) \in \mathbb{R}_+^{N}, s\in S\setminus\{1\}.
\end{eqnarray*}
\end{proposition}

The
main difference from \eqref{eq-sce-reformu2} is that, the slack variables $\theta(t)$, $v(t)$, $\eta(t)$ and $\lambda(t)$, $t=1,\ldots,T$, are stage-dependent in the time inconsistent model as they are added to determine the worst-case state-independent  utilities.
In contrast, the slack variables $\theta(s)$, $v(s)$, $\eta(s)$ and $\lambda(s)$, $s=1,\ldots,S^-$, are node-dependent in the time consistent model as they are added at each non-leaf node to determine the worst-case state-dependent utilities.
The convexity of the reformulation follows by the concavity of $h_{t}\left(\cdot, \cdot\right)$, $t=1,\ldots,T$, and the
convexity of $\mathscr{X}_{1}(\cdot)$,  $\mathscr{X}_{t}\left(\cdot, \cdot\right)$, $t=2,\ldots,T$.

\subsection{Time-inconsistent model with Kantorovich ball-based ambiguity set }

Finally, we apply the state-independent Kantorovich ball-based ambiguity set ${U}^{K}_{t}:=\mathcal{U}^{K}_{t}(\xi_0)$ to each stage $t=1,\ldots,T$ of problem \eqref{eq-pro-intc} and
obtain the following reformulation.


\begin{proposition}[Kantorovich ball based ambiguity set]
\label{theorem-scenario-tree-ball-sid}
Given the scenario tree structure of $\{\xi_{t}\}$ and $T$ Kantorovich ball based  state-independent ambiguity sets ${U}^{K}_{t}:=\mathcal{U}^{K}_{t}(\xi_0)$ at each stage $t=1,\ldots,T$, program \eqref{eq-pro-intc}
can be reformulated as
\begin{eqnarray*}
&\max &  \sum_{t=1}^{T}\bigg( \theta_{N-1}(t)+\sum_{i\in S(t)}\mu_{i,N} -L(t) \sum_{j=1}^{N-1} \eta_{j}(t)
-\sum_{j=2}^N \tilde{\beta}_j(t) w_j(t) -r(t)\varsigma(t) \bigg) \nonumber\\ &
\inmat{s.t.} & \sum_{j=1}^{N} {y}_{j} \mu_{i, j} \leq {p_i}
 h_{t(i)}(x(i^-), \xi(i)),\ i\in S\setminus\{1\}, \nonumber\\ &
& {p_{i}}-\sum_{j=1}^{N} \mu_{i,j}=0,\ i\in S\setminus\{1\},\nonumber\\ &
& \theta_{j-1}(t) {y}_{j-1}-\theta_{j-1}(t) {y}_{j}+v_{j-2}(t)\left({y}_{j-1}-{y}_{j-2}\right)+
w_j(t) +\eta_{j-1}(t)
\geq 0,\ j=3, \cdots, N-1,\ t=1,\ldots,T,\nonumber\\ &
& \theta_{1}(t) {y}_{1}-\theta_{1}(t) {y}_{2}+
w_2(t)+\eta_{1}(t)
\geq 0,\ t=1,\ldots,T,\nonumber\\ &
&\theta_{N-1}(t) {y}_{N-1}-\theta_{N-1}(t) {y}_{N}+v_{N-2}(t)\left({y}_{N-1}-{y}_{N-2}\right)+
w_N(t)+\eta_{N-1}(t)
\geq 0, \ t=1,\ldots,T, \nonumber\\ &
&\theta_{j-1}(t)-\theta_{j}(t)+\sum_{i\in s^+} \mu_{i,j}-v_{j-1}(t)+v_{j}(t)=0,\ j=2, \cdots, N-2,\ t=1,\ldots,T, \nonumber\\ &
&\theta_{N-2}(t)-\theta_{N-1}(t)+\sum_{i\in s^+} \mu_{i,N-1}-v_{N-2}(t)=0,\ t=1,\ldots,T, \nonumber\\ &
&  w_j(t)\leq z_{j-1}(t)(y_{j}- y_{j-1})+\frac{1}{2}(y_{j}- y_{j-1})^2 \varsigma(t) , \ j=2,\cdots,N,\ t=1,\ldots,T, \nonumber\\ &
&-w_j(t)\leq -z_{j-1}(t)(y_{j}- y_{j-1})+\frac{1}{2}(y_{j}- y_{j-1})^2 \varsigma(t),\ j=2,\cdots,N,\ t=1,\ldots,T, \nonumber\\ &
&  w_j(t)\leq z_{j}(t)(y_{j}- y_{j-1})+\frac{1}{2}(y_{j}- y_{j-1})^2 \varsigma(t) ,\  j=2,\cdots,N,\ t=1,\ldots,T,\nonumber\\ &
&-w_j(t)\leq -z_{j}(t)(y_{j}- y_{j-1})+\frac{1}{2}(y_{j}- y_{j-1})^2 \varsigma(t),\ j=2,\cdots,N,\ t=1,\ldots,T, \nonumber\\
&
&x({1}) \in \mathscr{X}_{1}, x(s) \in \mathscr{X}_{t(s)}\left(x[{s^-}],\xi[{s}]\right),\ s\in S^-\setminus\{1\},\nonumber\\ &
&\theta(t) \in \mathbb{R}^{N-1},\ v(t) \in \mathbb{R}_+^{N-2},\ \eta(t) \in \mathbb{R}_+^{N-1},\
t=1,\ldots,T,  \nonumber\\ &
&\varsigma(t)\in \mathbb{R}_+,\ w(t)\in \mathbb{R}^{N-1},\ z(t)\in \mathbb{R}^{N},\ t=1,\ldots,T,\ \mu(s) \in \mathbb{R}_+^{N},\ s\in S\setminus\{1\}.
\end{eqnarray*}
\end{proposition}
The convexity of the reformulation follows by the concavity of $h_{t}\left(\cdot, \cdot\right)$, $t=1,\ldots,T$, and the
convexity of $\mathscr{X}_{1}(\cdot)$,  $\mathscr{X}_{t}\left(\cdot, \cdot\right)$, $t=2,\ldots,T$.
}

{
\section{NBD 
algorithm and SDDP algorithm}\label{sec-dy-algo}

The scenario tree method 
is a generic solution approach which can handle nonlinear dependence structure between stages. However, it does not
exploit the dynamic programming 
structure
of the time consistent model
established in Theorem \ref{theorem-dp-reform}.
Here, we propose to
use the efficient DP-type methods
such as the NBD algorithm and SDDP algorithm to solve the state-dependent
multistage PRO models based on the recursive equations in Theorem \ref{theorem-dp-reform}.

\subsection{General principle of the
DP-type 
algorithm}



For simple problems with finite states and finite actions, we can apply tabular solution methods by maintaining a state/policy-to-value mapping table and updating it with value/policy iteration schemes. 

For problems with infinite actions (for instance, a polyhedral feasible set $\mathscr{X}$), a natural idea is to approximate the value function by a piecewise linear function,
and then use the optimal values obtained
from solving state-dependent problem in Theorem \ref{theorem-dp-sddp}
to update the approximation function.
This 
is known as the approximate dynamic programming approach.

For some particular problems, for instance in our MS-PRO problem, if the constraints at recourse stages have a linear block-diagonal  structure, i.e., only consecutive stages can be linked by linear constraints, meanwhile, the reward functions are linear, and we have applied the piecewise linear approximation to the value function, then the optimization problem \eqref{eq-mpro-tc-nomial-dy-PLA} in the dynamic programming equation is convex and thus strong duality holds. We can use the solution to its duality problem to generate some optimality cuts with tight approximation gap and good convergence property.
This is known the Benders' style algorithm.




\subsubsection*{\underline{Approximate the expected value operator.} }
To solve problem \eqref{eq:thm-recursive-formula},
we first need to estimate or approximate the expected value operator.
At stage $t$ and scenario $k$, we select $S_{t,k}$
samples of $\xi_t$, denoted by $\xi_t(s)$ with appearing probability $p_s$, $s=1,\ldots,S_{t,k}$.
If we have a scenario tree representation 
of $\xi_t$, we can use all the son nodes of $\xi^k_t$ as the samples, which is exact reformulation of \eqref{eq:app-forward}. This is known as the NBD algorithm.
When 
$\xi_t$ is continuously distributed or have a large number of realizations, we can draw finite i.i.d. samples instead to get a small approximation problem, which is known as the SDDP algorithm.
With the finite samples of $\xi_t$, we can reformulate/approximate the expected value operator, thus problem \eqref{eq:thm-recursive-formula} as
\bgeqn
\left. \begin{array}{ll}
\max _{x_{t}} & \inf _{u_{t} \in \mathcal{U}_{t}(\xi_{[t-1]}^k)}
\sum_{s=1}^{S_{t,k}}p_s
\left[u_{t}\left(h_{t}\left(x_{t}, \xi_{t}(s)\right)\right)+V^{i-1}_{t+1}\left(
x_{[t-1]}^k,x_{t}, \xi_{[t-1]}^{k},\xi_{t}\right)\right] \\
\text { s.t. } &
{W_{t-1}\left(\xi_{[t-1]}^{k}\right) x_{t}=b_{t-1}\left(\xi_{[t-1]}^{k}\right)-D_{t-1}\left(\xi_{[t-1]}^{k}\right) x_{t-1}^{ik}}.
\end{array}\right.
\label{eq:app-forward-finite}
\edeqn

\subsubsection*{\underline{Piecewise linear approximations} }
The
main challenge of a dynamic programming
type algorithm is to
find a good approximation to $V_{t+1}\left(x_{[t]}, \xi_{[t]}\right)$.
The Benders'
type algorithm
uses a piecewise linear approximation to $V_{t+1}\left(x_{[t]}, \xi_{[t]}\right)$,
denoted by
$V^{i-1}_{t+1}\left(x_{[t]}, \xi_{[t]}\right)$,
after $(i-1)$-th iteration/updating.
Under Assumption \ref{assum-linear-recourse}, the dynamic equation has a block diagonal structure.
Thus, the cost-to-go value function only depends on the current decision $x_t$
rather than
historical decisions $x_{[t-1]}$. Moreover, we consider finite scenarios for $\xi_{[t-1]}$ and finite samples for $\xi_t$,
which means that we can maintain a piecewise linear approximation function
in $x_t$ for each realization of
$\xi_{[t]}$.
Denote the approximation function by
$$
V_{t+1}(x_t;\xi_{[t]}):= \min_{r\in R^{i}(\xi_{[t-1]})}\left(
\beta_{t}^r(\xi_{[t]})^\top x_t + \alpha_{t+1}^r(\xi_{[t]})\right),
$$
where
$R^{i}(\xi_{[t]})$ is the index set of
linear pieces at the $i$-th iteration.
$\alpha_{t+1}^r(\xi_{[t]})$ is the intercept and $\beta_{t+1}^r(\xi_{[t]})$ is the slope of the $r$-th piece.

\subsubsection*{\underline{Forward pass}}
At the $i$-th iteration, with the approximation of the value function at the 
previous iteration, we solve approximately \eqref{eq:app-forward-finite} in sequence in some scenario to obtain some trial decision 
sequences.
Specifically,
for each scenario $k$ in
a selected scenario set $\mathcal{K}$,
we
solve the following problem
from stage $1$ to stage $T$,
\begin{align}
\max _{x_{t},\delta_{t+1}} & \left(\inf _{u_{t} \in \mathcal{U}_{t}(\xi_{[t-1]}^k) } \sum_{s=1}^{S_{t,k}}p_s
\left[u_{t}\left(h_{t}\left(x_{t}, \xi_{t}(s)\right)\right)\right] \right) + \sum_{s=1}^{S_{t,k}}p_s
\left[\delta_{t+1}^s\right] \nonumber\\
\text { s.t. } &
{W_{t-1}\left(\xi_{[t-1]}^{k}\right) x_{t}=b_{t-1}\left(\xi_{[t-1]}^{k}\right)-D_{t-1}\left(\xi_{[t-1]}^{k}\right) x_{t-1}^{ik}}:\ \pi_t^k \label{eq:app-forward}\\
&\delta_{t+1}^s\leq {\beta_{t+1}^r}(\xi_{[t-1]}^k,\xi_{t}(s))^\top x_t + {\alpha_{t+1}^r}(\xi_{[t-1]}^k,\xi_{t}(s)): \ \rho^{ikrs}_t, \
s=1,\ldots,S_{t,k},\
r\in R(\xi_{[t-1]}^k,\xi_t), \nonumber
\end{align}
for each realization of historical 
path $\xi_{[t-1]}^{k}$,
historical decisions
$x_{[t-1]}^k$ and $R(\xi_{[t-1]}^k,\xi_t) =R^{i-1}(\xi_{[t-1]}^k,\xi_t) $. We denote the optimal solution of \eqref{eq:app-forward}
by $x_{t}^k$, which is the trial decision at stage $t$.
Here $\pi_t^k$ is the optimal dual variable
of the dynamic balance equation and $\rho_t$ is the optimal dual variable 
of the linear cuts.

\subsubsection*{\underline{Kantorovich ball-based ambiguity set case.}}
Under the Kantorovich ball ambiguity set, we can apply Theorem \ref{theorem-dp-sddp} to get a reformulation of \eqref{eq:app-forward}:
\begin{eqnarray}
&\max & \quad \theta_{N-1}+\sum_{s=1}^{S_{t,k}}\left(\mu_{i,N}+
p_s
\delta_s \right)-L(\xi_{[t-1]}) \sum_{j=1}^{N-1} \eta_{j} \nonumber\\
&&\quad -\tilde{L}(\xi_{[t-1]}) \sum_{j=1}^{N-2}\left(\tau_{j}+\sigma_{j}\right)\left({y}_{j+2}-{y}_{j}\right)  -\sum_{j=2}^N \tilde{\beta}_j w_j - r_t(\xi_{[t-1]}) \varsigma \nonumber \\
&\text { s.t. } &
{ W_{t-1}\left(\xi_{[t-1]}^{k}\right) x_{t}=b_{t-1}\left(\xi_{[t-1]}^{k}\right)-D_{t-1}\left(\xi_{[t-1]}^{k}\right) x_{t-1}^{ik}}:\ \pi_t^{i,k} \label{eq:Benders'-kantorovich}\\
&&\delta_{t+1}^s\leq {\beta_{t+1}^r}(\xi_{[t-1]}^k,\xi_{t}(s))^\top x_t + {\alpha_{t+1}^r}(\xi_{[t-1]}^k,\xi_{t}(s)): \ \rho^{ikrs}_t, \
s=1,\ldots,S_{t,k},\
r\in R(\xi_{[t-1]}^k,\xi_t),  \nonumber\\
&& \sum_{j=1}^{N} {y}_{j} \mu_{i, j} \leq \mathbb{P}(\xi_t=\xi_t(s)|\xi_{[t-1]})
 h_t(x_{t}, \xi_{t}(s)),\ s=1,\ldots,S, \nonumber \\
&&  \eqref{eq:dp-kan-con2}-\eqref{eq:dp-kan-con-n} \nonumber
\end{eqnarray}


\subsubsection*{\underline{Backward pass.}}
With the trail decision sequence
$\{x_{t}^k\}$, we can solve the following sub-problem for all scenarios with the piecewise linear approximation $V^{i}_{t+1}$ updated backwardly in this loop,
 \bgeqn
 \label{eq:Benders'-backward} 
            \left. \begin{array}{ll}
\max _{x_{t}} & \inf _{u_{t} \in \mathcal{U}_{t}(\xi_{[t-1]}^k) } \mathbb{E}_{\mid F_{t-1}}\left[u_{t}\left(h_{t}\left(x_{t}, \xi_{t}\right)\right)+V^{i}_{t+1}\left(x_{[t]},\xi_{[t-1]}^{k}, \xi_{t}\right)\right] \\
\text { s.t. } & W_{t}\left(\xi_{[t-1]}^{k}\right) x_{t}=b_{t}\left(\xi_{[t-1]}^{k}\right)-D_{t-1}\left(\xi_{[t-1]}^{k}\right) x_{t-1}^{i k} 
\end{array}\right.
\edeqn
Similarly to the forward pass, we collect $S_{t,k}$ samples
as the set of all realizations of $\xi_{t}$ conditional on the realization of $\xi_{[t-1]}^k$.
Thus, we can have a similar linear programming reformulation of \eqref{eq:Benders'-backward} similar to \eqref{eq:Benders'-kantorovich}. The only difference is that we set $R(\xi_{[t-1]}^k,\xi_t)=R^{i}(\xi_{[t-1]}^k,\xi_t) $.

\subsubsection*{\underline{Generate new cuts.}}
We first store the optimal dual variables $\pi_t^{i,k}, \rho_t^{i k r s}, r \in R^{i}(\xi_{[t-1]}^k,\xi_t(s))$, $s=1,\ldots,S_{t,k}$, in \eqref{eq:Benders'-kantorovich} at stage $t$.
Then we create an optimality cut, indexed by $\hat{r} $, for $V_t(\cdot)$,  with
 \bgeqn
 \label{eq:Benders'-beta}
 \beta^{\hat{r}}_t(\xi^k_{[t-1]})= -(\pi_t^{ik})^\top D_{t-1}(\xi_{[t-1]}^k),
 \edeqn
 \bgeqn
 \label{eq:Benders'-alpha}
 \alpha^{\hat{r}}_t(\xi^k_{[t-1]})=(\pi_t^{ik})^\top b_{t}(\xi_{[t-1]}^k)+ \sum_{s=1}^{S_{t,k}}\sum_{r\in R^{i}(\xi_{[t-1]}^k,\xi_t(s))} a_{t+1}^r(\xi_{[t-1]}^k,\xi_t(s)) \rho_t^{i k r s}
           \edeqn
         by taking the optimal dual variables into the objective function of the dual problem of \eqref{eq:Benders'-kantorovich} with $R(\xi_{[t-1]}^k,\xi_t)=R^{i}(\xi_{[t-1]}^k,\xi_t) $.  Then we obtain the piecewise linear approximation $V_t^i(\cdot)$ at stage $t-1$ by adding the new cut into the index set $R^i(\xi_{[t-1]}^k):=R^{i-1}(\xi_{[t-1]}^k)\cup \{\hat{r}\} $.

Since the dual feasible set in \eqref{eq:Benders'-kantorovich} is defined by finitely many linear constraints for each scenario,
there exist only finitely many dual extreme points, which can be attained for the considered scenario. Hence, only finitely many
different cut coefficients can be generated which guarantees the convergence of the Benders' style algorithm \cite{Sha11-EC,FuR21-EC}.
Using backward recursion, in a similar way, cuts can be derived for any stage $t = T - 1, \ldots, 1$, using the already updated cut approximation $R(\xi_{[t-1]}^k,\xi_t)=R^{i}(\xi_{[t-1]}^k,\xi_t) $.

{
\subsection{
Algorithmic procedures
of the NBD algorithm} 

We give the detailed algorithmic procedures when the classical NBD algorithm is adopted
to solve problem  \eqref{eq-mpro-tc}.

\vspace{2mm}

\begin{breakablealgorithm}\label{algo:Benders'} \caption{\small Solving MS-PRO-SD 
by NBD
}
\begin{flushleft}
\textbf{Inputs:} 
Confidence level $\alpha \in(0,1)$,
maximum number of iterations
$N_{\max}$, tolerance $tol$.
\\  \textbf{Initialize:}
(a) Initialize the lower bound
$\underline{v}^0 :=-\infty$
and the
upper bound
$\overline{v}^0=\infty$,
(b) set the counter of iterations to $i \leftarrow 0$,
(c) set the scenario tree structure of $\{\xi_t\}$ (including the set of scenarios $\mathcal{K}$, the appearing probability $p_k$ of scenario $k$),
(d) set $V^0_t(\cdot,\xi_{[t]}) :=0$, for $t=2,\ldots,T+1$.
(e) set $R^0(\xi_{[t]})=\{r^0\}$ {with $\beta_t^{r^0}=0$ and a large enough $\alpha_t^{r^0}$ for all non-leaf nodes.}
\\  \textbf{Output:} $x^{ik}_{t}$, $k\in\mathcal{K}$, $t=1,\cdots,T$.
\\
\textbf{while} $i<N_{\max}$ \textbf{do}
\end{flushleft}
\begin{itemize}   \item
Set $i \leftarrow i+1$.  
\item Forward Pass
\begin{itemize}
    \item[] Solve the approximate first-stage problem
    $$\overline{v}^i: =\max _{x_{1} \in \mathscr{X}_{1}} \inf _{u_{1} \in {U}_{1}} \mathbb{E}_{\mid F_{0}}\left[u_{1}\left(h_{1}\left(x_{1}, \xi_{1}\right)\right)+V^{i-1}_{2}\left(x_{[1]}, \xi_{[1]}\right)\right]
    $$
        to obtain trial point $x_1^{i k}=x_1^i$ for all $k \in \mathcal{K}$
        (We store the optimal value of the first part of the objective function $\inf _{u_{1} \in {U}_{1}} \mathbb{E}_{\mid F_{0}}\left[u_{1}\left(h_{1}\left(x_{1}, \xi_{1}\right)\right)\right]$ as $\underline{v}_1^i$).

  \item[] \textbf{for} stages $t=2, \ldots, T$ \textbf{do}
   \begin{itemize}
       \item[]  \textbf{for} samples $k \in \mathcal{K}$ \textbf{do}
         \begin{itemize}
             \item[] Solve the approximate stage-$t$ subproblem \eqref{eq:app-forward} (a conditional one-stage PRO problem)
for $x_{t-1}^{i k}$ to obtain trial point $x_{t}^{i k}$. We store the optimal value of the first part of the objective function $\inf _{u_{t} \in \mathcal{U}_{t}(\xi_{[t-1]}^k)  } \mathbb{E}_{\mid F_{t-1}}\left[u_{t}\left(h_{t}\left(x_{t}, \xi_{t}\right)\right)\right]$ as $\underline{v}_t^{ik}$.
        \item[] At the final stage $T$, store the optimal dual values $\pi_{T-1}^{i}$
        , $\rho_{T-1}^{i} $.
         \end{itemize}
       \item[] \textbf{end for}
   \end{itemize}
         \item[]  \textbf{end for}

         \item[]
         Substitute $
         x_{t}^{i k},\ k\in \mathcal{K}$, into the objective function to derive a lower bound
         $\underline{v}^{i} = v_1^i + \sum_{t=2}^{T-1}\sum_{k\in\mathcal{K}} p_k v_t^{ik}$,
         where $p_k$ is the appearing probability of the $k$-th scenario.

\end{itemize}

\item Check the stopping criterion
   \begin{itemize}
       \item[] \textbf{if} $|\underline{v}^{i} - \overline{v}^{i} | \leq {\rm tol} $
         \begin{itemize}
         \item[] terminate loop and return output.
            \end{itemize}
       \item[] \textbf{end if}
   \end{itemize}

\item Backward Pass
\begin{itemize}
  \item[] Set $V^{i}_{T+1}\left(x_{[T]}, \xi_{[T]}\right):=0$

    \item[] \textbf{for} stages $t = T-1, \ldots , 2$ \textbf{do}
    \begin{itemize}

        \item[] \textbf{for} samples $k\in \mathcal{K}$ \textbf{do}
        \begin{itemize}
        \item[]- Load $Son(\xi_{[t-1]}^k)$, the set of all realizations of $\xi_{t}$ conditional on the realization of $\xi_{[t-1]}^k$, i.e., the set of all son nodes of the $\xi_{[t-1]}^k$, set $S_{t,k}$ as the number of elements in $Son(\xi_{[t-1]}^k)$.
                \item[]
           - Solve the updated approximate stage-$t$ subproblem
           \eqref{eq:Benders'-kantorovich} with $R(\xi_{[t-1]}^k,\xi_t)=R^{i}(\xi_{[t-1]}^k,\xi_t) $
          for $x_{t-1}^{i k}$.
         %
         Notice that for stage $T$, we do not 
         solve the optimization problem.
         We directly use
         { the dual variables
         recorded in solving \eqref{eq:app-forward}}
         in the forward pass 
         \item[]
           - Store the optimal dual values
           $\pi_t^{i,k}, \rho_t^{i k r s}, r \in R^{i}(\xi_{[t-1]}^k,\xi_t(s))$, $s=1,\ldots,S_{t,k}$
            \item[]
             - Create an optimality cut, indexed by $\hat{r} $, for $V_t(\cdot)$,  with
           \eqref{eq:Benders'-beta} and \eqref{eq:Benders'-alpha}
            \item[] - Update the cut approximation $V
            _t^i(\cdot)$ at stage $t-1$.
             $R^i(\xi_{[t-1]}^k):=R^{i-1}(\xi_{[t-1]}^k)\cup \{\hat{r}\} $

        \end{itemize}

        \item[] \textbf{end for}
    \end{itemize}

    \item[] \textbf{end for}
\end{itemize}

\end{itemize}
\begin{flushleft}
\textbf{end while}
\end{flushleft}
\label{alg_lirnn}  \end{breakablealgorithm}

\subsection{SDDP algorithm}
The SDDP algorithm
has the same
procedures as the NBD algorithm (Algorithm 2),
with
three exceptions.
First, the SDDP algorithm randomly selects a finite number of
scenarios, denoted by $N$,
to construct $\mathcal{K}$ at each loop rather than
use all scenarios
as in the NBD algorithm.
Moreover, at each node in both
the forward pass and the backward pass, the SDDP algorithm randomly generates $S$ samples to
compute the conditional expected value approximately,
rather than
use the realizations of all son nodes in the NBD algorithm.
Finally,
since the lower bound provided by the SDDP algorithm is not
exact but
relies on random sampling,
we usually set a stopping criterion 
based on
a confidence interval,
say,  tol=$z_{1-0.99/2} \frac{\sigma_{\underline{v}}}{ |\mathcal{K}|}$, where $\sigma_{\underline{v}}$ is the sample standard deviation of $\sum_{t=2}^{T-1} v_t^{ik}$, $k\in\mathcal{K}$ and $z_{1-0.99/2}$
is the $({1-0.99/2})$-quantile of the standard normal distribution $N(0,1)$, see \cite{FuR21-EC}.
To avoid
repeat, we do not give the complete algorithmic
procedures of the SDDP method.





}

{

\section{Details of preference elicitation and construction of the ambiguity sets}
\label{EC:Elicit-Ambiguity-consct}


We use pairwise comparison approach and scoring approach to elicit the investor(DM)'s
preference and use the elicited preference information to
construct the ambiguity sets.
The former is based on the
{\em random relative utility split scheme} (RRUS) which is widely used in the literature of
PRO models, see e.g.~\cite{AmD15-EC}.
The key idea of RRUS is to generate
a pair of lotteries and ask the investor  to select, we refer readers to \cite{AmD15-EC,GuX21-EC} for the detailed procedures.
The latter is to ask the investor to give scores at different levels of consumption and use them to construct an approximate
utility function. Here we use
RRUS to construct the pairwise comparison-based ambiguity sets and the scoring approach to 
construct
a nominal utility function and subsequently
the Kantorovich ball-based ambiguity sets.


 \tikzstyle{startstop} = [rectangle, rounded corners, minimum width=1cm, minimum height=1cm, text centered, text width=3cm, draw=black, fill=white]
 \tikzstyle{arrow} = [thick,->,>=stealth]

\tikzstyle{ifelse} = [diamond,
minimum width=0.2cm, minimum height=0.2cm, text centered, text width=2cm, draw=black, fill=white]
 \tikzstyle{arrow} = [thick,->,>=stealth]

\begin{figure}[!ht]
    \centering
	\scalebox{0.85}{
\begin{tikzpicture}[node distance=3cm]
\node (step1) [startstop] {$N$ historical consumption trajectories ($T$ stages) and $N\times T$ 
consumption-scores data-pair };
\node (step2) [startstop, right of=step1, xshift=1.5cm] {For each stage, sort out
$N$ consumption levels increasingly as breakpoints $x_i$ and collect scores as feedback utility value $\hat{u}_i$};
\node (step3) [startstop, right of=step2,  xshift=2cm, yshift=-1.7cm] {Determine nominal utility value $u_i$
via optimistic estimation
\eqref{eq-utility-opt}};
\node (step4) [startstop, above of=step3] {Determine nominal utility value $u_i$
via
pessimistic estimation
\eqref{eq-utility-per}};
\node (step3a) [startstop, above of=step4] {Determine nominal utility value $u_i$ as best-fit estimation by \eqref{eq-utility-unbiased}};
\node (step5) [startstop, below  of=step3] {Determine nominal utility value $u_i$
via unbiased estimation by \eqref{eq-utility-unbiased-new-ref} };
\node (step6) [startstop, right  of=step3,  xshift=2cm,yshift=1.7cm] {
Develop a piecewise linear 
utility and determine nominal slope $\tilde{\beta}_j(t)$ for MS-PRO-SID-Kan };
\draw [arrow] (step1) -- (step2);
\draw [arrow] (step2) -- (step3);
\draw [arrow] (step2) -- (step4);
\draw [arrow] (step2) -- (step5);
\draw [arrow] (step3) -- (step6);
\draw [arrow] (step4) -- (step6);
\draw [arrow] (step5) -- (step6);
\draw [arrow] (step2) -- (step3a);
\draw [arrow] (step3a) -- (step6);
\end{tikzpicture}
 }
\caption{
Flowcharts of procedures for constructing a nominal 
utility function in
MS-PRO-SID-Kan model.}
    \label{fig:preference-elicitation}

\end{figure}
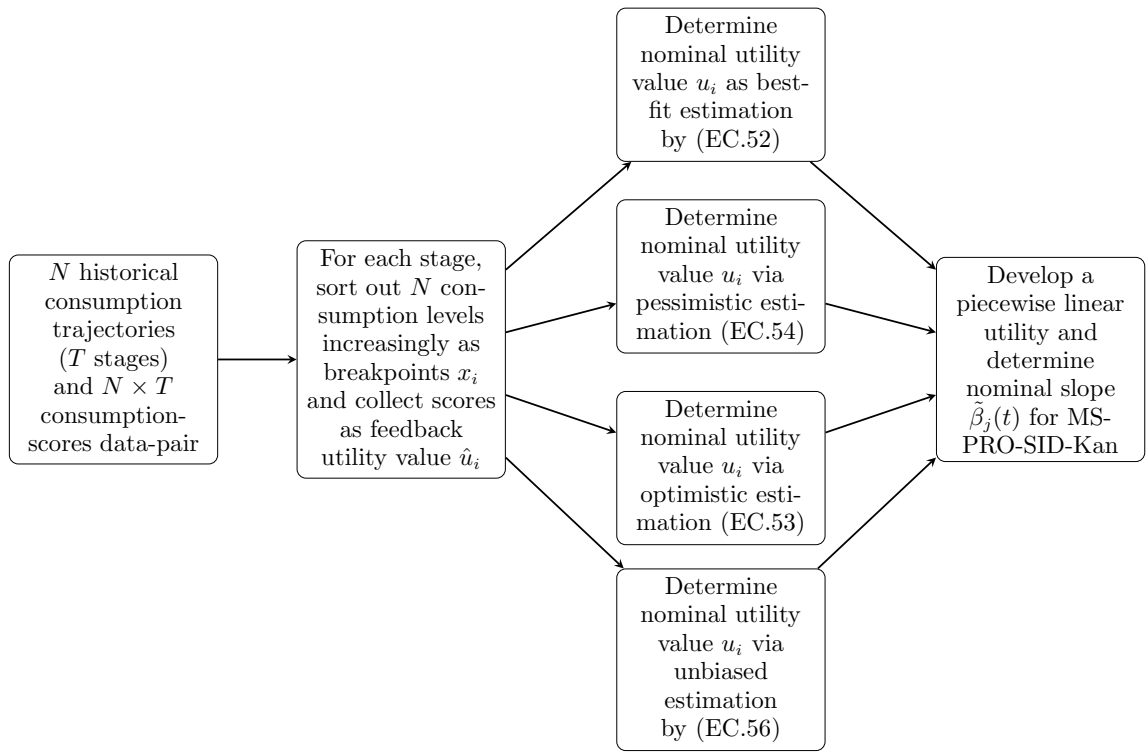


Specifically, we use a scoring and fitting method
to obtain a nominal utility function
based on the historical data related to
the investor.
Suppose that we have observed $N$ historical consumption trajectories (each of which consists of
$T$ stages) and  the investor's
 scores about
 the utility
 of the consumption at each stage for each trajectory. 
 This means that we have collected $N\times T$ 
pairs of consumption-score data.

\subsection{Estimating nominal utility
in state-independent case}

We consider $N$ consumption levels, denoted
by $x_i$, for $i=1,\cdots,N$ at each stage
of historical trajectory and let $\hat{u}_i$, $i=1,\ldots,N$ be the relevant utility scores
by the investor. By assorting them if necessary, we assume that the consumption
values are
in an increasing order.
A standard approach to
construct an approximate utility function is to
use a piecewise linear interpolation
passing through the observed data points (pairs of $(x_i,\hat{u}_i)$). However,
the
approximated piecewise linear
utility function constructed
as such is not necessarily
monotonically increasing
since  $\hat{u}_i$, $i=1,\ldots,N$  are not necessarily in an increasing order. This is primarily because  these scores might be obtained at different
states where the investor's risk preferences
are actually state-dependent. To tackle the issue, we propose four optimization-based models
to construct a piecewise linear nominal utility function.
The procedures are illustrated in the flowchart in Figure \ref{fig:preference-elicitation}.


\textbf{\underline{Best-fit estimation}.}
The first approach is to find a non-deceasing and concave function with least squares errors at the breakpoints from the scores by solving the following
minimization problem
\begin{subequations}\label{eq-utility-unbiased}
\begin{eqnarray}
   &
\min_u & \sum_{i=2}^{N-1} (u_i-\hat{u}_i)^2  \\
& {\rm s.t.} & (u_{i+1}-u_{i})/(x_{i+1}-x_{i}) \leq (u_{i}-u_{i-1})/(x_i-x_{i-1}) ,\ i=2,\ldots,N-1, \label{eq:utility-unbiased-1}\\
&&  0\leq (u_{N}-u_{N-1})/(x_N-x_{N-1}),  \label{eq:utility-unbiased-2}\\
&& u_1=0,\ u_N=1,\ 0\leq u_i \leq 1, \ i=2,\ldots,N-1.  \label{eq:utility-unbiased-3}
\end{eqnarray}
\end{subequations}
This approach is analogous
to the best-fit approach
in \cite{AmD15-EC}, the main difference
is that here $\hat{u}_i$, $i=1,\cdots,N$ are obtained from scoring.

\textbf{\underline{Optimistic estimation}.}
The second approach is to find the upper non-deceasing and concave envelope of the
graph of score points and construct a
(piecewise linear)
optimistic utility function
by solving the following minimization problem:
\begin{subequations}\label{eq-utility-opt}
\begin{eqnarray}
&  \min_u & \sum_{i=2}^{N-1}
(u_i-\hat{u}_i)^2\\
  & {\rm s.t.} & \eqref{eq:utility-unbiased-1}-\eqref{eq:utility-unbiased-3},\  \hat{u}_i \leq u_i, \ i=2,\ldots,N-1.
\end{eqnarray}
\end{subequations}
This estimation is optimistic because we consider the largest possible utility value of the utility function at each consumption level.
We denote the optimal value by $u^{U}=[u^U_1,\ldots,u^U_N]$, with slope $\beta_i^{U}$ between $u^U_{i-1}$ and $u^U_i$.

\textbf{\underline{Pessimistic estimation}.}
The third approach is to use the lower non-deceasing and concave envelope of
the graph of the
score points by solving
\begin{subequations}\label{eq-utility-per}
\begin{eqnarray}
&  \min\limits_u & \sum_{i=2}^{N-1}
(u_i-\hat{u}_i)^2\\
  & {\rm s.t.} & \eqref{eq:utility-unbiased-1}-\eqref{eq:utility-unbiased-3},\  u_i \leq \hat{u}_i, \ i=2,\ldots,N-1.
\end{eqnarray}
\end{subequations}
This estimation is pessimistic since we consider the smallest
possible utility value of the utility function at each consumption level.
We denote the optimal solution of $u$ as 
$u^{L}=[u^L_1,\ldots,u^L_N]$, with slope $\beta_i^{L}$ between $u^L_{i-1}$ and $u^L_i$.


To ensure the existence of pessimistic estimation
and
optimistic estimation, we assume that the scores are always larger than or equal to the consumption and smaller than or equal to $1$.

\textbf{\underline{Unbiased estimation}.}
{
Instead of considering
an optimistic or a pessimistic utility function based on the investor's historical utility scores, we may consider a utility
function which lies in the middle of the two under the Kantorovich distance. Specifically, we solve the following program:
\begin{subequations}\label{eq-utility-unbiased-new}
\begin{eqnarray}
   & \min\limits_{u,\beta}
& \max\{\dd_K(u,{u}^{U}),\dd_K(u,{u}^{L}) \}   \\
& \inmat{s.t.} &
\beta_i=  (u_{i}-u_{i-1})/(x_i-x_{i-1}) ,\ i=2,\ldots,N, \\
&& \beta_i \geq \beta_{i+1},\ i=1,\ldots,N=1,\ \beta_N\geq 0,\\
&& u_1=0,\ u_N=1,\ 0\leq u_i \leq 1, \ i=2,\ldots,N-1.
\end{eqnarray}
\end{subequations}
By 
utilizing the dual form of the Kantorovich distance in \eqref{eq:Kant-u-v-LP-dual}, we can
reformulate problem \eqref{eq-utility-unbiased-new} as the following linear programming problem:
\begin{subequations}\label{eq-utility-unbiased-new-ref}
\begin{align}
\min\limits_{u,\beta,\zeta,\lambda,\atop \mu,\rho,\phi}\
& \zeta   \\
 \inmat{s.t.}\ \  & \zeta \geq \frac{1}{2}\sum_{i=2}^N  (\lambda^k_i + \mu^k_i + \rho^k_i + \phi^k_i)(x_{i}- x_{i-1})^2,\ j=1,2,
\\
& {\beta}_i -\beta_i^{U} + \lambda^1_i-\mu^1_i + \rho^1_i - \phi^1_i =0,\  i=2,\cdots,N, \\
& {\beta}_i -\beta_i^{L} + \lambda^2_i-\mu^2_i + \rho^2_i - \phi^2_i =0,\  i=2,\cdots,N, \\
& (\mu^k_{2}-\lambda^k_{2})(x_{2}-x_1) =0,\ k=1,2, \\
& (\mu^k_{i+1}-\lambda^k_{i+1})(x_{i+1}-x_i) + (\rho^k_{i}-\phi^k_{i})(x_{i}-x_{i-1})=0, k=1,2, i=2,\cdots,N-1, \\
& (\rho^k_{N}-\phi^k_{N})(x_{N}-x_{N-1})   =0, \\
& \mu^k_i,\lambda^k_i,\rho^k_i,\phi^k_i \geq 0,\ k=1,2,\ i=2,\cdots,N.\\
&
\beta_i=  (u_{i}-u_{i-1})/(x_i-x_{i-1}) ,\ i=2,\ldots,N, \label{eq:utility-unbiased-new-1}\\
& \beta_i \geq \beta_{i+1},\ i=1,\ldots,N=1,\ \beta_N\geq 0,\\
& u_1=0,\ u_N=1,\ 0\leq u_i \leq 1, \ i=2,\ldots,N-1.
\end{align}
\end{subequations}
The optimal value of the program corresponds
to the radius of the Kantorovich ball.
The components $(\beta_i, u_i)$, $i=1,\ldots,N$, of the optimal solution of the above program
 are used to construct a piecewise linear nominal utility function.


\textbf{\underline{Illustrating examples}.}
We give two simple examples to show the piecewise linear nominal utility functions obtained through the four
approaches outlined above.
Figure~\ref{fig-score} depicts these functions.
20 data points are used in Figure~\ref{fig-score},
and 40 are used
in Figure~\ref{fig-score2}.
In both cases, consumption-utility score pairs are randomly generated with the true utility of the investor. We can see that the one based on unbiased estimation is much better than the one via the best-fit approach.
As for the setting of the  radius of the Kantorovich ball,
we can base on subjective judgement
or rely on \eqref{eq-utility-unbiased-new}.


\begin{figure}[ht]
	\centering
		\begin{minipage}[t]{0.44\textwidth}
		\centering
    \includegraphics[width=7.5cm]{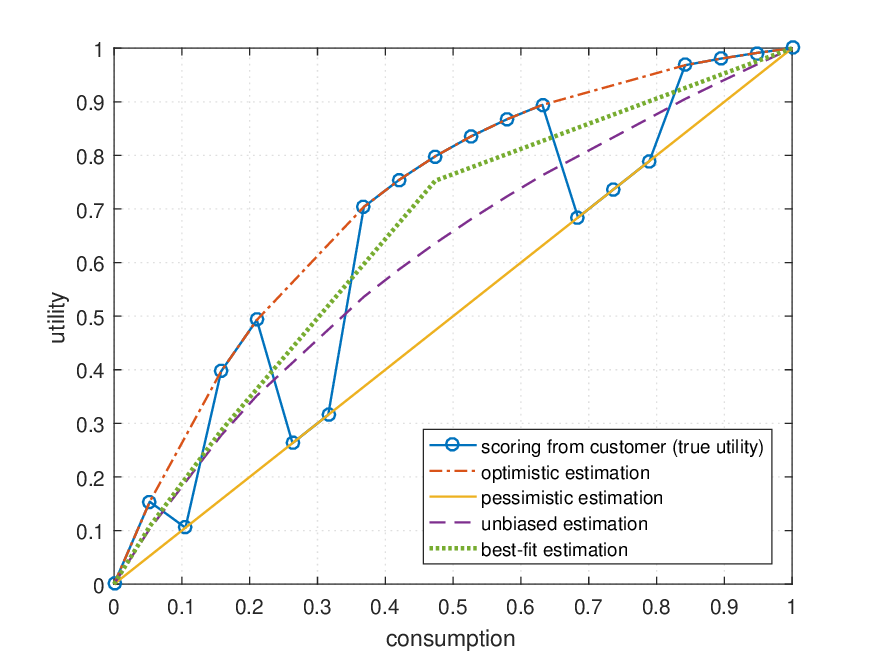}
\caption{
Using 20 utility scores (empty circles) at stage 1 (extracted
from 20 historical trajectories) to construct
four approximate
 state-independent nominal utility functions.
\label{fig-score}}
	\end{minipage}
\quad
 \begin{minipage}[t]{0.44\textwidth}
		\centering
	  \includegraphics[width=7.5cm]{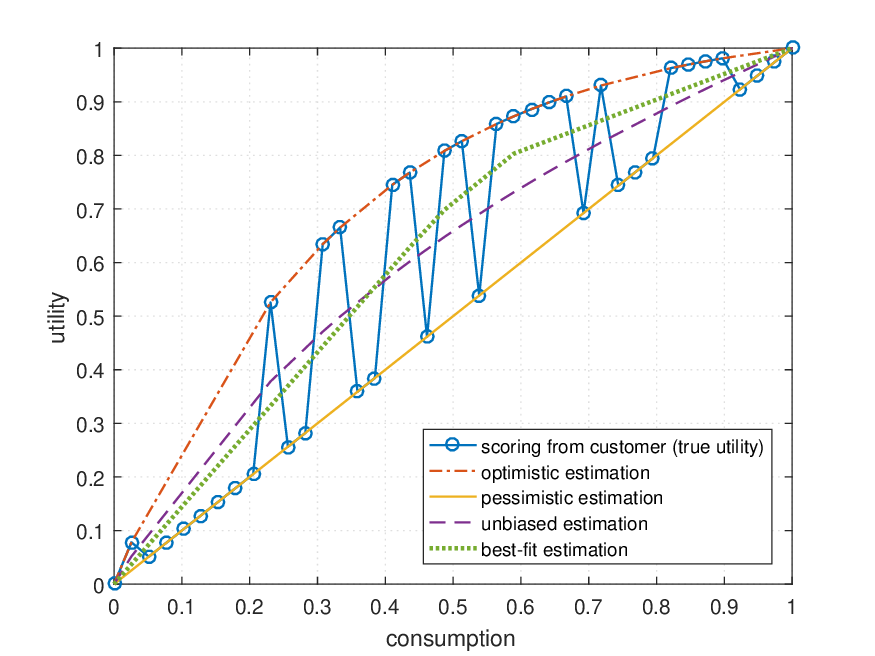}
\caption{
Using 40 utility scores at stage 1 (extracted
from 40 historical trajectories) to construct
four approximate
 state-independent nominal utility functions.
\label{fig-score2}}
	\end{minipage}
\end{figure}

}



\subsection{{Estimating state-dependent Kantorovich ball}}

We now turn to
discuss  construction of
the nominal utility function
of the Kantorovich ball
when the investor's
utility is state-dependent.
We consider
two cases depending on
whether the state information of the historical trajectories are observable or not by the modeller.

\textbf{\underline{States are  known}.}
If the modeller has complete information
about states,
then we can
separate the $N$ data points (consumption levels and scores) at each stage
in different groups, corresponding to different states (in our case study, the two states correspond to the high oil price state and the low oil price state).
 We
 apply
 one of the four approaches (best-fit for instance)
 in different 
 states respectively,
 and derive the state-dependent nominal utility functions accordingly.
 The procedures are
 explained in the flowchart
  in Figure \ref{fig:preference-elicitation-sd}.

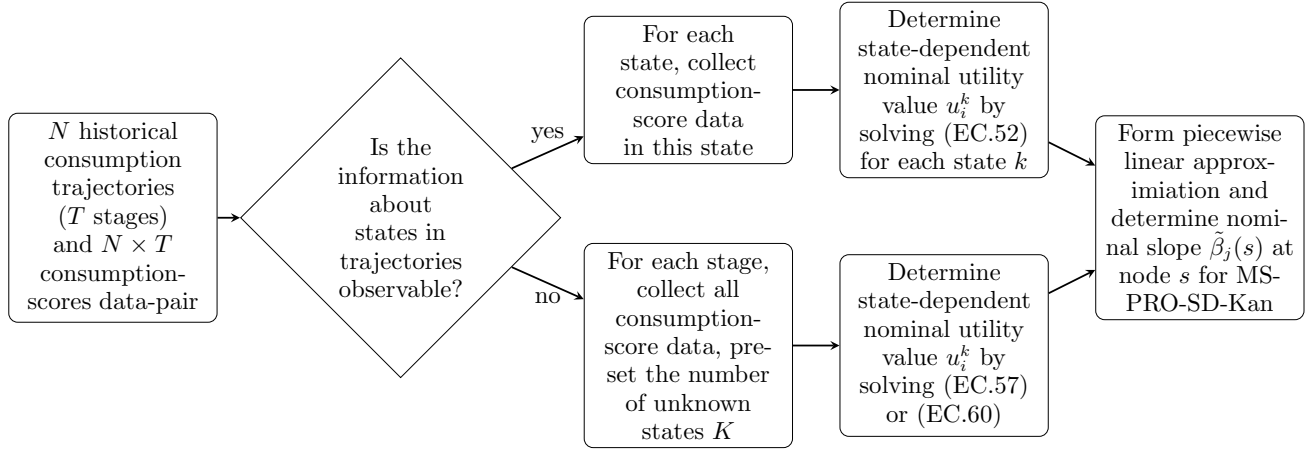
\begin{figure}[!ht]
    \centering
	\scalebox{0.85}{
\begin{tikzpicture}[node distance=3cm]
\node (step1) [startstop] {$N$ historical consumption trajectories ($T$ stages) and $N\times T$ 
consumption-scores data-pair };

\node (step1a) [ifelse, right  of=step1,  xshift=1.5cm] {Is the information about states in trajectories observable?};

\node (step2) [startstop, right  of=step1a,  xshift=1.5cm, yshift=2cm] {For each state, collect consumption-score data in this state};
\node (step2b) [startstop, right  of=step1a,  xshift=1.5cm, yshift=-2cm] {For each stage, collect all consumption-score data, pre-set the number of
unknown
states $K$};

\node (step3) [startstop, right of=step2,  xshift=1cm] {Determine state-dependent nominal utility value $u^k_i$ by solving \eqref{eq-utility-unbiased} for each state $k$};
\node (step3b) [startstop, right of=step2b,  xshift=1cm] {Determine state-dependent nominal utility value $u^k_i$ by solving \eqref{eq-utility-unbiased-sd} or \eqref{eq-utility-true_unbiased-sd-form}};
\node (step6) [startstop, right  of=step3,  xshift=1cm, yshift=-2cm] {Form piecewise linear approximiation and determine nominal slope $\tilde{\beta}_j(s)$ at node $s$ for MS-PRO-SD-Kan};
\draw [arrow] (step1) -- (step1a);
\draw [arrow] (step1a) -- node[above]{yes} (step2);
\draw [arrow] (step2) -- (step3);
\draw [arrow] (step1a) --
 node[below]{no} (step2b);
\draw [arrow] (step2b) -- (step3b);
\draw [arrow] (step3) -- (step6);
\draw [arrow] (step3b) -- (step6);
\end{tikzpicture}
 }
\caption{
Flowcharts of procedures for constructing a nominal 
utility function in
MS-PRO-SD-Kan model.}
    \label{fig:preference-elicitation-sd}

\end{figure}

\textbf{\underline{States are not known}.}
If we do not have complete information about states, then
we may divide the data points into several groups,
each of which will be used  to construct a piecewise linear utility function.
Alternatively, we can sort out
the state-dependent
data points with
all available data
by solving a single optimization problem and construct
the state-dependent utility functions accordingly with the optimal solutions.
We can do so by
the best-fit approach or
unbiased estimation approach
with some minor modifications.
In the former case, we solve program:
\begin{subequations}\label{eq-utility-unbiased-sd}
\begin{align}
 \min\limits_{u
   }\ \
   & \sum_{i=1}^{N} \min_{k=1,\ldots,K}\{ (u_i^k-\hat{u}_i)^2 \} \\
  {\rm s.t.}\ \    
&
 (u^k_{i+1}-u^k_{i})/(x_{i+1}-x_{i}) \leq (u^k_{i}-u^k_{i-1})/(x_i-x_{i-1}),i=2,\ldots,N-1,\ k=1,\ldots,K, \label{eq:utility-unbiased-sd-1}\\
&  0\leq (1-{u}^k_N)/(x_{1}-x_{N}) \leq (u^k_{N}-u^k_{N-1})/(x_N-x_{N-1}),\ k=1,\ldots,K,  \label{eq:utility-unbiased-sd-2}\\
& (u^k_{2}-u^k_{1})/(x_{2}-x_{1}) \leq u_{1}/x_1,\ k=1,\ldots,K,  \label{eq:utility-unbiased-sd-3}\\
& 0\leq u^k_i \leq 1, \ i=1,\ldots,N,\ k=1,\ldots,K. \label{eq:utility-unbiased-sd-4}
\end{align}
\end{subequations}
The
key idea
here is to 
use $u^k_{i}$ instead of $u_i$
where $k$ represents state $k$ for
 $k=1,\ldots,K$,
 at
 each breakpoint (consumption level).
Constraints \eqref{eq:utility-true-unbiased-1}-\eqref{eq:utility-true-unbiased-4}
are imposed to
ensure monotonicity and concavity of the
utility function at
state $k$, for
$k=1,\ldots,K$.
The
objective is
to minimize
the least squares errors/gaps
between
the state-dependent utility (to be decided)
and
the collected utility values.
The modified unbiased estimation approach
uses the same idea:
\begin{subequations}\label{eq-utility-true_unbiased-sd}
\begin{align}
   \min\limits_{u^{C},u^{U},u^{L},\beta^{C}, \atop \beta^{U},\beta^{L},S_k}
& \sum_{k=1}^K \max\{\dd_K(u^C_k,{u}_k^{U}),\dd_K(u_k^C,{u}_k^{L}) \}   \\
 \inmat{s.t.}\quad\ \ &
\beta_i^{k,j}=  (u^{k,j}_{i}-u^{k,j}_{i-1})/(x_i-x_{i-1}) ,\ i=2,\ldots,N,\ k=1,\ldots,K, \ j\in \{U,L,C\},\label{eq:utility-true-unbiased-1}\\
& \beta^{k,j}_i \geq \beta^{k,j}_{i+1},\ i=1,\ldots,N-1,\ \beta_N\geq 0,\  k=1,\ldots,K, \ j\in \{U,L,C\},\\
& u^{k,j}_1=0,\ u^{k,j}_N=1,\ 0\leq u^{k,j}_i \leq 1, \ i=2,\ldots,N-1,\  k=1,\ldots,K, \ j\in \{U,L,C\},
\label{eq:utility-true-unbiased-3}\\
& u^{k,L}_i \leq u_i\leq u^{k,U}_i,\ \inmat{if}\ i\in S_k,\ i=2,\ldots,N-1,\  k=1,\ldots,K, \label{eq:utility-true-unbiased-4}
\end{align}
\end{subequations}
Here, we generate $K$ Kantorovich balls simultaneously
with center 
$u_k^C$, upper boundary $u^U_k$ and
lower boundary $u^L_k$
for the $k$-th ball at each breakpoint.
The objective is to
minimize the sum of
the Kantorovich distance from the center of each ball to the upper
or
the lower boundary whichever is greater.
Constraints
\eqref{eq:utility-true-unbiased-1}-\eqref{eq:utility-true-unbiased-3} are imposed
to ensure monotonicity and concavity of the utility function
underlying the values
$u_k^C$, $u^U_k$ and $u^L_k$ in each state $k$.
\eqref{eq:utility-true-unbiased-4} is to guarantee that $u^U_k$ and $u^L_k$ cover all the consumption-utility pairs under state $k$.
Since in this case we do not know which the state
each of the
consumption-utility pairs belongs to,
the index set $S_k$ is a decision variable.
To get rid of the index set,
we may introduce  variables
$z_{i,k}$ which takes a value of $0$ or $1$.
Consequently we can reformulate \eqref{eq-utility-true_unbiased-sd} as
\begin{subequations}
\label{eq-utility-true_unbiased-sd2}
\begin{eqnarray}
   & \min\limits_{u^{C},u^{U},u^{L},\atop \beta^{C},\beta^{U},\beta^{L},z}
& \sum_{k=1}^K \max\{\dd_K(u^C_k,{u}_K^{U}),\dd_K(u^C,{u}^{L}) \}   \\
& \inmat{s.t.} & \eqref{eq:utility-true-unbiased-1}-\eqref{eq:utility-true-unbiased-3}, \\
&& u^{k,L}_i \leq u_i + (1-z_{i,k})M, \ i=2,\ldots,N-1,\  k=1,\ldots,K, \label{eq:utlity-new-unbiased2-1} \\
&&  u_i\leq u^{k,U}_i + (1-z_{i,k})M,\ i=2,\ldots,N-1,\  k=1,\ldots,K, \label{eq:utlity-new-unbiased2-2}\\
&& \sum_{k=1}^K z_{i,k}=1,\ z_{i,k}\in \{0,1\}^{(N-2)\times K},\ i=2,\ldots,N-1. \label{eq:utlity-new-unbiased2-3}
\end{eqnarray}
\end{subequations}
By
the dual formualtion
of the
Kantorovich distance in \eqref{eq:Kant-u-v-LP-dual}, we can
 recast
\eqref{eq-utility-true_unbiased-sd2}
as an
MILP:
\begin{subequations}\label{eq-utility-true_unbiased-sd-form}
\begin{align}
\min\limits_{u^{C},u^{U},u^{L},\beta^{C},\atop \beta^{U},\beta^{L},z,\zeta}
& \sum_{k=1}^K \zeta_k   \\
 \inmat{s.t.} & \; \eqref{eq:utility-true-unbiased-1}-\eqref{eq:utility-true-unbiased-3},\eqref{eq:utlity-new-unbiased2-1}-\eqref{eq:utlity-new-unbiased2-3}, \\
 & \zeta_k \geq \frac{1}{2}\sum_{i=2}^N  (\lambda^{k,j}_i + \mu^{k,j}_i + \rho^{k,j}_i + \phi^{k,j}_i)(x_{i}- x_{i-1})^2,\ k=1,\ldots,K,\ j\in \{U,L\},
\\
& {\beta}^{k,C}_i -\beta_i^{k,U} + \lambda^{k,U}_i-\mu^{k,U}_i + \rho^{k,U}_i - \phi^{k,U}_i =0,\  i=2,\cdots,N,\ k=1,\ldots,K, \\
& {\beta}^{k,C}_i -\beta^{k,L} + \lambda^{k,L}_i-\mu^{k,L}_i + \rho^{k,L}_i - \phi^{k,L}_i =0,\  i=2,\cdots,N,\ k=1,\ldots,K, \\
& (\mu^{k,j}_{2}-\lambda^{k,j}_{2})(x_{2}-x_1) =0,\ k=1,\ldots,K,\ j\in \{U,L\}, \\
& (\mu^{k,j}_{i+1}-\lambda^{k,j}_{i+1})(x_{i+1}-x_i) + (\rho^{k,j}_{i}-\phi^{k,j}_{i})(x_{i}-x_{i-1})   =0,\\
& \qquad i=2,\cdots,N-1,\ k=1,\ldots,K, \  j\in \{ U,L,C \}, \\
& (\rho^{k,j}_{N}-\phi^{k,j}_{N})(x_{N}-x_{N-1})   =0,\ k=1,\ldots,K, \  j\in \{ U,L,C \}, \\
 & \mu^{k,j}_i, \lambda^{k,j}_i, \rho^{k,j}_i, \phi^{k,j}_i \geq 0,\ i=2,\cdots,N,\ k=1,\ldots,K, \  j\in \{ U, L, C \}.
\end{align}
\end{subequations}
By solving \eqref{eq-utility-true_unbiased-sd-form}, we can 
figure out simultaneously
the optimistic estimation $u^U_k$, the pessimistic estimation $u^L_k$ and the unbiased estimation $u^C_k$
at each state $k$.
Subsequently, we can
use the unbiased estimation
as the center of the Kantorovich ball in each state and
the optimal value of $\zeta_k$ as the radius of the ball.


We
use
two examples to illustrate the estimations with different a-priori number of states ($K$).
In the first group of tests,
we generate 40 score points
with
a true utility function
which depends on
two states, see the empty circles in Figures \ref{fig-score-sd-2-1}-\ref{fig-score-sd-2-2}. The true utility function is defined in Section 6.2.
We
work out the unbiased estimations with $K=1$
and $K=2$ by solving \eqref{eq-utility-true_unbiased-sd-form}. The
resulting unbiased utility functions
are
displayed in Figure \ref{fig-score-sd-2-1} and Figure \ref{fig-score-sd-2-2}.

In the second set of tests,
we also generate 40 score points by a true utility which are dependent of three states,
see the empty circles in Figures \ref{fig-score-sd-3-1}-\ref{fig-score-sd-3-3}.
We
figure out the unbiased estimations 
for $K=1$, $K=2$
and $K=3$
by solving \eqref{eq-utility-true_unbiased-sd-form} respectively,
the approximate piecewise linear utility functions
are displayed in Figure \ref{fig-score-sd-3-1}, Figure \ref{fig-score-sd-3-2} and Figure \ref{fig-score-sd-3-3}.








\begin{figure}[htbp]
	\centering
		\begin{minipage}[t]{0.44\textwidth}
		\centering
    \includegraphics[width=7.5cm]{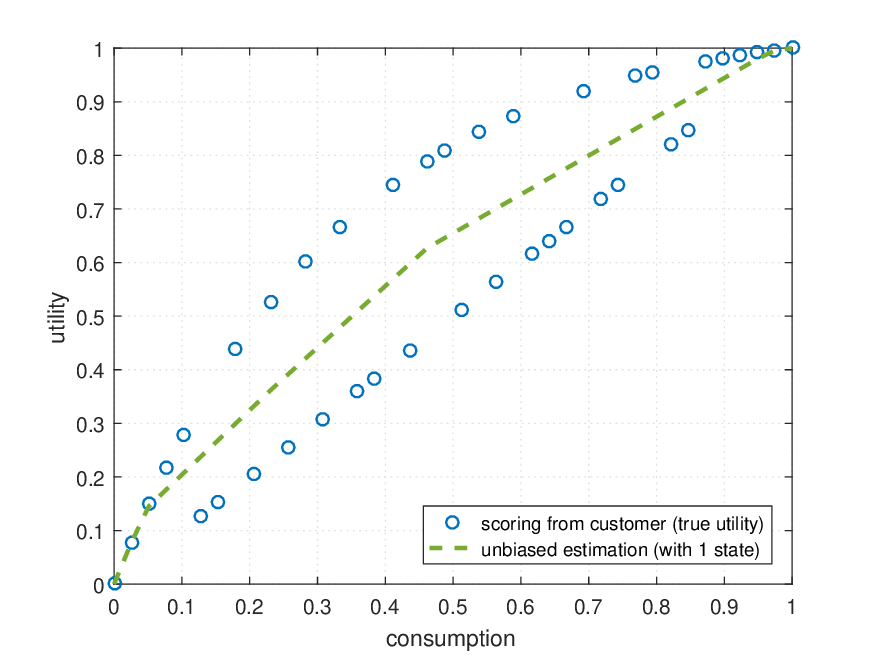}
\caption{
Empty circles represent 40 sample
scores generated by the true utility which are dependent of two states.
The green dashed curve is the unbiased utility function obtained from solving \eqref{eq-utility-true_unbiased-sd-form} with $K=1$.
\label{fig-score-sd-2-1}}
	\end{minipage}
\quad
 \begin{minipage}[t]{0.44\textwidth}
		\centering
	  \includegraphics[width=7.5cm]{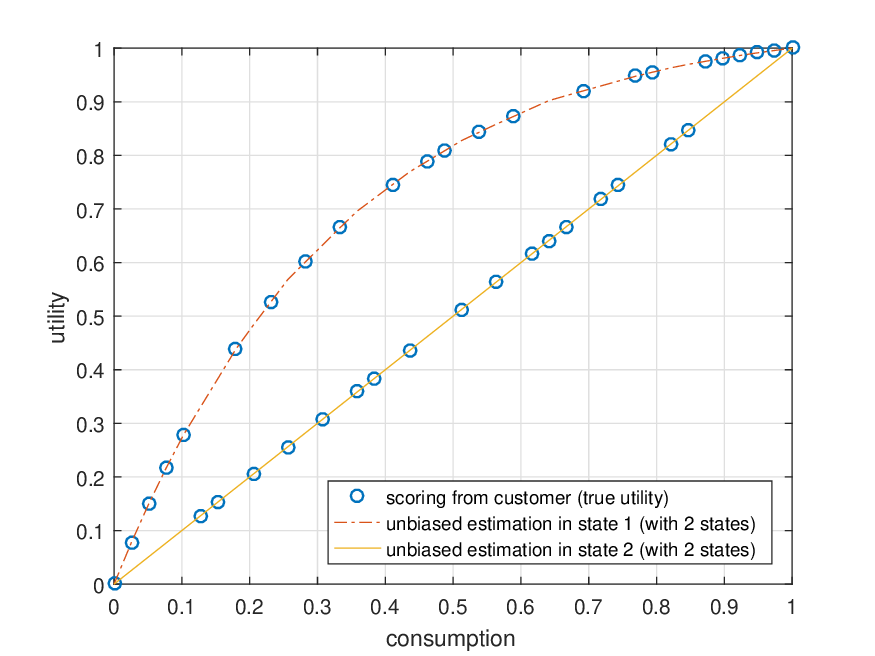}
\caption{
Empty circles represent 40 sample
scores generated by the true utility which are dependent of two states.
The pink and yellow dashed curve are two unbiased utility functions obtained from solving \eqref{eq-utility-true_unbiased-sd-form} with $K=2$.
\label{fig-score-sd-2-2}}
	\end{minipage}
\end{figure}

\begin{figure}[htbp]
	\centering
		\begin{minipage}[t]{0.44\textwidth}
		\centering
    \includegraphics[width=7.5cm]{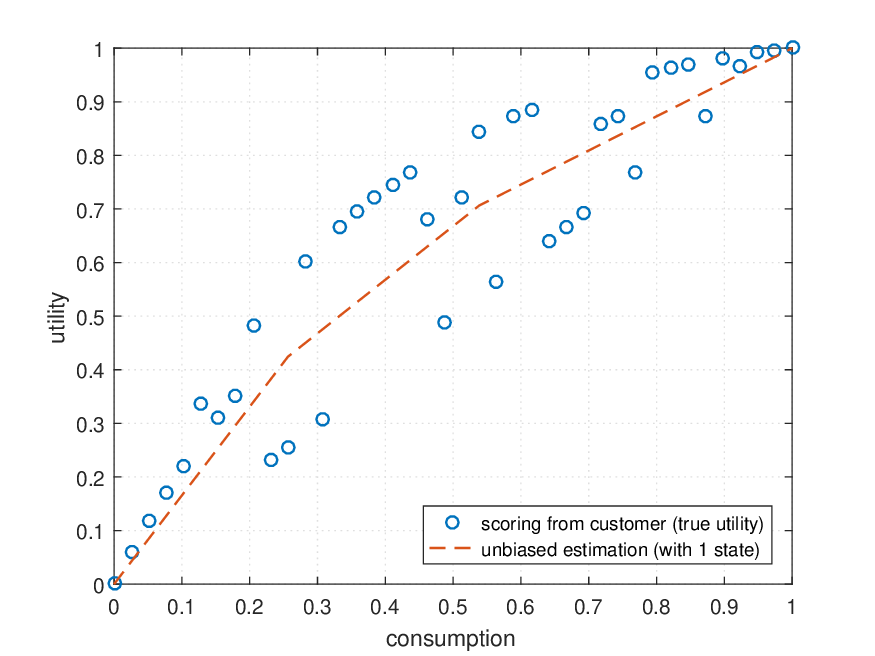}
\caption{
Empty circles represent 40 sample
scores generated by the true utility which are dependent of three states.
The red dashed curve is an unbiased utility function obtained from solving \eqref{eq-utility-true_unbiased-sd-form} with $K=1$.
\label{fig-score-sd-3-1}}
	\end{minipage}
\quad
 \begin{minipage}[t]{0.44\textwidth}
		\centering
	  \includegraphics[width=7.5cm]{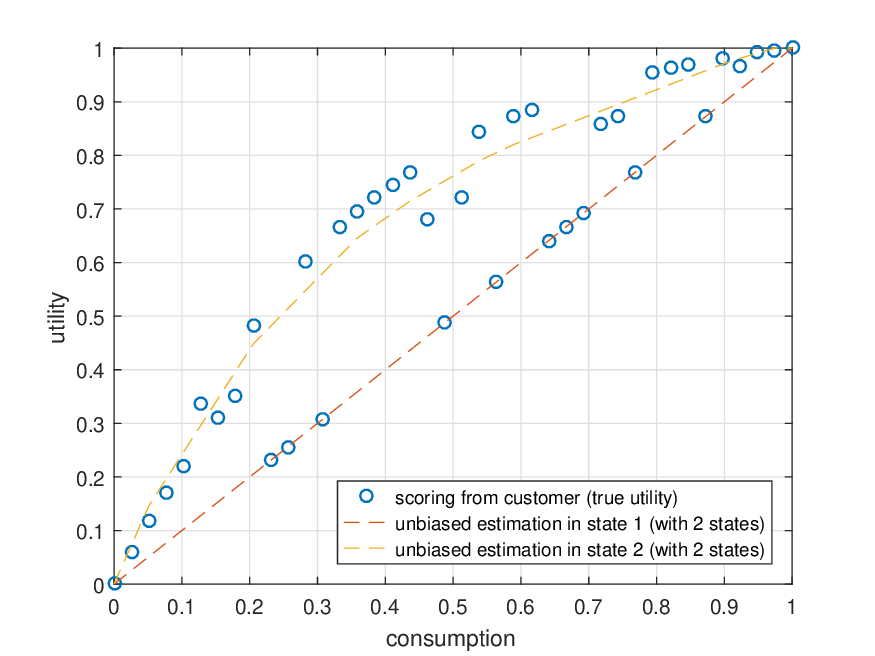}
\caption{
Empty circles represent 40 sample
scores generated by the true utility which are dependent of three states.
The yellow dashed curves and the red dashed line are two unbiased utility functions
obtained from solving \eqref{eq-utility-true_unbiased-sd-form} with $K=2$.
\label{fig-score-sd-3-2}}
	\end{minipage}
\end{figure}

\begin{figure}[htbp]
	\centering
		\begin{minipage}[t]{0.44\textwidth}
		\centering
    \includegraphics[width=7.5cm]{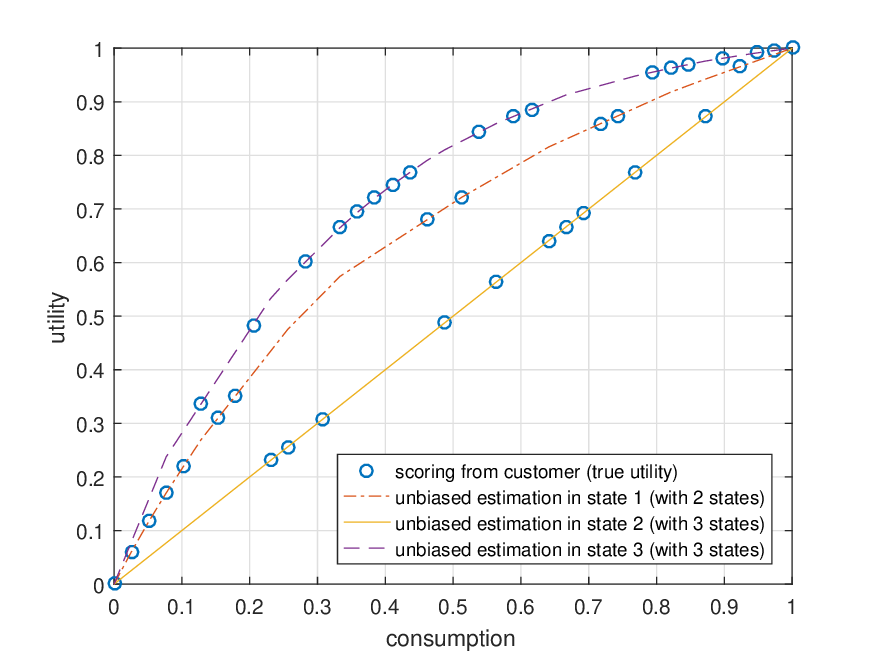}
\caption{
Empty circles represent 40 sample
scores generated by the true utility which are dependent of three states.
The pink and red dashed curves
and the yellow dashed line are three unbiased utility functions obtained from solving \eqref{eq-utility-true_unbiased-sd-form} with $K=3$.
\label{fig-score-sd-3-3}}
	\end{minipage}
\end{figure}


\section{Figures from numerical tests}



\begin{figure}[htbp]
	\centering
	\begin{minipage}[t]{0.44\textwidth}
  \centering
    \includegraphics[width=8cm]{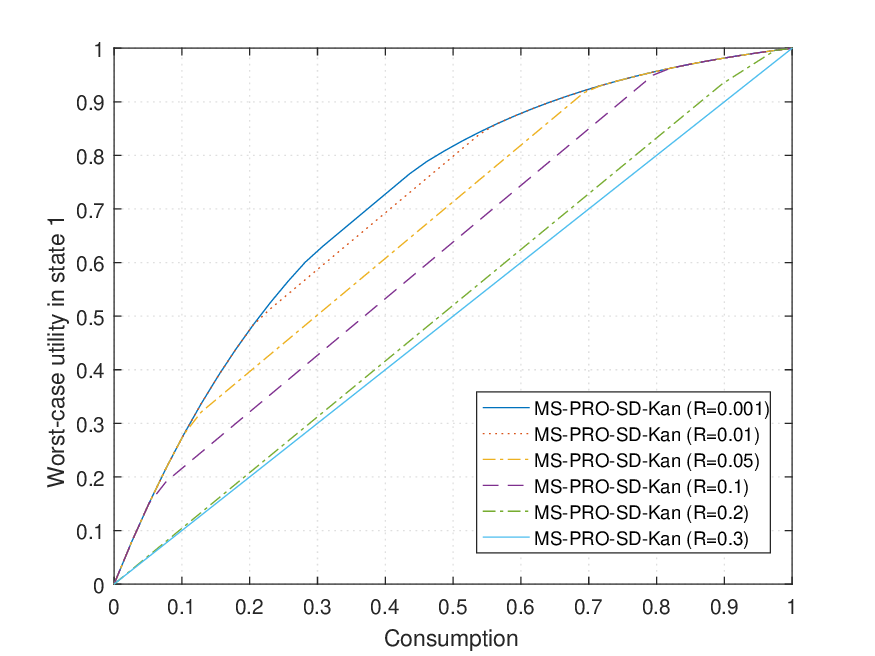}
    \caption{Worst-case utility functions in state 1 (high oil price) of MS-PRO-SD-Kan with different radii. }
    \label{ECfig:worst-case-utility}
	\end{minipage}
	\quad	
 \begin{minipage}[t]{0.44\textwidth}
		\centering
		 \includegraphics[width=7.5cm]{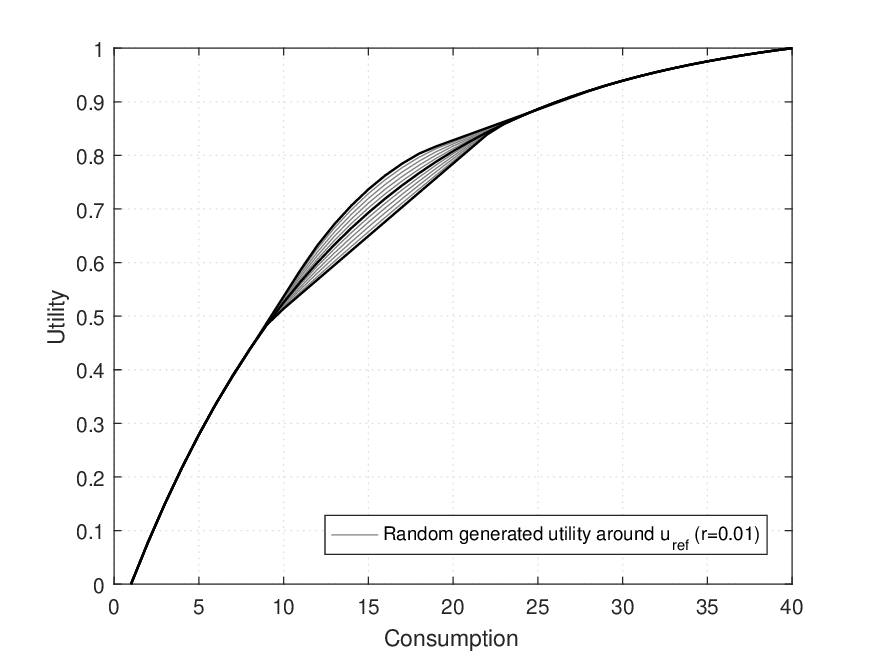}
\caption{
Utility functions randomly generated within 0.01 Kantorovich ball centered at the reference utility function in state 1 (high oil price)  ($N=40$). }\label{fig-random-u-001}
	\end{minipage}
\end{figure}

\begin{figure}[htbp]
	\centering
	\begin{minipage}[t]{0.44\textwidth}
		\centering
	 \includegraphics[width=7.5cm]{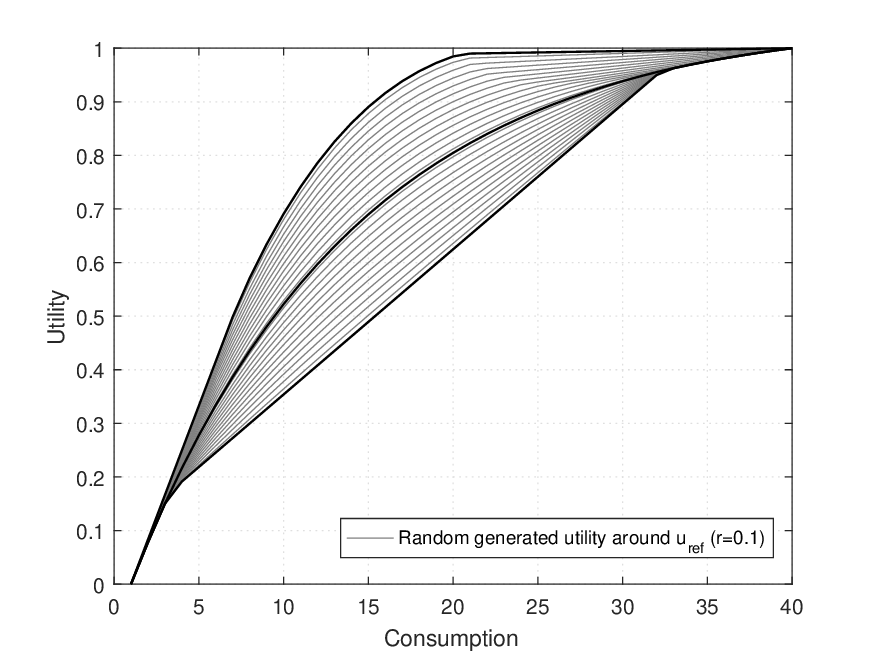}
\caption{
Utility functions randomly generated within 0.1 Kantorovich ball centered at the reference utility function in state 1 (high oil price) ($N=40$).}\label{fig-random-u-01}
	\end{minipage}
	\quad	\begin{minipage}[t]{0.44\textwidth}
		\centering
	 \includegraphics[width=7.5cm]{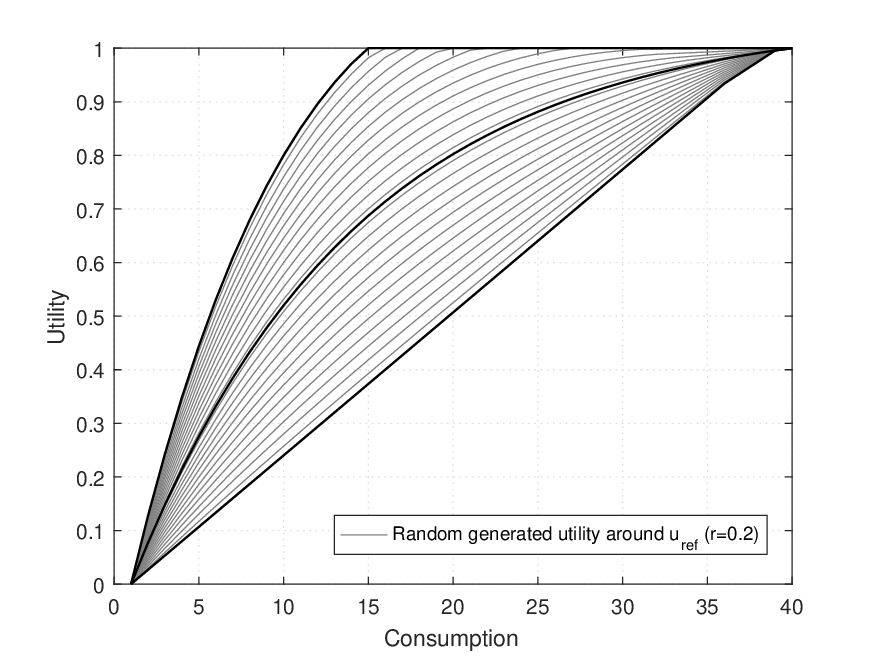}
\caption{
Utility functions randomly generated within 0.2 Kantorovich ball centered at the reference utility function in state 1 (high oil price) ($N=40$).}\label{fig-random-u-02}
	\end{minipage}
\end{figure}

}






 \newpage
\newpage

\end{document}